\documentclass[a4paper,11pt]{article}

\usepackage[fleqn]{amsmath}
\usepackage{amsthm, amssymb, amsfonts, mathtools, stmaryrd, mathrsfs, mathdots}
\usepackage[T1]{fontenc}
\usepackage{graphicx}
\usepackage{color}
\usepackage{enumitem}
\usepackage{cancel}
\usepackage[english]{babel}
\usepackage{etoolbox}
\usepackage{xcolor}
\usepackage{thmtools}
\usepackage{yfonts}
\usepackage{mfirstuc}
\usepackage{nameref}
\usepackage{bbm}
\usepackage{hyperref}
\usepackage{cleveref}
\usepackage{float}

\newtoggle{darkDoc}
\definecolor{backgroundcolor}{rgb}{\iftoggle{darkDoc}{0.1,0.1,0}{1,1,1}}
\definecolor{textcolor}{rgb}{\iftoggle{darkDoc}{1,0.9,0.8}{0,0,0}}
\definecolor{highlightcolor}{rgb}{\iftoggle{darkDoc}{0.5,0.7,1}{0,0,0.8}}
\definecolor{commentcolor}{rgb}{\iftoggle{darkDoc}{1,0.5,0.4}{0.8,0,0}}
\definecolor{commenttcolor}{rgb}{\iftoggle{darkDoc}{0.5,1,0.4}{0,0.8,0}}

\pagecolor{backgroundcolor}
\color{textcolor}

\usepackage{tikz}
\usepackage{tikz-cd}
\usetikzlibrary{arrows}
\tikzset{%
  dot/.style={circle, draw, fill=black, inner sep=0pt, minimum width=4pt},
  ddot/.style={circle, draw, thick, fill=black, double, double distance=1pt, inner sep=0pt, minimum width=5pt},
  sep/.style={inner sep=2.5pt},
  thicker/.style={line width=2pt},
  dub/.style={double,double distance=2pt},
  north/.style={label={[sep]90:#1}},
  northeast/.style={label={[sep]45:#1}},
  east/.style={label={[sep]0:#1}},
  southeast/.style={label={[sep]315:#1}},
  south/.style={label={[sep]270:#1}},
  southwest/.style={label={[sep]225:#1}},
  west/.style={label={[sep]180:#1}},
  northwest/.style={label={[sep]135:#1}},
}
\tikzcdset{%
   scale cd/.style={every label/.append style={scale=#1},
   cells={nodes={scale=#1}}}
}

\usepackage{fullpage}
\setlist{nosep}
\usepackage{titlesec}
\titleformat{\section}{\centering\bfseries\scshape}{\thesection.\quad}{0em}{}
\titleformat{\subsection}{\scshape}{\thesubsection.\quad}{0em}{}
\titlespacing*{\section}{0mm}{6mm}{3mm}
\titlespacing*{\subsection}{0mm}{6mm}{3mm}

\declaretheoremstyle[
    headfont=\bfseries, 
    notebraces={--- \makefirstuc}{},
    notefont=\normalfont\itshape,
    bodyfont=\normalfont,
    headpunct={\vspace{0.15\topsep}},
    spaceabove=\topsep,
    spacebelow=0pt,
    postheadspace=\newline,
    qed={$\triangleleft$},
]{mystyle}
\theoremstyle{mystyle}

\declaretheorem[style=mystyle,name=Theorem,numberwithin=section]{thm}
\crefname{thm}{theorem}{theorems}
\Crefname{thm}{Theorem}{Theorems}
\declaretheorem[style=mystyle,name=Definition,sibling=thm]{dfn}
\crefname{dfn}{definition}{definitions}
\Crefname{dfn}{Definition}{Definitions}
\declaretheorem[style=mystyle,name=Lemma,sibling=thm]{lmm}
\crefname{lmm}{lemma}{lemmas}
\Crefname{lmm}{Lemma}{Lemmas}
\declaretheorem[style=mystyle,name=Corollary,sibling=thm]{crl}
\crefname{crl}{corollary}{corollaries}
\Crefname{crl}{Corollary}{Corrolaries}

\crefname{prp}{proposition}{propositions}
\Crefname{prp}{Proposition}{Propositions}

\crefname{exm}{example}{examples}
\Crefname{exm}{Example}{Examples}
\declaretheorem[style=mystyle,name=Remark,sibling=thm]{rmk}
\crefname{rmk}{remark}{remarks}
\Crefname{rmk}{Remark}{Remarks}

\crefname{rmd}{reminder}{reminders}
\Crefname{rmd}{Reminder}{Reminders}

\crefname{pdx}{paradox}{paradoxes}
\Crefname{pdx}{Paradox}{Paradoxes}

\crefname{clm}{claim}{claims}
\Crefname{clm}{Claim}{Claims}
\declaretheorem[style=mystyle,name=Fact,sibling=thm]{fct}
\crefname{fct}{fact}{facts}
\Crefname{fct}{Fact}{Facts}

\crefname{cnj}{conjecture}{conjectures}
\Crefname{cnj}{Conjecture}{Conjectures}
\declaretheorem[style=mystyle,name=Question,sibling=thm]{qst}
\crefname{qst}{question}{questions}
\Crefname{qst}{Question}{Questions}

\crefname{exc}{exercise}{exercises}
\Crefname{exc}{Exercise}{Exercises}
\declaretheorem[style=mystyle,name=Theorem,numbered=no]{thm*}
\declaretheorem[style=mystyle,name=Definition,numbered=no]{dfn*}
\declaretheorem[style=mystyle,name=Lemma,numbered=no]{lmm*}
\declaretheorem[style=mystyle,name=Corollary,numbered=no]{crl*}
\declaretheorem[style=mystyle,name=Proposition,numbered=no]{prp*}
\declaretheorem[style=mystyle,name=Example,numbered=no]{exm*}
\declaretheorem[style=mystyle,name=Remark,numbered=no]{rmk*}
\declaretheorem[style=mystyle,name=Reminder,numbered=no]{rmd*}
\declaretheorem[style=mystyle,name=Paradox,numbered=no]{pdx*}
\declaretheorem[style=mystyle,name=Question,numbered=no]{qst*}
\declaretheorem[style=mystyle,name=Claim,numbered=no]{clm*}
\declaretheorem[style=definition,name={Exercise},sibling=exc,preheadhook={\hspace{-0.4cm}\raisebox{0.25mm}{$*$}\vspace{-\baselineskip}\vspace{-\topsep}}]{exc*}

\makeatletter
\renewenvironment{proof}[1][\proofname]{
  \pushQED{\qed}%
  \normalfont
  \topsep0pt \partopsep0pt 
  \trivlist
  \item[\hskip\labelsep
        \itshape
    #1\@addpunct{.}]\ignorespaces
}{%
  \popQED\endtrivlist\@endpefalse
  \addvspace{0pt} 
}
\makeatother

\let\phi\varphi
\let\epsilon\varepsilon
\let\emptyset\varnothing

\let\subset\subseteq
\let\supset\supseteq

\let\bar\overline

\renewcommand{\c}{\mathcal}
\newcommand{\bb}{\mathbb} 
 
\renewcommand{\b}{\mathbf}
\newcommand{\s}{\mathsf}
\newcommand{\sr}{\mathscr}

\renewcommand{\r}{\mathrm}
\newcommand{\f}{\mathfrak}

\newcommand{\inj}{\mathrel{\textnormal{\guilsinglright}\mathrel{\mkern -9mu}\rightarrow}}		
\newcommand{\srj}{\rightarrow\mathrel{\mkern -14mu}\rightarrow}						
\newcommand{\bij}{\mathrel{\textnormal{\guilsinglright}\mathrel{\mkern -9mu}\rightarrow\mathrel{\mkern -14mu}\rightarrow}}

\newcommand{\xto}[1]{\xrightarrow}												

\newcommand{\dwa}{\downarrow}										

\newcommand{\la}{\land}				
\newcommand{\lo}{\lor}				
\newcommand{\lc}{\bot}				
\newcommand{\emp}{\emptyset}		
\renewcommand{\Cap}{\bigcap}		
\renewcommand{\Cup}{\bigcup}		
\newcommand{\La}{\bigwedge\!} 		

\makeatletter
\providecommand*{\Dashv}{\mathrel{\mathpalette\@Dashv\vDash}}
\newcommand*{\@Dashv}[2]{\reflectbox{$\m@th#1#2$}}
\providecommand*{\dashV}{\mathrel{\mathpalette\@dashV\Vdash}}
\newcommand*{\@dashV}[2]{\reflectbox{$\m@th#1#2$}}
\makeatother
\newcommand{\md}{\vDash} 				
\newcommand{\fc}{\Vdash} 				
\newcommand{\nfc}{\mathrel{\cancel\Vdash}} 	

\newcommand{\rb}[1]{\left(#1\right)} 					
\newcommand{\ab}[1]{\left\langle#1\right\rangle} 		
\newcommand{\dsb}[1]{\left\llbracket#1\right\rrbracket} 	
\newcommand{\ap}[1]{\text{``}\,#1\,\text{''}} 		
\newcommand{\card}[1]{\left|#1\right|} 				
\newcommand{\norm}[1]{\left\lVert#1\right\rVert}		
\newcommand{\st}[1]{\left\{#1\right\}} 				
\newcommand{\smid}{\ \middle |\ } 					

\newcommand{\sab}[1]{\langle#1\rangle} 		
\newcommand{\sst}[1]{\{#1\}} 				

\newcommand{\rl}{\mathrel}					
\newcommand{\nrl}[1]{\mathrel{\cancel{#1}}}		

\newcommand{\dom}{\r{dom}}
\newcommand{\ran}{\r{ran}}
\newcommand{\ot}{\r{ot}}
\newcommand{\cf}{\r{cf}}
\newcommand{\id}{\s{id}}
\newcommand{\cov}{\r{cov}}
\newcommand{\non}{\r{non}}
\newcommand{\cof}{\r{cof}}
\newcommand{\add}{\r{add}}
\newcommand{\pow}{\s{pow}}
\newcommand{\Loc}{\r{Loc}}

\newcommand{\Lev}{\r{Lev}}

\newcommand{\Split}{\r{Split}}
\newcommand{\poss}[2]{\r{poss}(#1,{<}#2)}
\newcommand{\possq}[2]{\r{poss}(#1,{\leq}#2)}

\newcommand{\suc}{\r{suc}}
\newcommand{\supp}{\r{supp}}

\newcommand{\gcs}{{}^{\kappa}2}
\newcommand{\gfcs}{{}^{<\kappa}2}
\newcommand{\bs}{{}^{\omega}\omega}

\newcommand{\gbs}{{}^{\kappa}\kappa}
\newcommand{\gfbs}{{}^{<\kappa}\kappa}
\newcommand{\frcS}[1]{{\bb S^{#1}_\kappa}}

\newcommand{\frcQ}[1]{{\bb{Q}^{#1}_\kappa}}

\newcommand{\bstar}[1]{\f b_\kappa^{#1}(\in^*)}
\newcommand{\dstar}[1]{\f d_\kappa^{#1}(\in^*)}
\newcommand{\binf}[1]{\f b_\kappa^{#1}(\nnii)}
\newcommand{\dinf}[1]{\f d_\kappa^{#1}(\nnii)}
\newcommand{\bneq}[1]{\f b_\kappa^{#1}(\neqi)}
\newcommand{\dneq}[1]{\f d_\kappa^{#1}(\neqi)}
\newcommand{\bleq}[1]{\f b_\kappa^{#1}(\leq^*)}
\newcommand{\dleq}[1]{\f d_\kappa^{#1}(\leq^*)}
\newcommand{\AL}{\sr A\!\!\sr L}
\newcommand{\ED}{\sr E\!\sr D}
\newcommand{\AD}{\bb{AD}}
\newcommand{\SN}{\c S\c N}
\newcommand{\sSN}{\sr S\!\!\sr N}

\newcommand{\Locf}{\bb{L}\mathbbm{oc}}
\newcommand{\ALocforc}{\bb{AL}\mathbbm{oc}_\kappa^{b,h}(C)}
\newcommand{\nwdb}{\c N_\kappa^b}
\newcommand{\mgrb}{\c M_\kappa^b}

\newcommand{\ins}{\in^*}
\newcommand{\ini}{\in^\infty}
\newcommand{\nins}{\mathrel{\cancel{\ins}}}
\newcommand{\nini}{\mathrel{\cancel{\in^\infty}}}

\newcommand{\nii}{\ni^\infty}

\newcommand{\nnii}{\mathrel{\cancel{\nii}}}

\newcommand{\eqi}{=^\infty}

\newcommand{\neqi}{\mathrel{\cancel{=^\infty}}}

\newcommand{\ts}{\textstyle}

\newcommand{\ft}{\mathbbm 1}

\renewcommand{\emph}{\textbf}

\title{\sc Cardinal Characteristics on\\Bounded Generalised Baire Spaces}
\author{Tristan van der Vlugt\footnote{Fachbereich Mathematik, Universität Hamburg $-$ email: tristan@tvdvlugt.nl $-$ The author was partially supported by the Fields Institute, Toronto $-$ The author wishes to thank Jörg Brendle for valuable discussions, suggestions and extensive comments. Valuable comments have been given by Tatsuya Goto, Jing Zhang and Hiroshi Sakai.}}
\begin{document}

\maketitle

\begin{abstract}
We will give an overview of four families of cardinal characteristics defined on subspaces $\prod_{\alpha\in\kappa}b(\alpha)$ of the generalised Baire space $\gbs$, where $\kappa$ is strongly inaccessible and $b\in\gbs$. The considered families are bounded versions of the dominating, eventual difference, localisation and antilocalisation numbers, and their dual cardinals. We investigate parameters for which these cardinals are nontrivial and how the cardinals relate to each other and to other cardinals of the generalised Cicho\'n diagram. Finally we prove that different choices of parameters may lead to consistently distinct cardinals.
\end{abstract}

\section{Introduction}

Cardinal characteristics of the continuum are cardinal numbers that consistently lie strictly between $\aleph_1$ and the cardinality of the reals $2^{\aleph_0}$. Many natural questions about the cardinality of sets of real numbers lead to definitions of cardinal characteristics, such as the least size of a set of reals of positive Lebesgue measure, or the least size of an unbounded set of functions from $\omega$ to $\omega$. In studying such cardinal characteristics, one would usually work with perfect Polish spaces, such as the Baire space $\bs$, instead of directly working with the set of reals.

In recent years, there have been significant developments in generalising the theory of cardinal characteristics of the continuum to the context of the generalised Baire space $\gbs$, where $\kappa$ is an uncountable cardinal. In many cases, the situation on $\gbs$ closely resembles the situation on $\bs$, but this is not always the case. For instance, the meagre ideal $\c M$, as defined on $\bs$, generalises without much trouble to a $\kappa$-meagre ideal on $\gbs$, but the ideal $\c N$ of Lebesgue null sets does not; after all, there is no adequate notion of a Lebesgue measure on $\gbs$. Another important difference, is the rich structure of the club filter and of stationary sets on $\kappa$, which is nonexistent on $\omega$.

Due to the lack of a Lebesgue measure, and hence a null ideal, it is unclear what analogue in the generalised Baire space would fit for cardinal characteristics defined in terms of the null ideal. In the classical Baire space $\bs$ there is a characterisation of the cardinal characteristics $\add(\c N)$ and $\cof(\c N)$ in terms of slaloms, that is, functions $\phi:\omega\to\c P(\omega)$ such that $|\phi(n)|=2^n$ for each $n\in\omega$. Specifically, $\cof(\c N)$ is equal to the least size of a set $W$ of slaloms such that every $f\in\bs$ is localised by a slalom in $W$, that is, $f(n)\in \phi(n)$ for all but finitely many $n$. This combinatorially defined cardinal characteristic is usually named the localisation number. 

We will consider a parametrised localisation number, where we consider a set of slaloms that localises all $f\in\bs$, with the additional requirement that each slalom $\phi$ in our witness set $W$ is bound by a function $h\in\bs$, in the sense that $|\phi(n)|\leq h(n)$ for all $n\in\omega$. The effect of this parameter $h$ is rather limited on $\bs$, since the $h$-localisation cardinal is equal to $\cof(\c N)$ for any parameter $h$ that is increasing and cofinal, and thus its cardinality does not depend on the choice of $h$. On the generalised Baire space $\gbs$ for $\kappa$ inaccessible, the situation is very different, as it is consistent that there exist many parametrised localisation numbers of different cardinalities, as shown by Brendle, Brooke-Taylor, Friedman \& Montoya in \cite{BBFM} and improved by the author in \cite{vdV}.

On $\bs$, we gain a significant amount of complexity when we do not consider cardinal characteristics defined on the entire space $\bs$, but on a bounded subspace $\prod_{n\in\omega}b(n)$ for some increasing cofinal $b\in\bs$. This results in $(b,h)$-localisation numbers, defined with $(b,h)$-slaloms $\phi$, where $|\phi(n)|\leq h(n)$ and $\phi(n)\subset b(n)$ for all $n\in\omega$, localising functions $f\in\prod_{n\in\omega} b(n)$. Such cardinals were first studied by Goldstern \& Shelah \cite{GS93}. Later, a connection between these cardinals, Yorioka ideals and the strong measure zero ideal was given by Osuga \& Kamo \cite{OK08,OK14} and it was shown recently by Cardona, Klausner \& Mej\'ia \cite{CKM21} that there can be models with $2^{\aleph_0}$ many $(b,h)$-localisation numbers with mutually distinct cardinality. 

The goal of this article is to study such parametrised localisation numbers in the context of $\gbs$ where $\kappa$ is inaccessible. Alongside, we will also describe antilocalisation cardinals, whose (non-parametrised) cardinality gives a combinatorial definition of the cardinal characteristics $\cov(\c M)$ and $\non(\c M)$, and of parametrised versions of the dominating number and the eventual difference number (which also has non-parametrised versions that happen to be equal to $\cov(\c M)$ and $\non(\c M)$).

After introducing the notation and giving an overview of the situation on the (unbounded) Baire space in section \ref{section: notation}, we will use section \ref{section: triviality} to investigate for which parameters $b$ and $h$ our cardinal numbers are nontrivial, in the sense that their values lie between $\kappa^+$ and $2^\kappa$ without being equal to either bound. Although we can describe sufficient conditions for the cardinals to be nontrivial, there are still several open problems concerning the necessity of these conditions.

In section \ref{section: relations} we will give an overview of how our cardinals relate to each other. We do this by giving Tukey connections between relational systems. We also introduce cardinal characteristics defined by taking the infimum or supremum of (anti)localisation cardinals over all possible parameters.

Finally, in section \ref{section: consistency} we will prove consistency results regarding our cardinals. In particular we will construct a model in which there are $\kappa^+$ many mutually distinct localisation cardinals (using the same type of forcing notion as in \cite{vdV}), and we describe a family of $\kappa$ many parameters such that any finite set of parameters has associated antilocalisation cardinals that can be forced to be mutually distinct (using a generalisation of the forcing notion from \cite{CM19}).

\section{Notation \& Definition of Concepts}\label{section: notation}
Let us establish our notation. Usually we will use the Greek letters $\alpha,\beta,\gamma,\delta,\epsilon,\eta,\xi$ for ordinals, while $\kappa,\lambda,\mu,\nu$ will be used for cardinals. The class of all ordinals is written as $\r{Ord}$. 

We say that a property $P$ holds for \emph{almost all} $\alpha\in\kappa$ if there is a $\beta\in\kappa$ such that $P(\alpha)$ holds for all $\alpha>\beta$, and we abbreviate this as $\forall^\infty\alpha P(\alpha)$. Dually, a property is said to hold for \emph{cofinally many} $\alpha\in\kappa$ if for every $\beta\in\kappa$ there is some $\alpha>\beta$ for which $P(\alpha)$ holds, abbreviated as $\exists^\infty\alpha P(\alpha)$. 

Given two functions $f,f'$ with domain $\kappa$ and $\vartriangleleft$ a relation defined on $\ran(f)\times\ran(f')$, we write 
\begin{itemize}
\item $f\rl \vartriangleleft f'$ as a shorthand for $\forall\alpha\in\kappa(f(\alpha)\rl \vartriangleleft f'(\alpha))$,
\item $f\rl \vartriangleleft^* f'$ as a shorthand for $\forall^\infty\alpha(f(\alpha)\vartriangleleft f'(\alpha))$,
\item $f\rl \vartriangleleft^\infty f'$ as a shorthand for $\exists^\infty\alpha(f(\alpha)\vartriangleleft f'(\alpha))$.
\end{itemize}
The intended meaning of $\cancel{\vartriangleleft^*}$ and $\cancel{\vartriangleleft^\infty}$ are the negations of $\vartriangleleft^*$ and $\vartriangleleft^\infty$ respectively, as should be clear on sight, in contrast to the ambiguously notated $\not\hspace{-.3mm}\vartriangleleft^*$  and $\not\hspace{-.3mm}\vartriangleleft^\infty$. For this reason we will henceforth use the former notation.

Unsurprisingly, we write cardinal exponentiation as $\lambda^\mu$ and use the abbreviation $\lambda^{<\mu}=\Cup_{\alpha\in\mu}\lambda^{|\alpha|}$. If instead we want to discuss the set of functions from $X$ to $Y$, we write this as ${}^XY$. If $\alpha$ is an ordinal, then $^{<\alpha}Y$ denotes the set $\Cup_{\xi\in\alpha}{}^\xi Y$. We define 
\begin{align*}
[Y]^\mu&=\st{X\in\c P(Y)\mid |X|=\mu}\\
[Y]^{<\mu}&=\st{X\in\c P(Y)\mid |X|<\mu}.
\end{align*}
Naturally, ${}^{\leq\alpha}Y$ and $[Y]^{\leq\alpha}$ have the obvious meaning that is implicit from the above.
Given sequences $s\in{}^\alpha X$ and $t\in{}^\beta X$, we write $s^\frown t\in{}^{\alpha+\beta}X$ for the concatenation of $s$ and $t$. If $x\in X$, we write $\ab{x}$ for the sequence of length $1$ containing only $x$. If $s\in{}^\alpha X$ is a sequence, we write $\ot(s)=\alpha$ for the length or order-type of $s$. Similarly, if $A\subset\r{Ord}$ is a set of ordinals, we write $\ot(A)$ for the order-type of $\ab{A,\in}$.

An increasing function $f:\kappa\to\r{Ord}$ is called \emph{continuous} at $\gamma\in\kappa$ if $f(\gamma)=\Cup_{\alpha<\gamma}f(\alpha)$ and \emph{discontinuous} at $\gamma\in \kappa$ if $f(\gamma)>\Cup_{\alpha<\gamma}f(\alpha)$. If $A\subset\kappa$, then $f$ is (dis)continuous on $A$ if $f$ is (dis)continuous at $\gamma$ for every limit ordinal $\gamma\in A$.

If $f,g$ are functions, we interpret arithmetical operators on the functions elementwise, such as $f+g:\alpha\mapsto f(\alpha)+g(\alpha)$ and $2^f:\alpha\mapsto 2^{f(\alpha)}$ and if $\xi$ is an ordinal $f+\xi:\alpha\mapsto f(\alpha)+\xi$. We will often work with functions $b$ where $b(\alpha)$ is a cardinal for each $\alpha$. In such cases, we also establish that $\cf(b):\alpha\mapsto\cf(b(\alpha))$ and $b^+:\alpha\mapsto (b(\alpha))^+$. Finally if $\alpha$ is an ordinal, we write $\bar\alpha$ for the constant function $\kappa\to\st\alpha$.

\subsection{Bounded Generalised Baire Spaces}

As the name suggests, the \emph{generalised Baire space} $\gbs$ is a generalisation of the classical Baire space $\bs$. Instead of functions between the natural numbers, we consider functions between uncountable cardinals. In our case, we are specifically interested in generalised Baire spaces where $\kappa$ is a (strongly) inaccessible cardinal. As such, we will fix the convention for the remainder of this article that $\kappa$ is \emph{inaccessible}. We would like to emphasise that \emph{we leave out the mention of $\kappa$ being inaccessible from our lemmas and theorems}, even if this is a necessary assumption for the theorem to be provable.

We restrict our attention to inaccessible cardinals, because the Cicho\'n diagram on $\gbs$ behaves quite similar to its counterpart on $\bs$. We refer to \cite{BBFM} for an overview of the Cicho\'n diagram on the generalised Baire space and to \cite{Bre22} for the case where $\kappa$ is not inaccessible.

The generalised Baire space $\gbs$ is often studied as a topological space by endowing $\kappa$ with the discrete topology and $\gbs$ with the ${<}\kappa$-box topology. This is the topology generated by basic open sets of the form $[s]=\st{f\in\gbs\mid s\subset f}$, where $s\in{}^\alpha\kappa$ for some $\alpha\in\kappa$. Note that each basic set $[s]$ is in fact clopen.

The main focus of this article will be to study cardinal characteristics in the context of bounded subspaces of the generalised Baire space. Given a function $b\in\gbs$, we consider the product space $\prod b=\prod_{\alpha\in\kappa}b(\alpha)$ where each $b(\alpha)$ has the discrete topology and $\prod b$ once again has the ${<}\kappa$-box topology. We will write 
\begin{align*}
\textstyle\prod_{<\kappa}b=\Cup_{\gamma\in\kappa}\rb{\prod_{\alpha\in\gamma}b(\alpha)},
\end{align*}
then the topology on $\prod b$ is generated by sets $[s]=\st{f\in\prod b\mid s\subset f}$ with $s\in\prod_{<\kappa}b$. As such, we see that $\prod b$ is a closed subspace of $\gbs$. We will refer to spaces of the form $\prod b$ as \emph{bounded generalised Baire spaces}. Generally we will only consider \emph{increasing} functions $b\in\gbs$ such that $b(\alpha)$ is an \emph{infinite cardinal} for every $\alpha\in\kappa$, unless we state otherwise. 

A set $X\subset\prod b$ is \emph{nowhere dense} if for every basic open $[s]\subset\prod b$ there exists a basic open $[t]\subset[s]$ such that $[t]\cap X=\emp$. It is easy to show that the closure of a nowhere dense set is nowhere dense. The complement of a closed nowhere dense set is an \emph{open dense set}. We will write $\nwdb$ for the family of nowhere dense subsets of $\prod b$. A set $X$ is \emph{$\kappa$-meagre} if there exists a family $\st{N_\alpha\mid \alpha\in\kappa}$ of nowhere dense sets such that $X=\Cup_{\alpha\in\kappa}N_\alpha$. Let $\mgrb$ be the set of $\kappa$-meagre subsets of $\prod b$. We will use the notation $\c M_\kappa$ for the family $\kappa$-meagre sets of $\gbs$.

Given a ${\leq}\kappa$-complete ideal $\c I$ on the space $X$, we define the following four cardinal functions:
\begin{align*}
\cov(\c I)&=\min\st{|\c C|\smid \c C\subset\c I\la \ts\Cup\c C=X}\\
\non(\c I)&=\min\st{|N|\smid N\subset X\la N\notin\c I}\\
\add(\c I)&=\min\st{|\c A|\smid \c A\subset\c I\la \ts\Cup\c C\notin\c I}\\
\cof(\c I)&=\min\st{|\c F|\smid \c F\subset\c I\la \forall I\in\c I\exists J\in\c F(I\subset J)}
\end{align*}
It is easy to show that families $\mgrb$ (and $\c M_\kappa)$ are ${\leq}\kappa$-complete proper ideals. Moreover, by the following lemma, the choice of $b\in\gbs$ does not matter for the value of these cardinal functions evaluated over $\mgrb$, and indeed, we could work with the cardinal functions evaluated over $\c M_\kappa$ without loss of generality.

\begin{lmm}
There exists $M\in\mgrb$ such that the subspace $\prod b\setminus M$ of $\prod b$ is homeomorphic to $\gbs$.
\end{lmm}
\begin{proof}
First, we recursively define a function $\phi:\gfbs\to\prod_{<\kappa}b$. Let $\phi(\emp)=\emp$. Given $t\in\gfbs$ such that $\phi(t)$ has been defined, let $A=\st{s_\alpha\mid \alpha\in\kappa}\subset\prod_{<\kappa}b$ be an antichain of size $\kappa$ with $\phi(t)\subset s_\alpha$, such that $A$ is maximal with this property, and send $\phi:t^\frown\ab\alpha\mapsto s_\alpha$. If $\gamma=\dom(t)$ is limit and $\phi(t\restriction\alpha)$ is defined for each $\alpha<\gamma$, we have by construction that 
\begin{align*}
\ts\Cup_{\alpha\in\gamma}\phi(t\restriction\alpha)\in\prod_{<\kappa}b.
\end{align*}
Therefore we set $\phi(t)=\Cup_{\alpha\in\gamma}\phi(t\restriction\alpha)$.

Note that $\phi[{}^\alpha\kappa]$ forms a maximal antichain in $\prod_{<\kappa}b$: if $\alpha$ is the least ordinal such that $\phi[{}^\alpha\kappa]$ is not maximal, let $s\notin\phi[{}^\alpha\kappa]$ be such that $\phi[{}^\alpha\kappa]\cup\st s$ is an antichain, then for any $\xi<\alpha$ there is $s_\xi\subset s$ such that $s_\xi\in\phi[{}^\xi\kappa]$, but then $\Cup_{\xi\in\alpha}s_\xi=s'\subset s$ and $s'\in\phi[{}^\alpha\kappa]$.

We define $\Phi:\gbs\to\prod b$ induced by $\phi$ as sending $f\mapsto\Cup_{\alpha\in\kappa}\phi(f\restriction\alpha)$. Note that the range $\Phi[\gbs]$ is $\kappa$-comeagre in $\prod b$, because for each $\alpha\in\kappa$ we can define
\begin{align*} 
F_\alpha=\st{f\in\ts\prod b\mid f\notin \Cup_{t\in{}^\alpha\kappa}[\phi(k)]}.
\end{align*}
Then $F_\alpha$ is nowhere dense, since $\phi[{}^\alpha\kappa]$ is a maximal antichain, and $\Phi[\gbs]=\prod b\setminus \Cup_{\alpha\in\kappa}F_\alpha$, hence $M=\Cup_{\alpha\in\kappa}F_\alpha$ is a $\kappa$-meagre set and we will see that $\Phi:\gbs\bij\prod b\setminus M$ is a homeomorphism.

It is clear that $\Phi$ is bijective. Note that $\Phi$ is open, since for any $t\in\gfbs$ we have $\phi(t)\in\prod_{<\kappa}b$, thus $\Phi[[t]]=[\phi(t)]\setminus M$ is open in $\prod b\setminus M$. To see that $\Phi$ is continuous, let $s\in\prod_{<\kappa}b$ with $\dom(s)=\alpha$. If $t\in{}^\alpha\kappa$, then it follows by construction that $\alpha\subset\dom(\phi(t))$. Let 
\begin{align*}
T=\st{t\in{}^\alpha\kappa\mid s\subset \phi(t)}.
\end{align*}
Note that $f\in\Cup_{t\in T}[t]$ if and only if $s\subset\phi(f\restriction\alpha)$ off $\Phi(f)\in[s]$ if and only if $f\in\Phi^{-1}[[s]]$.
\end{proof}

\begin{crl}
$\cov(\c M_\kappa)= \cov(\mgrb)$,
 $\non(\c M_\kappa)=\non(\mgrb)$,
 $\add(\c M_\kappa)=\add(\mgrb)$,
 $\cof(\c M_\kappa)=\cof(\mgrb)$.
\end{crl}

Apart from the $\kappa$-meagre ideal, we will also consider the ideal of $\kappa$-strong measure zero sets. Contrary to Lebesgue measure zero sets, strong measure zero sets can be generalised to the context of $\gbs$. We will work in $\gcs$, with basic open sets $[s]=\st{f\in\gcs\mid s\subset f}$ for $s\in\gfcs$. A set $X\subset\gcs$ is called \emph{$\kappa$-strong measure zero} if for every $f\in\gbs$ there exists a sequence $\ab{s_\alpha\in\gfcs\mid \alpha\in\kappa}$ with $\ot(s_\alpha)=f(\alpha)$ such that $X\subset\Cup_{\alpha\in\kappa}[s_\alpha]$. We write $\SN_\kappa$ for the family of $\kappa$-strong measure zero subsets of $\gcs$.

It is again easy to see that $\SN_\kappa$ is a ${\leq}\kappa$-complete proper ideal. We conclude this section by giving an equivalent definition of $\SN_\kappa$.
\begin{lmm}
The following are equivalent for a set $X\subset\gcs$.
\begin{enumerate}[label=(\arabic*)]
\item $X\in\SN_\kappa$
\item For every $f\in\gbs$ there exists a sequence $\bar s=\ab{s_\alpha\mid \alpha\in\kappa}$ with $s_\alpha\in{}^{f(\alpha)}2$ for each $\alpha$ such that $X\subset\Cup_{\alpha\in\kappa}[s_\alpha]$.
\item For every $f\in\gbs$ there exists a sequence $\bar s=\ab{s_\alpha\mid \alpha\in\kappa}$ with $s_\alpha\in{}^{f(\alpha)}2$ for each $\alpha$ such that $X\subset\Cap_{\alpha_0\in\kappa}\Cup_{\beta\in[\alpha_0,\kappa)}[s_\beta]$.\qedhere
\end{enumerate}
\end{lmm}
\begin{proof}
Note that (2) is simply the definition of elements of $\SN_\kappa$. That (3) implies (2) is obvious.

To see that (2) implies (3), let $\pi:\kappa\times\kappa\bij\kappa$ be a bijection and $f\in\gbs$. We define $f_\xi\in\gbs$ by $f_\xi(\alpha)=f(\pi(\xi,\alpha))$ and use (2) to find a sequence $\bar s^\xi=\sab{s_\alpha^\xi\mid \alpha\in\kappa}$ with $s_\alpha^\xi\in{}^{f_\xi(\alpha)}2$ such that $X\subset\Cup_{\alpha\in\kappa}[s^\xi_\alpha]$. Given $\xi,\alpha\in\kappa$, we define $s_{\pi(\xi,\alpha)}=s_\alpha^\xi$ then $s_\beta\in f(\beta)$ for all $\beta\in\kappa$. It follows that $X\subset\Cap_{\alpha_0\in\kappa}\Cup_{\beta\in[\alpha_0,\kappa)}[s_\beta]$.
\end{proof}
If $\bar s=\ab{s_\alpha\mid \alpha\in\kappa}$ is a sequence such that $X\subset\Cup_{\alpha\in\kappa}[s^\xi_\alpha]$, we say $X$ is \emph{covered} by $\bar s$. If also $X\subset\Cap_{\alpha_0\in\kappa}\Cup_{\beta\in[\alpha_0,\kappa)}[s_\beta]$, then we say $X$ is \emph{cofinally covered} by $\bar s$.

\subsection{Cardinal Characteristics \& Relational Systems}

In order to study the relations between our cardinal characteristics, we will make use of relational systems. All of the cardinal characteristics we are interested in, can be expressed as the norm of a relational system, thus we will use the language of relational systems to define our cardinal characteristics. We will only give a concise description of relational systems below, and refer to \cite{Blass} for a detailed description.

A relational system $\sr R=\ab{X,Y, R}$ is a triple of sets $X$ and $Y$ and a relation $R\subset X\times Y$. We define the \emph{norm} of $\sr R$ as 
\begin{align*}
\norm{\sr R}=\min\st{|W|\mid W\subset Y\text{ and }\forall x\in X\exists y\in W (x\rl R y)}.
\end{align*}
We refer to a set $W\subset Y$ such that $\forall x\in X\exists y\in W(x\rl R y)$ as a \emph{witness} for $\norm{\sr R}\leq |W|$. We define a \emph{dual} relation to $R$ by
\begin{align*}
\cancel R^{-1}=\st{(y,x)\in Y\times X\mid (x,y)\notin R}.
\end{align*}
Correspondingly we can define the dual relational system $\sr R^\lc =\ab{Y,X,\cancel R^{-1}}$.

Intuitively, we can see $Y$ as a set of possible responses and $X$ as the set of potential challenges. We want to find a set of responses $Y'\subset Y$ such that every challenge can be met with a response from $Y'$, and the norm expresses the least number of responses necessary to do so. The dual relational system answers the question how many challenges we should gather such that no single response can meet them all.

Given two relational systems $\sr R=\ab{X,Y,R}$ and $\sr R'=\ab{X',Y',R'}$, a \emph{Tukey connection} from $\sr R$ to $\sr R'$ is a pair of functions $\rho_-:X\to X'$ and $\rho_+:Y'\to Y$ such that for any $x\in X$ and $y'\in Y'$ with $(\rho_-(x),y')\in R'$ we also have $(x, \rho_+(y'))\in R$. We let $\sr R\preceq\sr R'$ denote the claim that there exists a Tukey connection from $\sr R$ to $\sr R'$, and we let $\sr R\equiv\sr R'$ abbreviate $\sr R\preceq\sr R'\preceq\sr R$. Note that if $\ab{\rho_-,\rho_+}$ witnesses $\sr R\preceq \sr R'$, then $\ab{\rho_+,\rho_-}$ witnesses $\sr R'^\lc\preceq \sr R^\lc$.

We will use Tukey connections to give an ordering between the norms of relational systems, through the following lemma.

\begin{lmm}\renewcommand{\qed}{\hfill$\square$}
If $\sr R\preceq \sr R'$, then $\norm{\sr R}\leq \norm{\sr R'}$ and $\norm{\sr R'^\lc}\leq\norm{\sr R^\lc}$.
\end{lmm}
We will on one occasion require a composition of several relational systems. In particular we need a general form of the categorical product that is defined in \cite{Blass}. We will state the definition and give two lemmas that are useful to compute the norm.

Let $\sr R_\alpha=\ab{X_\alpha,Y_\alpha,R_\alpha}$ be relational systems for each $\alpha\in A$, where $A$ is a set of ordinals (but we will omit mention of $A$ from now on for the sake of brevity). We define the categorical product as the system $\bigotimes_\alpha \sr R_\alpha=\ab{\Cup_\alpha (X_\alpha\times\st\alpha),\prod_\alpha Y_\alpha,Z}$, where $(x,\alpha)\rl Z\bar y$ off $x\rl R_\alpha \bar y(\alpha)$. Dually we have the categorical coproduct $\bigoplus_{\alpha}\sr R_\alpha=(\bigotimes_{\alpha}\sr R_\alpha^\lc)^\lc=\ab{\prod_\alpha X_\alpha,\Cup_\alpha(Y_\alpha\times\st\alpha),N}$, where $\bar x\rl N(y,\alpha)$ iff $\bar x(\alpha)\rl R_\alpha y$.

\begin{lmm}
$\norm{\bigotimes_{\alpha} \sr R_\alpha}=\sup_{\alpha}\norm{\sr R_\alpha}$.
\end{lmm}
\begin{proof}

($\leq$)\quad Let $Y'_\alpha\subset Y_\alpha$ be a witness for $|Y'_\alpha|=\norm{\sr R_\alpha}$. Let $\lambda=\sup_{\alpha}\norm{\sr R_\alpha}$ and let $\sigma_\alpha:\lambda\srj Y'_\alpha$ be surjections. Define $\bar y_\xi:\alpha\mapsto \sigma_\alpha(\xi)$ and let $Y=\st{\bar y_\xi\mid \xi\in\lambda}\subset\prod_{\alpha} Y_\alpha$. If $x\in X_\alpha$, let $y\in Y'_\alpha$ such that $x\rl R_\alpha y$ and find $\xi\in\lambda$ such that $y=\sigma_\alpha(\xi)$, then $(x,\alpha)\rl Z\bar y_\xi$. Hence $Y$ witnesses that $\norm{\bigotimes_\alpha\sr R_\alpha}\leq\lambda$.

($\geq$)\quad Let $Y\subset\prod_{\alpha} Y_\alpha$ be a witness for $|Y|=\norm{\bigotimes_{\alpha}\sr R_\alpha}$, and let $Y'_\alpha=\st{\bar y(\alpha)\mid \bar y\in Y}\subset Y_\alpha$. Then for every $x\in X_\alpha$ there exists $\bar y\in Y$ such that $(x,\alpha)\rl Z\bar y$, hence $\bar y(\alpha)\in Y'_\alpha$ has the property that $x\rl R_\alpha \bar y(\alpha)$, thus $\norm{\sr R_\alpha}\leq |Y'|\leq |Y|=\norm{\bigotimes_{\alpha}\sr R_\alpha}$.
\end{proof}

\begin{lmm}
$\norm{\bigoplus_{\alpha} \sr R_\alpha}=\inf_{\alpha}\norm{\sr R_\alpha}$.
\end{lmm}
\begin{proof}
($\leq$)\quad Let $Y\subset Y_\alpha$ be a witness for $|Y|=\norm{\sr R_\alpha}$. Let $Y'=\st{(y,\alpha)\mid y\in Y}$. If $\bar x\in \prod_\alpha X_\alpha$, then there exists $y\in Y$ such that $\bar x(\alpha)\rl R_\alpha y$ and thus $\bar x\rl N(y,\alpha)$. Therefore $\norm{\bigoplus_\alpha \sr R_\alpha}\leq |Y'|=|Y|=\norm{\sr R_\alpha}$.

($\geq$)\quad Let $Y\subset\Cup_\alpha(Y_\alpha\times\st\alpha)$ with $|Y|<\inf_\alpha \norm{\sr R_\alpha}$. Let $Y'_\alpha=\st{y\in Y_\alpha\mid (y,\alpha)\in Y}$. Since $|Y'_\alpha|<\norm{\sr R_\alpha}$ there is $x_\alpha\in X_\alpha$ such that $x_\alpha\nrl R_\alpha y$ for all $y\in Y'_\alpha$. Let $\bar x:\alpha\mapsto x_\alpha$, then $\bar x\nrl N (y,\alpha)$ for every $(y,\alpha)\in Y$. Thus $|Y|<\norm{\bigoplus_\alpha\sr R_\alpha}$.
\end{proof}

Apart from the previously mentioned cardinal functions on ideals, we will define four other families of cardinals. In order to do so, we need to discuss \emph{slaloms}, a concept first formulated by Bartoszy\'nski to give a combinatorial characterisation of the additivity and cofinality of the Lebesgue null ideal in the classical Baire space (see \cite{Bart87,BJ} or the next subsection).

\begin{dfn}\label{slalom definition}
Let $h\in\gbs$ be an increasing function such that $h(\alpha)$ is a nonzero cardinal for all $\alpha\in\kappa$. We will generally assume $h\leq b$. Let $\Loc_\kappa^{b,h}$ be the set of functions $\phi$ with $\dom(\phi)=\kappa$ such that $\phi(\alpha)\subset b(\alpha)$ and $\card{\phi(\alpha)}< h(\alpha)$ for each $\alpha\in\kappa$. If $\phi\in\Loc_\kappa^{b,h}$, then we call $\phi$ a \emph{$(b,h)$-slalom}.\footnote{Note that this definition differs from the usual definition: we define $\phi$ such that $|\phi(\alpha)|<h(\alpha)\leq b(\alpha)$, instead of the traditional $|\phi(\alpha)|\leq h(\alpha)<b(\alpha)$. Our definition is more versatile. For example, if $h(\alpha)$ is a limit cardinal for each $\alpha$, the resulting set of all slaloms $\phi$ with $|\phi(\alpha)|<h(\alpha)$ cannot be expressed using the traditional definition. On the other hand, the set of all traditionally defined $(b,h)$-slaloms is the set of $(b,h^+)$-slaloms under our definition.} 

Given $\phi\in\Loc_\kappa^{b,h}$ and $f\in \prod b$, then we say $\phi$ \emph{localises} $f$ iff $f\in^* \phi$ and $\phi$ \emph{antilocalises} $f$ iff $f\nini\phi$.

Similar to how we fixed the notation $\prod_{<\kappa}b$ for initial segments of functions in $\prod b$, we will fix the notation $\Loc_{<\kappa}^{b,h}$ for initial segments of $(b,h)$-slaloms, that is, $s\in\Loc_{<\kappa}^{b,h}$ if and only if $\dom(s)\in\kappa$ and for all $\alpha\in\dom(s)$ we have $s(\alpha)\in[b(\alpha)]^{<h(\alpha)}$.
\end{dfn}

We now define the following four relational systems and associated cardinals:
\begin{align*}
\sr D_b&=\ab{\textstyle\prod b,\prod b,\leq^*}&
\norm{\sr D_b}&=\dleq{b}&
\norm{{\sr D_b^\lc}}&=\bleq{b}\\
\ED_b&=\ab{\textstyle\prod b,\prod b,\neqi}&
\norm{\ED_b}&=\dneq{b}&
\norm{{\ED_b^\lc}}&=\bneq{b}\\
\sr L_{b,h}&=\ab{\textstyle\prod b,\Loc_\kappa^{b,h},\in^*}&
\norm{\sr L_{b,h}}&=\dstar{b,h}&
\norm{{\sr L_{b,h}^\lc}}&=\bstar{b,h}\\
\AL_{b,h}&=\ab{\Loc_\kappa^{b,h},\textstyle\prod b,\nnii}&
\norm{\AL_{b,h}}&=\dinf{b,h}&
\norm{{\AL_{b,h}^\lc}}&=\binf{b,h}
\end{align*}
We call $\dleq{b}$ the \emph{dominating} number, $\bleq{b}$ the \emph{unbounding} number, $\dneq{b}$ the \emph{eventual difference} number, $\bneq{b}$ the \emph{cofinal equality} number\footnote{In the classical setting, the relation of being cofinally equal is often referred to as being infinitely equal, but in our generalised setting we require more than just an infinite amount of equality.}. Both $\dstar{b,h}$ and $\bstar{b,h}$ are sometimes called \emph{localisation} cardinals, but we believe it will be useful to distinguish between these classes of cardinals with separate terms. We therefore opt to call $\dstar{b,h}$ the \emph{localisation} cardinal, and $\bstar{b,h}$ the \emph{avoidance} cardinal, for the reason that a witness to $\bstar{b,h}$ is a set of functions $F\subset\prod b$ that cannot be localised by a single $(b,h)$-slalom; localisation by any specific slalom is always avoided by at least one member of $F$. Finally this leaves the cardinals $\dinf{b,h}$, which we will call the \emph{anti-avoidance cardinal} and $\binf{b,h}$, which we will call the \emph{antilocalisation} cardinal.

It should be mentioned that (anti)localisation cardinals not only suffer from ambiguous names, they also suffer from nonuniformity of notation. In previous literature, the notations $\f c_{b,h}^\forall$ and $\f v_{b,h}^\forall$ have also frequently been used for $\f d_\omega^{b,h}(\in^*)$ and $\f b_\omega^{b,h}(\in^*)$ respectively, while $\f c_{b,h}^\exists$ and $\f v_{b,h}^\exists$ have been used for $\f b_\omega^{b,h}(\nnii)$ and $\f d_\omega^{b,h}(\nnii)$ respectively. For example, this notation was used in \cite{Kel08,OK14,KM21,CKM21}. Here $\f c$ stands for a {\bf c}over of $\prod b$ with slaloms by the relations $\in^*$ or $\in^\infty$, and $\f v$ stands for a{\bf v}oidance by or e{\bf v}asion by a single slalom. We believe that \emph{avoidance} forms the better antonym to \emph{localisation}, as it prevents confusion with other cardinal characteristics known as evasion cardinals (such as described in the eponymous section of \cite{Blass}). Our choice to use the notation $\dstar{b,h}$ over $\f c_{b,h}^\forall$ has the main benefit that the relevant relational system can be deduced from our notation. Finally, we opted to use the relation $\nnii$ instead of $\in^\infty$ in our relational system for antilocalisation, since we will later see that $\dinf{b,h}$ bears more similarities to $\dstar{b,h}$ than to $\bstar{b,h}$.

\subsection{Classical and Higher Cicho\'n Diagram}

In the next section we will investigate the provable relations between the cardinals we have defined and what influence the parameters $b$ and $h$ have on their cardinality. Before doing so, we would like to discuss what is known about the case where we let $b=\bar\kappa$ be the constant function, in other words, what is known about these cardinals in the (unbounded) generalised Baire space, and how this compares to the same cardinals as defined classically on $\bs$. When discussing the cardinals in the context of the entire space $\gbs$, we simply leave out the parameter $b$ from the notation, e.g. $\f d_\kappa(\leq^*)$ instead of $\dleq{\bar \kappa}$ and $\binf{h}$ instead of $\binf{\bar\kappa,h}$. When discussing the cardinal in context of the classical Baire space, we replace $\kappa$ by $\omega$, e.g. $\f b_\omega^b(\neqi)$.

In the context of the classical Baire space $\bs$, each of the cardinals mentioned so far is an element of the Cicho\'n diagram. Let $\c N$ denote the ideal of sets of reals with negligible Lebesgue measure, and let $\c M$ denote the ($\omega$-)meagre ideal on $\bs$. The Cicho\'n diagram gives an overview of the relations between our cardinals, where $\f x\leq \f y$ is provable in $\s{ZFC}$ if and only if there exists a path from $\f x$ to $\f y$ following the arrows.

\begin{center}
\begin{tikzpicture}[xscale=2.5,yscale=1.3]

\node (a1) at (0.25,0) {$\aleph_1$};
\node (aN) at (1,0) {$\r{add}(\c N)$};
\node (aM) at (2,0) {$\r{add}(\c M)$};
\node (cN) at (1,2) {$\r{cov}(\c N)$};
\node (nM) at (2,2) {$\r{non}(\c M)$};
\node (b) at (2,1) {$\f b_\omega(\leq^*)$};
\node (d) at (3,1) {$\f d_\omega(\leq^*)$};
\node (cM) at (3,0) {$\r{cov}(\c M)$};
\node (nN) at (4,0) {$\r{non}(\c N)$};
\node (fM) at (3,2) {$\r{cof}(\c M)$};
\node (fN) at (4,2) {$\r{cof}(\c N)$};
\node (c) at (4.75,2) {$2^{\aleph_0}$};

\draw (a1) edge[->] (aN);
\draw (aN) edge[->] (cN);
\draw (aN) edge[->] (aM);
\draw (cN) edge[->] (nM);
\draw (aM) edge[->] (b);
\draw (aM) edge[->] (cM);
\draw (b) edge[->] (nM);
\draw (b) edge[->] (d);
\draw (nM) edge[->] (fM);
\draw (cM) edge[->] (d);
\draw (cM) edge[->] (nN);
\draw (d) edge[->] (fM);
\draw (fM) edge[->] (fN);
\draw (nN) edge[->] (fN);
\draw (fN) edge[->] (c);

\end{tikzpicture}
\end{center}
Proofs for these relations and for the following facts can also be found in \cite{BJ} or \cite{Blass}. 
\begin{fct}[\ \!\!\cite{Truss,M81}]
$\add(\c M)=\min\st{\f b_\omega(\leq^*),\cov(\c M)}$ and $\cof(\c M)=\max\st{\f d_\omega(\leq^*),\non(\c M)}$
\end{fct}
\begin{fct}[\ \!\!\cite{Bart87}]
$\cov(\c M)=\f d_\omega(\neqi)=\f d_\omega^h(\nnii)$ and $\non(\c M)=\f b_\omega(\neqi)=\f b_\omega^h(\nnii)$ and $\cof(\c N)=\f d^h_\omega(\ins)$ and $\add(\c M)=\f b^h_\omega(\ins)$ for any increasing cofinal $h\in\bs$.
\end{fct}

In the context of the generalised Baire space $\gbs$ (where we assume, as mentioned, that $\kappa$ is inaccessible), we do not have Lebesgue measure, thus the four cardinal invariants related to $\c N$ cannot be generalised. The other cardinals can be generalised, as we have defined in the previous section, and the Cicho\'n diagram looks identical to the classical case if we assume that $\kappa$ is strongly inaccessible\footnote{This assumption is necessary for all of the cardinals of (the middle part of) the Cicho\'n diagram to be nontrivial, but there are some interesting questions for $\kappa$ being accessible, see \cite{Bre22}.} and $h$ is a cofinal increasing function. 

\begin{center}
\begin{tikzpicture}[xscale=2.5,yscale=1.3]

\node (a1) at (0.25,0) {$\kappa^+$};
\node (aN) at (1,0) {$\bstar{h}$};
\node (aM) at (2,0) {$\r{add}(\c M_\kappa)$};
\node (nM) at (2,2) {$\r{non}(\c M_\kappa)$};
\node (b) at (2,1) {$\bleq{}$};
\node (d) at (3,1) {$\dleq{}$};
\node (cM) at (3,0) {$\r{cov}(\c M_\kappa)$};
\node (fM) at (3,2) {$\r{cof}(\c M_\kappa)$};
\node (fN) at (4,2) {$\dstar{h}$};
\node (c) at (4.75,2) {$2^{\kappa}$};

\draw (a1) edge[->] (aN);
\draw (aN) edge[->] (aM);
\draw (aM) edge[->] (b);
\draw (aM) edge[->] (cM);
\draw (b) edge[->] (nM);
\draw (b) edge[->] (d);
\draw (nM) edge[->] (fM);
\draw (cM) edge[->] (d);
\draw (d) edge[->] (fM);
\draw (fM) edge[->] (fN);
\draw (fN) edge[->] (c);

\end{tikzpicture}
\end{center}
Proofs for these relations can be found in \cite{BBFM}.
\begin{fct}[\ \!\!\cite{BBFM}]\label{add cof meagre}
$\add(\c M_\kappa)=\min\st{\bleq{},\cov(\c M_\kappa)}$ and $\cof(\c M_\kappa)=\max\st{\dleq{},\non(\c M_\kappa)}$
\end{fct}
\begin{fct}[\ \!\!\cite{BHZ}, \cite{L}]\label{evt dif and meagre}\renewcommand{\qed}{\hfill$\square$}
$\non(\c M_\kappa)=\bneq{}$ and $\cov(\c M_\kappa)=\dneq{}$.
\end{fct}

Contrary to the classical case, the parameter $h$ of $\dstar{h}$ matters in the generalised case. If $\id:\alpha\mapsto |\alpha|$ and $\pow:\alpha\mapsto 2^{|\alpha|}$, then the corresponding cardinals are consistently different:
\begin{fct}[\ \!\!\cite{BBFM}]
It is consistent that $\dstar{\pow}<\dstar{\id}$.
\end{fct}
This consistency result can be improved considerably. 
\begin{fct}[\ \!\!\cite{vdV}]
(Consistently) there exists a family $H\subset\gbs$ with $|H|=\kappa^+$ such that for any enumeration $\ab{h_\xi\mid \xi\in\kappa^+}$ of $H$ there exists a forcing extension in which $\dstar{h_\xi}<\dstar{h_{\xi'}}$ for all $\xi<\xi'$.
\end{fct}
A dual result to the above would be the consistency of $\bstar{\id}<\bstar{\pow}$, but the forcing techniques used to give the above-mentioned consistency results cannot be dualised, and other forcing methods are hard to use due to a lack of preservation theorems in the context of $\gbs$.

Meanwhile, the parameter $h$ does not influence the cardinality of $\dinf{h}$ or $\binf{h}$, since we have the following result, which we will prove with \Cref{uninteresting antiloc}.
\begin{fct}
$\dinf{h}=\cov(\c M_\kappa)$ and $\binf{h}=\non(\c M_\kappa)$ for any choice of $h\in\gbs$.
\end{fct}
We will see that the parameter $b$ does make a difference, and that there can be consistently different cardinalities of the form $\dinf{b,h}$ for different values of $b,h\in\gbs$.

\section{Trivial Choices of Parameters}\label{section: triviality}

It is perhaps not very surprising that some choices of parameters $b$ and $h$ will result in the cardinal characteristics having trivial values. What we mean precisely with a cardinal characteristic $\f x$ having a `trivial' value, is that $\f x$ is undefined, or that it is provable in $\ap{\s{ZFC}+\kappa\text{ is inaccessible}}$ that $\f x\leq\kappa^+$ or $\f x=2^\kappa$.

It is perhaps not evident that the determination of $b$ and $h$ such that our cardinals are nontrivial, is itself not a trivial task. Indeed, this section contains several nontrivial open questions regarding the triviality of cardinal characteristics.

We will establish a general pattern that it is possible to give a complete characterisation of the cases in which the cardinals $\bleq{b}$, $\bstar{b,h}$, $\binf{b,h}$ and $\bneq{b}$ are trivial. For each of these families of cardinals we are able to formulate a trichotomy of the case where the cardinal is ${<}\kappa$, the case where it is exactly $\kappa$ and the case where the cardinal is ${>}\kappa$. In the last case we can show that the cardinal is nontrivial and we give an independence proof.

For the cardinals $\dleq{b}$, $\dstar{b,h}$, $\dinf{b,h}$ and $\dneq{b}$, the natural conjecture is that these are trivial exactly when their $\f b$-duals are trivial, but this turns out to be hard to prove in each case. We will give partial results and some related problems. On the other hand, we can show that these $\f d$-side cardinals are nontrivial whenever the $\f b$-side cardinals are nontrivial.

\subsection{Domination \& Unboundedness}
Let us start with $\dleq{b}$ and $\bleq{b}$. It is clear that value of these cardinals only depends on the cofinality of $b(\alpha)$ by the following lemma.

\begin{lmm}\label{dominating cofinality lemma}
$\sr D_b\equiv\sr D_{\cf(b)}$.
\end{lmm}
\begin{proof}
For each $\alpha\in\kappa$ let $\sab{\beta^\alpha_\xi\mid \xi<\cf(b(\alpha))}$ be a strictly increasing sequence of ordinals that is cofinal in $b(\alpha)$ and for any $\eta\in b(\alpha)$ let $\xi^\alpha_\eta=\min\sst{\xi\in \cf(b(\alpha))\mid \eta\leq \beta^\alpha_\xi}$

For $\sr D_b\preceq\sr D_{\cf(b)}$ let $\rho_{-}(f):\alpha\mapsto \xi^\alpha_{f(\alpha)}$ and $\rho_+(g'):\alpha\mapsto \beta^\alpha_{g'(\alpha)}$.

For $\sr D_{\cf(b)}\preceq \sr D_b$ let $\rho_-(f'):\alpha\mapsto \beta^{\alpha}_{f'(\alpha)}$ and $\rho_+(g):\alpha\mapsto \xi^\alpha_{g(\alpha)}$.
\end{proof}

It also becomes apparent quite immediately why we require $b\in\gbs$ to be such that $b(\alpha)$ is an infinite cardinal for (almost) all $\alpha\in\kappa$: 
\begin{lmm}
If $b(\alpha)$ is a successor ordinal for almost all $\alpha$, then $\dleq b=1$ and $\bleq{b}$ is undefined.
\end{lmm}
\begin{proof}
Let $b(\alpha)=f(\alpha)+1$ for almost all $\alpha\in\kappa$, then clearly $f$ dominates all functions in $\prod b$, hence $\dleq b=1$ and $\bleq b$ is undefined.
\end{proof}

The following lemma gives a complete characterisation of the functions $b$ for which $\bleq{b}$ is trivial. Note that the cases (i), (ii) and (iii) form a trichotomy.
\begin{lmm}\label{bounds on bleq}
For each regular cardinal $\lambda<\kappa$ let $D_\lambda=\st{\alpha\in\kappa\mid \cf(b(\alpha))=\lambda}$.
\begin{enumerate}[label=(\roman*)]
\item If there exists a least regular cardinal $\lambda<\kappa$ such that $D_\lambda$ is cofinal in $\kappa$, then $\bleq b=\lambda$.
\item If $D_\lambda$ is bounded for all regular $\lambda<\kappa$ and there exists a stationary set $S$ such that for each $\xi\in S$ there exists $\alpha_\xi\geq \xi$ with $\cf(b(\alpha_\xi))\leq\xi$, then $\bleq b=\kappa$.
\item If $D_\lambda$ is bounded for all regular $\lambda<\kappa$ and there exists a club set $C$ such that for each $\xi\in C$ we have $\cf(b(\alpha))>\xi$ for all $\alpha\geq\xi$, then $\bleq b\geq \kappa^+$.\qedhere
\end{enumerate}
\end{lmm}
\begin{proof}
(i)\quad For each $\gamma\in D_\lambda$ let $\ab{\delta_\gamma^\alpha\mid \alpha\in\lambda}$ be an increasing cofinal sequence in $\gamma$. Let $f_\alpha\in\prod b$ be any function such that $f_\alpha(\gamma)=\delta_\gamma^\alpha$ for each $\gamma\in D_\lambda$, then we claim that $B=\st{f_\alpha\mid \alpha\in\lambda}$ is unbounded. Let $g\in\prod b$. By the pigeonhole principle there exists $\alpha\in\lambda$ such that $g(\gamma)<\delta_\gamma^\alpha$ for cofinally many $\gamma\in D_\lambda$, hence we see that $f_\alpha\mathrel{\cancel{\leq^*}}g$.

On the other hand, if $|B|<\lambda$, then let $\alpha_0$ be large enough such that $\cf(b(\alpha))\geq\lambda$ for all $\alpha\geq\alpha_0$, then $|\st{f(\alpha)\mid f\in B}|<\cf(b(\alpha))$ for all $\alpha\geq\alpha_0$, thus we can pick $g\in\prod b$ such that $g(\alpha)=\sup\st{f(\alpha)\mid f\in B}<b(\alpha)$ for each $\alpha\geq\alpha_0$ to see that $f\leq^* g$ for all $f\in B$.

(ii)\quad Since each $D_\lambda$ is bounded, we may assume that $\alpha_\xi\neq \alpha_{\xi'}$ for each $\xi\in S$. For each $\xi\in S$, let $\sst{\beta_\xi^\eta\mid\eta\in\xi}$ be a cofinal subset of $b(\alpha_\xi)$ (not necessarily increasing). Given $\alpha\in\kappa$ we define $f_\alpha$ such that $f_\alpha(\alpha_\xi)=\beta_\xi^\alpha$ for all $\xi\in S$ with $\alpha<\xi$ and arbitrary otherwise. We claim that $B=\st{f_\alpha\mid \alpha\in\kappa}$ is unbounded.

Let $g\in\prod b$, then for each $\xi\in S$ there is a minimal $\eta_\xi<\xi$ such that $g(\alpha_\xi)<\beta_\xi^{\eta_\xi}=f_{\eta_\xi}(\alpha_\xi)$, hence $g':\xi\mapsto \eta_\xi$ is a regressive function on $S$. Therefore there is stationary $S'\subset S$ such that $g'$ is constant on $S'$, say with value $\eta$, then we see that $f_\eta\mathrel{\cancel{\leq^*}}g$.

On the other hand, suppose that $B\subset\prod B$ with $\mu=|B|<\kappa$ and let $\lambda>\mu$ be regular, then $D_\lambda$ is bounded, thus there exists $\alpha_0$ such that $\cf(b(\alpha))>\mu$ for all $\alpha\geq\alpha_0$. Let $g\in\prod b$ be such that $g(\alpha)=\sup\st{f(\alpha)\mid f\in B}$ for each $\alpha\geq\alpha_0$ to see that $f\leq^*g$ for all $f\in B$.

(iii)\quad Let $B\subset\prod b$ with $|B|=\kappa$ and enumerate $B$ as $\st{f_\eta\mid \eta\in\kappa}$. Let $\ab{\alpha_\xi\mid \xi\in\kappa}$ be an increasing enumeration of $C$. We define $g(\alpha)=\Cup_{\eta\in\xi}f_\eta(\alpha)$ for all $\alpha\in[\alpha_\xi,\alpha_{\xi+1})$ and $\xi\in\kappa$, then $g(\alpha)$ is the supremum of a sequence of length $\xi\leq\alpha_\xi<\cf(b(\alpha))$, hence $g\in\prod b$. Clearly $f_\xi\leq^* g$ for each $\xi\in\kappa$, so $B$ is bounded.
\end{proof}

Hence, we see that $\bleq{b}$ is trivial in cases (i) and (ii). We will later prove that $\kappa^+<\bleq{b}$ is consistent in case (iii), but let us take a look at the dual $\dleq{b}$ first.

If $\dleq{b}$ behaves dually to $\bleq{b}$, then we expect $\dleq{b}<2^\kappa$ to be inconsistent in cases (i) and (ii). Remember that we assume that $b$ is increasing, hence by \Cref{dominating cofinality lemma}, we can reduce case (i) to the situation where $b$ is a constant function with a regular cardinal as value. In other words, we have to study the dominating number as defined in the space ${}^\kappa\lambda$ where $\lambda<\kappa$ is a regular cardinal. 

Note that even if we drop the assumption that $b$ is increasing, the following Tukey connection shows that the behaviour on the space ${}^\kappa\lambda$ is essentially the relevant part of the (in)consistency of $\dleq{b}<2^\kappa$.

\begin{lmm}\label{subsequence domination}
If $\ab{\alpha_\xi\mid \xi\in\kappa}\in\gbs$ is a strictly increasing sequence and $b':\xi\mapsto b(\alpha_\xi)$, then $\sr D_{b'}\preceq\sr D_{b}$.
\end{lmm}
\begin{proof}
Easy.
\end{proof}

Dominating numbers in the space ${}^\kappa\lambda$ have been studied by many in the past. Brendle showed in the last section of \cite{Bre22} (using different notation where the roles of $\kappa$ and $\lambda$ are reversed) that if $\lambda<\kappa$ and $\lambda$ is regular uncountable, then $\dleq{\bar\lambda}<2^\kappa$ is actually consistent. An example for a model where this holds, is the model resulting from adding $\kappa^{++}$ many $\mu$-Cohen reals over $\s{GCH}$, where $\mu<\lambda$. Because this also destroys the inaccessibility of $\kappa$, we cannot use this in our context.

The question whether $\dleq{\bar\lambda}<2^\kappa$ is consistent with $\kappa\geq 2^{<\lambda}$ is mentioned as Question 16 in \cite{Bre22}. Moreover, the special case where $\kappa=\omega_1$ and $\lambda=\omega$ is a famous open problem of Jech and Prikry {\cite{JP79}} that is still unsolved almost half a century later. We will give a partial answer and prove that $\dleq{\bar\lambda}=2^\kappa$ when $\kappa$ is inaccessible. 

\begin{thm}\label{bounds on dleq}
If $\lambda<\kappa$ is regular and $\cf(b(\alpha))=\lambda$ for cofinally many $\alpha\in\kappa$, then $\dleq{b}=2^\kappa$.
\end{thm}

We will delay the essential part of the proof of this theorem to the next subsection, since it will be  a corollary of \Cref{bounds on dstar}, which states that $\dstar{b,h}=2^\kappa$ if $h(\alpha)=\lambda$ for cofinally many $\alpha\in\kappa$. This is related to $\dleq{b}$ through the following lemma.

\begin{lmm}\label{relation star and leq 1}
If $b=^*h$, then $\sr L_{b,h}\preceq\sr D_b$.
\end{lmm}
\begin{proof}
We define $\rho_-:\prod b\to \prod b$ and $\rho_+:\prod b\to\Loc_\kappa^{b,h}$ as a Tukey connection for $\sr L_{b,h}\preceq\sr D_b$. 

We will have $\rho_+$ sending $g\in\prod b$ to $\phi\in\Loc_\kappa^{b,h}$ where $\phi(\alpha)=g(\alpha)$ whenever $h(\alpha)=b(\alpha)$ and arbitrary otherwise. Since $h=^*b$ and $g(\alpha)<b(\alpha)$ for all $\alpha\in\kappa$, this is well-defined. Let $\rho_-:f\mapsto f+1$.

If $f,g\in\prod b$, let $\phi=\rho_+(g)$. Then $f+1\leq^* g$ implies $f<^*g$, hence $f\in^*\phi$.
\end{proof}

\begin{proof}[Proof of \Cref{bounds on dleq}]
Let $\lambda<\kappa$ be regular and $\cf(b(\alpha))=\lambda$ for cofinally many $\alpha\in\kappa$. By \Cref{dominating cofinality lemma} and \Cref{subsequence domination} we may assume that $b(\alpha)=\lambda$  for all $\alpha\in\kappa$. Let $h=b$, then the conditions of \Cref{bounds on dstar,relation star and leq 1} are satisfied, thus $2^\kappa=\dstar{b,h}\leq\dleq{b}\leq2^\kappa$.
\end{proof}

We currently do not know whether $\dleq{b}$ also is trivial in case (ii) of \Cref{bounds on bleq}.

\begin{qst}
Let $D_\lambda=\st{\alpha\in\kappa\mid \cf(b(\alpha))=\lambda}$ be bounded for all regular $\lambda<\kappa$ and $S$ be stationary such that for each $\xi\in S$ there exists $\alpha_\xi\geq \xi$ with $\cf(b(\alpha_\xi))\leq\xi$. Is $\dleq{b}<2^\kappa$ consistent?
\end{qst}

To conclude this subsection, we will give an independence result to show that both $\bleq{b}$ and $\dleq{b}$ are nontrivial in case (iii). We prove this with a generalisation of Hechler forcing.

\begin{dfn}\label{Hechler general definition}
Let $C\subset\kappa$. We define \emph{$(b,C)$-Hechler forcing} $\bb D_\kappa^b(C)$ as the forcing notion with conditions $p=(s,f)$ where $s\in\prod_{<\kappa}b$ is such that $\dom(s)\in C$ and $f:\kappa\setminus \dom(s)\to \kappa$ is such that $f(\alpha)\in b(\alpha)$ for all $\alpha\in\dom(f)$, where the ordering is defined as $(t,g)\leq(s,f)$ iff $s\subset t$ and $f(\alpha)\leq g(\alpha)$ for all $\alpha\in\dom(g)$ and $f(\alpha)\leq t(\alpha)$ for all $\alpha\in\dom(t)\setminus\dom(s)$.
\end{dfn}
It is clear from the definition, under the assumptions of case (iii), that $\bb D_\kappa^b(C)$ adds a generic element of $\prod b$ that dominates all functions in $\prod b$ from the ground model. We will show that $\bb D_\kappa^b(C)$ preserves cardinals and cofinalities. Since $\bb D_\kappa^b(C)$ is not ${<}\kappa$-closed in general, we need a slightly weaker variant of closure.

\begin{dfn}
A forcing notion $\bb P$ is \emph{strategically ${<}\kappa$-closed} if White has a winning strategy for $\c G(\bb P,p)$ for all $p\in\bb P$. Here $\c G(\bb P,p)$ is the game of length $\kappa$ as follows: at stage $\alpha\in\kappa$, White chooses a condition $p_\alpha$ such that $p_\alpha\leq p'_\xi$ for all $\xi<\alpha$ (and $p_0\leq p$) and Black subsequently chooses a condition $p'_\alpha\leq p_\alpha$. White wins $\c G(\bb P,p)$ if White can make a move at every stage $\alpha\in\kappa$ of the game.
\end{dfn}

\begin{lmm}
Let $C$ be a club set such that for each $\xi\in C$ we have $\cf(b(\alpha))>\xi$ for all $\alpha\geq\xi$, then $\bb D_\kappa^b(C)$ is strategically ${<}\kappa$-closed and ${<}\kappa^+$-c.c.
\end{lmm}
\begin{proof}
For strategic ${<}\kappa$-closure, at stage $\alpha$ of the game $\c G(\bb D_\kappa^b(C),p)$, let $\sab{p_\xi\mid \xi<\alpha}$ and $\sab{p'_\xi\mid \xi<\alpha}$ be the sequences of previous moves by White and Black respectively. The winning strategy for White will be to choose $(s_\alpha,f_\alpha)$ such that $\dom(s_\alpha)\geq\alpha$ in successor stages. Under this strategy, if $\alpha$ is a limit ordinal, then we have $\xi\leq\dom(s_\xi)\in C$ for each $\xi\in\alpha$. It follows that $\alpha\leq\dom(s_\alpha)$ and hence $\cf(b(\beta))>\alpha$ for all $\beta\geq\dom(s_\alpha)$ by the properties of $C$. Therefore $\ab{f_\xi(\beta)\mid \xi\in\alpha}$ is not cofinal in $b(\beta)$, thus we can define $f_\alpha:\beta\mapsto \Cup_{\xi\in\alpha}f_\xi(\beta)$ for each $\beta\in\kappa\setminus\dom(s_\alpha)$.

For ${<}\kappa^+$-c.c., note that for any $(s,f),(s,g)\in\bb D_\kappa^b(C)$ we can choose $h(\alpha)=\max\st{f(\alpha),g(\alpha)}$ for all $\alpha\in\kappa\setminus\dom(s)$ to see that have $(s,h)\leq (s,f)$ and $(s,h)\leq (s,g)$. Thus, if $\c A\subset\bb D_\kappa^b(C)$ is an antichain and $(s,f),(t,g)\in \c A$ are distinct, then we must have $s\neq t$, hence $|\c A|\leq \kappa$.
\end{proof}

\begin{crl}
$\bb D^b_\kappa(C)$ preserves cardinals and cofinality and does not add any elements to $\gfbs$.
\end{crl}

\begin{fct}
Iterations of arbitrary length with ${<}\kappa$-support preserve strategic ${<}\kappa$-closure and ${<}\kappa^+$-c.c., and hence preserve cardinals, cofinality and do not add any elements to $\gfbs$.
\end{fct}

This implies especially that iteration will not destroy the inaccessibility of $\kappa$.

\begin{thm}\label{Hechler consistency proof}
Let $C$ be a club set such that for each $\xi\in C$ we have $\cf(b(\alpha))>\xi$ for all $\alpha\geq\xi$, then it is consistent that $\bleq b>\kappa^+$ and that $\dleq b<2^\kappa$.
\end{thm}
\begin{proof}
Let $\bb P=\sab{\bb P_\alpha,\dot{\bb Q}_\alpha\mid \alpha\in\kappa^{++}+\kappa^+}$ be a ${<}\kappa$-support iteration with $\bb P_\alpha\fc\ap{\dot{\bb Q}_\alpha=\dot{\bb D}_\kappa^b(C)}$ for each $\alpha$ and let $\b V\md\ap{\s{GCH}}$. If $B\subset(\prod b)^{\b V^{\bb P_{\kappa^{++}}}}$ is such that $|B|\leq \kappa^+$, then $B\in\b V^{\bb P_\alpha}$ for some $\alpha<\kappa^{++}$, since the $\kappa^{++}$-th stage of the iteration does not add any $\kappa$-reals. The generic $\bb D_\kappa^b(C)$-real that is added in the $\alpha+1$-th stage then dominates all elements of $B$, hence $B$ is not unbounded. This shows that $\b V^{\bb P_{\kappa^{++}}}\md\ap{\bleq b>\kappa^+}$.

Next, let $D$ consist of the $\kappa$-reals that are added in the final $\kappa^+$ stages of the iteration, then $D$ clearly forms a dominating family of size $\kappa^+$ in $\b V^{\bb P}$, and it is easy to see that $\b V^{\bb P}\md\ap{2^\kappa=\kappa^{++}}$, showing that $\b V^{\bb P}\md\ap{\dleq b<2^\kappa}$.
\end{proof}

\subsection{Localisation \& Avoidance}

With localisation and avoidance cardinals, we have not only the parameter $b$, but also the parameter $h$ giving the width of our slaloms. As with the dominating and unbounding numbers, there are certain choices of these parameters for which we have trivial cardinal characteristics.

We will motivate our assumption that $h\leq b$ with the following two lemmas:
\begin{lmm}\label{totally trivial loc}
$b<^* h$ if and only if $\dstar{b,h}=1$.\\
If $b<^* h$, then $\bstar{b,h}$ is undefined.
\end{lmm}
\begin{proof}
If $b<^*h$, let $B=\st{\alpha\in\kappa\mid b(\alpha)<h(\alpha)}$, and choose some $\phi\in\Loc_\kappa^{b,h}$ such that $\phi(\alpha)=b(\alpha)$ for all $\alpha\in B$. Since almost all $\alpha\in \kappa$ are in $B$, we see that $f\in\prod b$ implies $f\in^*\phi$. Hence $\phi$ localises the entirety of $\prod b$, making $\dstar{b,h}=1$ and $\bstar{b,h}$ undefined.

Reversely, if there is a strictly increasing sequence $\ab{\alpha_\xi\mid \xi\in\kappa}$ with $h(\alpha_\xi)\leq b(\alpha_\xi)$ for all $\xi$, and $\phi\in\Loc_\kappa^{b,h}$, then there is $\gamma_\xi\in b(\alpha_\xi)\setminus \phi(\alpha_\xi)$ for each $\xi$ since $|\phi(\alpha_\xi)|<h(\alpha_\xi)\leq b(\alpha_\xi)$. Therefore, if $f\in\prod b$ is such that $f(\alpha_\xi)=\gamma_\xi$ for all $\xi\in\kappa$, then $f\nins\phi$.
\end{proof}

We will state two essential properties relating two relational systems of the form $\sr L_{b,h}$ with different parameters to each other. The first deals with cofinal subsets of indices (cf. \Cref{subsequence domination}) and the second shows monotonicity with regards to $\leq^*$. The proofs are elementary.
\begin{lmm}\label{loc cofinal subset}\renewcommand{\qed}{\hfill$\square$}
If $\ab{\alpha_\xi\mid \xi\in\kappa}\in\gbs$ is a strictly increasing sequence and $b':\xi\mapsto b(\alpha_\xi)$ and $h':\xi\mapsto h(\alpha_\xi)$, then $\sr L_{b',h'}\preceq\sr L_{b,h}$.
\end{lmm}
\begin{lmm}\label{monotonicity loc}\renewcommand{\qed}{\hfill$\square$}
Let $h\leq^* h'$ and $b\geq^*b'$, then there exist Tukey connections ${\sr L_{b',h'}}\preceq {\sr L_{b,h}}$.
\end{lmm}

Mirroring the situation of \Cref{bounds on bleq}, we will give a complete characterisation of the cases in which $\bstar{b,h}$ is trivial.

\begin{lmm}\label{bounds on bstar}
For each $\lambda<\kappa$ define $D_\lambda=\st{\alpha\in\kappa\mid h(\alpha)=\lambda}$.
\begin{enumerate}[label=(\roman*)]
\item If there exists a least cardinal $\lambda<\kappa$ such that $D_\lambda$ is cofinal in $\kappa$, then $\bstar{b,h}=\lambda$. 
\item If $D_\lambda$ is bounded for all $\lambda<\kappa$ and $h$ is continuous on a stationary set, then $\bstar{b,h}=\kappa$, 
\item If $D_\lambda$ is bounded for all $\lambda<\kappa$ and  $h$ is discontinuous on a club set, then $\kappa^+\leq\bstar{b,h}$.\qedhere
\end{enumerate}
\end{lmm}
\begin{proof}
(i)\quad Assume that $\lambda$ is minimal such that $D_\lambda$ is cofinal. For each $\eta<\lambda$, define $f_\eta\in\prod b$ elementwise by $f_\eta:\alpha\mapsto \eta$ for all $\alpha\in D_\lambda$ and $f_\eta:\alpha\mapsto 0$ otherwise. If $\phi\in\Loc_\kappa^{b,h}$ and $\alpha\in D_\lambda$, then $|\phi(\alpha)|<h(\alpha)=\lambda$, so there exists $\eta<\lambda$ for which $\eta\notin \phi(\alpha)$, hence $f_\eta\nins\phi$. Therefore $\c F=\st{f_\eta\mid \eta<\lambda}$ witnesses that $\bstar{b,h}\leq |\c F|=\lambda$. 

On the other hand, if $\c F\subset \prod b$ and $|\c F|<\lambda$, then by minimality of $\lambda$ we see that $|\c F|<h(\alpha)$ for almost all $\alpha\in\kappa$, thus we can choose $\phi\in\Loc_\kappa^{b,h}$ such that $\phi:\alpha\mapsto \st{f(\alpha)\mid f\in \c F}$ whenever $|\c F|<h(\alpha)$, then we see that $f\in^*\phi$ for all $f\in \c F$, proving that $|\c F|<\bstar{b,h}$.

(ii)\quad Since each $D_\lambda$ is bounded, it follows that $\st{\alpha\in\kappa\mid h(\alpha)\leq\lambda}$ is bounded for every $\lambda<\kappa$. Let $\c F\subset \prod b$ with $|\c F|<\kappa$, and let $\beta\in\kappa$ be such that $h(\alpha)\leq |\c F|$ implies $\alpha<\beta$. If $\phi\in\Loc_\kappa^{b,h}$ and $\xi\geq \beta$, then $|\c F|<h(\xi)$, hence we can define $\phi(\xi)=\st{f(\xi)\mid f\in \c F}$ for all $\xi\geq \beta$. It is clear that $f\in^*\phi$ for all $f\in\c F$, thus $|\c F|<\bstar{b,h}$. Since $\c F$ was arbitrary such that $|\c F|<\kappa$, we see that $\kappa\leq\bstar{b,h}$.

To see $\kappa=\bstar{b,h}$, consider $\c F=\st{f_\eta\mid \eta\in\kappa}$ where $f_\eta:\alpha\mapsto \eta$ when $\eta\in b(\alpha)$, and $f_\eta:\alpha\mapsto 0$ otherwise. Since we assume $h\leq^*b$ and $h$ is cofinal, for every $\eta\in\kappa$ there is $\alpha_0$ such that $b(\alpha)>\eta$ for all $\alpha\geq\alpha_0$, thus $f_\eta(\alpha)=\eta$ for almost all $\alpha\in\kappa$. Since $h$ is continuous on a stationary set $S_0$, the set of fixed-points $S_1=\st{\alpha\in S_0\mid h(\alpha)=\alpha}$ is stationary, implying that $\alpha\mapsto |\phi(\alpha)|$ is regressive on the stationary $S_1$. Consequently Fodor's lemma tells us that there exists a stationary set $S_2\subset S_1$ and $\lambda\in\kappa$ such that $|\phi(\alpha)|<\lambda$ for all $\alpha\in S_2$. If $f_\eta\in^*\phi$ for all $\eta\in\lambda$, then $N_\eta=\st{\alpha\in \kappa\mid \eta\notin \phi(\alpha)}$ is nonstationary for each $\eta\in\lambda$, so $\Cup_{\eta\in\lambda}N_\eta$ is nonstationary, and thus $S_2\setminus \Cup_{\eta\in\lambda}N_\eta$ is stationary. However, for any $\alpha\in S_2\setminus \Cup_{\eta\in\lambda}N_\eta$ we have $\lambda\subset \phi(\alpha)$, which contradicts that $|\phi(\alpha)|<\lambda$ for $\alpha\in S_2$. Therefore $\c F$ forms a witness for $\bstar{b,h}\leq \kappa$.

(iii)\quad Let $\c F\subset \prod b$ with $|\c F|=\kappa$ and enumerate $\c F$  as $\ab{f_\eta\mid \eta\in\kappa}$. Let $C$ be a club set containing no successor ordinals such that $h$ is discontinuous on $C$, and let $\ab{\alpha_\xi\mid \xi\in\kappa}$ be the increasing enumeration of $C\cup \st0$. For each $\xi\in\kappa$ let $\lambda_\xi=\Cup_{\alpha\in\alpha_\xi}h(\alpha)$, where we have the convention that $\Cup\emp=0$, then $\lambda_\xi<h(\alpha_\xi)$ for all $\xi\in\kappa$ by discontinuity. Given $\alpha\in\kappa$, let $\xi$ be such that $\alpha\in[\alpha_\xi,\alpha_{\xi+1})$, which exists by $C$ being club, and let $\phi(\alpha)=\st{f_\eta(\alpha)\mid \eta\in\lambda_\xi}$. Since $h$ is increasing, $|\phi(\alpha)|\leq\lambda_\xi<h(\alpha_\xi)\leq h(\alpha)$, thus $\phi\in\Loc_\kappa^{b,h}$. Finally $h$ is cofinal, thus $\ab{\lambda_\xi\mid \xi\in\kappa}$ is cofinal, hence for every $\eta\in\kappa$ we have $f_\eta\in^*\phi$, showing that $\kappa<\bstar{b,h}$.
\end{proof}
We will once again see that $\bstar{b,h}$ is nontrivial in case (iii), but first let us consider $\dstar{b,h}$.

In case (i) we have $\dstar{b,h}=2^\kappa$. The proof is a generalisation based on Lemmas 1.8, 1.10 and 1.11 from \cite{GS93}, where the analogous theorem is proved for $\bs$.

\begin{lmm}\label{bounds on dstar}
Suppose that there exists some $\lambda<\kappa$ such that $D_\lambda=\st{\alpha\in\kappa\mid h(\alpha)=\lambda}$ is cofinal in $\kappa$. Then $\dstar{b,h}=2^\kappa$.
\end{lmm}
\begin{proof}
We are only interested in finding a lower bound of $\dstar{b,h}$, thus by \Cref{loc cofinal subset} we could restrict our attention to a cofinal subset of $\kappa$. We therefore assume without loss of generality that $D_\lambda=\kappa$, that is, $h(\alpha)=\lambda$ for all $\alpha\in\kappa$. To prove this lemma we assume furthermore without loss of generality that $b=h$. This suffices to prove the lemma, by \Cref{monotonicity loc}.

Let $b'$ be defined by $b'(\alpha)=2^{|\alpha|}$ for all $\alpha\in \kappa$. We start with proving that $\dstar{b',h}= 2^\kappa$. 

Let $\pi_\alpha:{}^\alpha2\inj b'(\alpha)$ be an injection for every $\alpha\in \kappa$. For some arbitrary $g\in\gcs$, define $f_g\in\prod b'$ to be such that $f_g(\alpha)=\pi_\alpha(g\restriction\alpha)$ for all $\alpha\in \kappa$. If $g,g'\in\gcs$ are distinct, then $g\restriction\alpha\neq g'\restriction \alpha$ for almost all $\alpha\in \kappa$, hence $f_g(\alpha)\neq f_{g'}(\alpha)$ for almost all $\alpha\in \kappa$. Therefore, if $\phi\in \Loc_\kappa^{b',h}$, then there are at most $\lambda$ many functions $g\in\gcs$ such that $f_g\in^*\phi$. Since $\gcs$ cannot be the union of less than $2^\kappa$ sets of size $\lambda$, we see that $\dstar{b',h}=2^\kappa$.

We will construct a Tukey connection $\sr{L}_{b',h}\preceq\sr{L}_{b,h}$. Let $\ab{W_\alpha\mid \alpha\in\kappa}$ be a partition of $\kappa$ such that $|W_\alpha|=\lambda^{2^{|\alpha|}}$ and let $\ab{\Phi_\xi^\alpha\mid \xi\in W_\alpha}$ be an enumeration of all functions $2^{|\alpha|}\to \lambda$. We let $\rho_-:\prod b'\to\prod b$ send a function $f'$ to the function $f$ defined as follows: for $\xi\in\kappa$, let $\alpha$ be such that $\xi\in W_\alpha$, then we let $f(\xi)=\Phi_\xi^\alpha(f'(\alpha))$. We let $\rho_+:\Loc_\kappa^{b,h}\to \Loc_\kappa^{b',h}$ send a slalom $\phi$ to the slalom $\phi'$, where $\phi'(\alpha)=\st{\eta\in b'(\alpha)\mid \forall \xi\in W_\alpha(\Phi_\xi^\alpha(\eta)\in \phi(\xi))}$. We show that $\phi'$ is indeed a $(b',h)$-slalom by proving that $|\phi'(\alpha)|<h(\alpha)=\lambda$.

Assume towards contradiction that there exists a sequence $\ab{\eta_\beta\mid \beta\in\lambda}$ of distinct elements of $\phi'(\alpha)$. We define a function $\Phi:b'(\alpha)\to \lambda$ by sending $\eta_\beta\mapsto \beta$ and $\eta\mapsto 0$ if $\eta\neq\eta_\beta$ for all $\beta\in\lambda$. Note that $\dom(\Phi)=b'(\alpha)=2^{|\alpha|}$, hence there is $\xi\in W_\alpha$ such that $\Phi=\Phi_\xi^\alpha$. Since $|\phi(\xi)|<h(\xi)=\lambda$ there is some $\beta\in \lambda\setminus \phi(\xi)$, but then $\Phi(\eta_{\beta})=\Phi_\xi^\alpha(\eta_{\beta})\notin \phi(\xi)$, which implies the contradictory $\eta_\beta\notin \phi'(\alpha)$.

Finally we have to prove that $\rho_-$, $\rho_+$ form a Tukey connection. If we assume that $f\in^*\phi$, then $\Phi_\xi^\alpha(f'(\alpha))\in \phi(\xi)$ for almost all $\xi\in\kappa$, where $\alpha$ is such that $\xi\in W_\alpha$. This means that for almost all $\alpha\in\kappa$ we have $\Phi_\xi^\alpha(f'(\alpha))\in \phi(\xi)$ for all $\xi\in W_\alpha$. Therefore, for almost all $\alpha\in\kappa$ we have $f'(\alpha)\in \phi'(\alpha)$, showing that $f'\in^*\phi'$.
\end{proof}

The situation in case (ii) appears to be more complicated, and currently our best lower bound is given by the maximal size of almost disjoint families of functions. A family $\c A\subset \prod b$ is called \emph{almost disjoint} if $f=^\infty g$ implies that $f=g$ for all $f,g\in\c A$.
\begin{lmm}
If $D_\lambda=\st{\alpha\in\kappa\mid h(\alpha)=\lambda}$ is bounded for all $\lambda\in\kappa$ and $h$ is increasing and continuous on a stationary set and $\c A\subset \prod b$ is an almost disjoint family, then $|\c A|\leq \dstar{b,h}$.
\end{lmm}
\begin{proof}
Since we assume $h\leq b$ and $h$ is cofinal, $b$ is also cofinal. Note that there exists an almost disjoint family $\c A\subset\prod b$ with $|\c A|=\kappa$: let $f_\eta\in\prod b$ send $\alpha\mapsto \eta$ if $\eta\in b(\alpha)$ and $\alpha\mapsto 0$ otherwise, then $\c A=\st{f_\eta\mid \eta\in\kappa}$ suffices. We will therefore assume without loss of generality that $\c A$ is an almost disjoint family with $|\c A|\geq \kappa$.

Given $\phi\in\Loc_\kappa^{b,h}$, let $\lambda_\phi$ be minimal such that there is stationary $S_\phi$ with $|\phi(\alpha)|=\lambda_\phi$ for all $\alpha\in S_\phi$. Fix some arbitrary $A\subset \c A$ with $|A|=\lambda_\phi^+$, enumerate $A$ as $\ab{f_\alpha\mid \alpha\in\lambda_\phi^+}$ and define $\xi_{\alpha,\beta}=\min\st{\xi\in\kappa\mid\forall \eta\in[\xi,\kappa)(f_\alpha(\eta)\neq f_\beta(\eta))}$. Since $\xi_{\alpha,\beta}<\kappa$ for all distinct $\alpha,\beta\in\lambda_\phi^+$, and $\lambda_\phi^+<\kappa$ and $\kappa$ is inaccessible, we see that $\xi=\Cup_{\alpha\in\lambda_\phi^+}\Cup_{\beta\in\alpha}\xi_{\alpha,\beta}\in\kappa$. If $\eta\geq \xi$, then $f_\alpha(\eta)$ is distinct for each $\alpha\in\lambda_\phi^+$, thus $\card{\st{f_\alpha(\eta)\mid \alpha\in\lambda_\phi^+}}=\lambda_\phi^+$. For every $\eta\in S_\phi\setminus \xi$ there is $\alpha\in\lambda_\phi^+$ such that $f_\alpha(\eta)\notin\phi(\eta)$, hence by the pigeonhole principle there exists  $\alpha\in\lambda_\phi^+$ such that $f_\alpha\nins\phi$.

Thus, we see that $\c A_\phi=\st{f\in\c A\mid f\in^*\phi}$ has $|\c A_\phi|\leq \lambda_\phi$. If $\Phi\subset\Loc_\kappa^{b,h}$ is of minimal cardinality to witnesses $\dstar{b,h}$, then $\Cup_{\phi\in\Phi} \c A_\phi=\c A$, and thus $\kappa\leq |\c A|\leq|\Phi|\cdot\sup_{\phi\in\Phi}\lambda_\phi=|\Phi|$. For the last equality, note that $\sup_{\phi\in\Phi}\lambda_\phi\leq\kappa$ and that this inequality would be strict if $|\Phi|<\kappa$.
\end{proof}

Note that an almost disjoint family $\c A\subset \prod b$ with $|\c A|=2^\kappa$ exists when there exists a continuous strictly increasing sequence $\ab{\alpha_\xi\mid \xi\in\kappa}$ such that $2^{|\xi|}\leq|b(\alpha_\xi)|$ for all $\xi\in\kappa$. The construction for such a family is done by fixing bijections $\pi_\xi:{}^\xi2\bij b(\alpha_\xi)$ and for any $f\in\gcs$ considering the function $f'\in\prod b$ given by $f':\alpha\mapsto \pi_\xi(f\restriction\xi)$ for each $\alpha\in[\alpha_\xi,\alpha_{\xi+1})$, then $\st{f'\mid f\in\gcs}$ forms an almost disjoint family. 

Even if $b$ does not grow fast enough such that the above construction can be done, it is still possible to add an almost disjoint family of size $2^\kappa$ with forcing. Let us focus on the case where $b=\id:\alpha\mapsto\alpha$ is the identity function, and look at a forcing notion $\AD$ that adds a $\lambda$-sized almost disjoint family of regressive functions, that is, elements of $\prod\id$. 

\begin{dfn}
The forcing notion $\AD$ has the conditions  $p:X_p\times\beta_p\to\kappa$ such that $X_p\in[\lambda]^{<\kappa}$, $|X_p|\leq\beta_p\in\kappa$, and $p(\xi,\alpha)<\alpha$ for all $\xi\in X_p$ and $0<\alpha\in\beta_p$. Given $\xi\in X_p$, we write $p_\xi:\beta_p\to\kappa$ for the (regressive) function $p_\xi:\alpha\mapsto p(\xi,\alpha)$. The ordering on $\AD$ is given by $q\leq p$ iff  $p\subset q$ (implicitly $X_p\subset X_q$ and $\beta_p\leq \beta_q$), and $q_\xi(\alpha)\neq q_{\xi'}(\alpha)$ for any $\alpha\in\beta_q\setminus\beta_p$ and distinct $\xi,\xi'\in X_p$.
\end{dfn}

\begin{lmm}\label{AD: extension}
The set of all $p\in\AD$ such that $(\xi,\alpha)\in \dom(p)$, is dense for any $\xi\in\lambda$ and $\alpha\in\kappa$.
\end{lmm}
\begin{proof}
Let $p\in\AD$ and $\beta_p\leq \alpha\in\kappa$, then we will first find $q\leq p$ with $X_q=X_p$ and $\beta_q=\alpha+1$. Fix an enumeration $\ab{\xi_\eta\mid \eta<|X_p|}$ of $X_p$, then for any $\gamma\in [\beta_p,\alpha]$ note that $|X_p|\leq\beta_p\leq\gamma$, hence we can define $q_{\xi_\eta}(\gamma)=\eta<\gamma$, then $q\leq p$.

Next we show how to increase $X_p$. Let $p\in\AD$ and $\mu<\kappa$, then by the above there exists $q\leq p$ with $X_q=X_p$ and $|X_p|+\mu\leq\beta_q$. If $X\in[\lambda\setminus X_p]^\mu$, we can find $r\leq q$ with $X_r=X_p\cup X$ and $\beta_r=\beta_q$ simply by letting $r_\xi(\alpha)=0$ for all $\alpha\in\beta_q$ and $\xi\in X$.
\end{proof}
\begin{lmm}\label{AD: generic}
If $G$ is an $\AD$-generic filter over $\b V$, and we define $f_\xi=\Cup\st{p_\xi\mid p\in G\la \xi\in X_p)}$, then $f_\xi\in\gbs$ is a regressive function and $\st{f_\xi\mid \xi\in(\lambda)^{\b V}}$ is almost disjoint.
\end{lmm}
\begin{proof}
It is clear from the definition of $\AD$ and the above lemma that $f_\xi\in\gbs$ and that $f_\xi$ is regressive. If $\xi,\xi'\in (\lambda)^{\b V}$ are distinct, then let $p\in G$ be such that $\xi,\xi'\in X_p$ and let $\alpha\geq \beta_p$. For any $q\leq p$ with $q\in G$ and $\alpha\in\beta_q$ we have $q_\xi(\alpha)\neq q_{\xi'}(\alpha)$. By the above lemma there exist such $q\leq p$ with $q\in G$, hence $f_\xi(\alpha)\neq f_{\xi'}(\alpha)$ for any $\alpha>\beta_p$. Therefore $f_\xi$ and $f_{\xi'}$ are almost disjoint.
\end{proof}

\begin{lmm}\label{AD: closure}
$\AD$ is ${<}\kappa$-closed and has the ${<}\kappa^+$-c.c..
\end{lmm}
\begin{proof}
First, let us prove ${<}\kappa$-closure. Let $\gamma\in\kappa$ and let $\ab{p^\eta\in\AD\mid \eta<\gamma}$ be a descending chain of conditions. It is clear that $X=\Cup_{\eta\in\gamma}X_{p^\eta}\in[\lambda]^{<\kappa}$ and $|X|\leq\beta=\Cup_{\eta\in\gamma}\beta_{p^\eta}\in \kappa$, and that $p=\Cup_{\eta\in\gamma}p^\eta:X\times\beta\to\kappa$ and $p(\xi,\eta)<\eta$ for all $(\xi,\eta)\in X\times\beta$. If $\xi,\xi'\in X_{p^\eta}$ are distinct and $\alpha\in \beta\setminus \beta_{p^\eta}$, then there is $\zeta>\eta$ such that $\alpha\in \beta_{p^{\zeta}}$. Since the chain of conditions is descending we have $p^{\zeta}\leq p^\eta$, which implies that $p_\xi(\alpha)=p^{\zeta}_\xi(\alpha)\neq p^\zeta_{\xi'}(\alpha)=p_{\xi'}(\alpha)$, so $p\leq p_\eta$.

As for ${<}\kappa^+$-c.c., suppose $\c B\subset \AD$ is a subset with $|\c B|=\kappa^+$. We let $\c A=\st{\dom(p)\mid p\in \c B}$, then by $|X_p|<\kappa$ and $\beta_p<\kappa$ for each $p\in \c B$ we see that $\c A$ is a family of sets of cardinality ${<}\kappa$. Applying the $\Delta$-system lemma on $\c A$, we know that there exists $\c B_0\subset \c B$ such that $|\c B_0|=\kappa^+$ and some sets $X\in[\lambda]^{<\kappa}$ and $\beta\in\kappa$ such that $\dom(p)\cap \dom(p')=X\times \beta$ for all distinct $p,p'\in \c B_0$. 

If $X=\emp$, then any $p,p'\in\c B_0$ have some $q\leq p$ and $q'\leq p'$ such that $|X_p|+|X_{p'}|\leq \beta_q=\beta_{q'}$ using \Cref{AD: extension}. Then $q\cup q'\leq p$ and $q\cup q'\leq p'$, thus $\c B$ is not an antichain.

On the other hand, if $X\neq \emp$, then $\beta=\beta_p=\beta_{p'}$ for all $p,p'\in\c B_0$. Since $|X\times \beta|<\kappa$ and $\kappa^{<\kappa}=\kappa$, there exists $\c B_1\subset\c B_0$ with $|\c B_1|=\kappa^+$ such that $p\restriction (X\times \beta)=p'\restriction (X\times \beta)$ for all $p,p'\in\c B_1$. Once again, using \Cref{AD: extension} we can see that $p,p'$ are compatible, and thus $\c B$ is not an antichain.
\end{proof}
\begin{crl}\label{AD: preservation of cardinals}
$\AD$ preserves cardinals and cofinality and does not add any elements to $\gfbs$.
\end{crl}

\begin{crl}
It is consistent that there exists an almost disjoint family $\c A\subset\prod\id$ of cardinality $2^\kappa=\lambda$ for any $\lambda$ with $\cf(\lambda)>\kappa$.
\end{crl}
\begin{proof}
Starting with a model where $\b V\md\ap{2^\kappa=\lambda}$, let $G$ be $\AD$-generic over $\b V$, and let $\c F=\st{f_\xi\mid \xi\in(2^\kappa)^{\b V}}$ be the almost disjoint family described in \Cref{AD: generic}. An argument by counting names shows that $(2^\kappa)^\b V=(2^\kappa)^{\b V[G]}$.
\end{proof}

Although the ${<}\kappa$-closure implies that $\AD$ preserves the inaccessibility of $\kappa$, the same cannot be said for stronger large cardinal assumptions on $\kappa$. Indeed, it is inconsistent that an almost disjoint family $\c A\subset\prod\id$ of size larger than $\kappa$ exists for measurable cardinals, as was pointed out to me by Jing Zhang.
\begin{thm}
If $\kappa$ is measurable, $b\in\gbs$ is continuous and $\c A\subset\prod b$ is almost disjoint, then $|\c A|\leq \kappa$. 
\end{thm}
\begin{proof}
Let $\c U\subset\c P(\kappa)$ be a ${<}\kappa$-complete nonprincipal normal ultrafilter on $\kappa$. Assume towards contradiction that $\c A\subset\prod b$ is an almost disjoint family with $|\c A|=\kappa^+$. Since $b$ is continuous, the set of fixed points of $b$ contains a club set, thus every $f\in \c A$ is regressive on a club set. Therefore there exists $X_f\in\c U$ such that $f$ is regressive on $X_f$, and since $\c U$ is normal there exists $Y_f\in\c U\restriction X_f$ and $\gamma_f\in\kappa$ such that $\ran(f\restriction Y_f)=\st{\gamma_f}$. By the pigeonhole principle there exists $\c A'\subset \c A$ and $\gamma\in\kappa$ with $|\c A'|=|\c A|$ such that $\gamma_f=\gamma$ for all $f\in \c A'$. Then for any distinct $f,f'\in \c A'$ we have $Y_f\cap Y_{f'}\in \c U$, contradicting that $f$ and $f'$ are almost disjoint.
\end{proof}

In the end, we do not know whether the conditions of (ii) imply that $\dstar{b,h}$ is trivial, that is, the following question remains open:
\begin{qst}
Is $\dstar{b,h}<2^\kappa$ consistent with $b$ continuous on a club set?
\end{qst}

Finally, in case (iii) we can prove that each of $\kappa^+<\bstar{b,h}$ and $\dstar{b,h}<2^\kappa$ are consistent using a generalisation of localisation forcing. This mirrors what we saw for domination and unboundedness in \Cref{Hechler consistency proof}.

\begin{dfn}
Let $C\subset\kappa$. We define \emph{$(b,h,C)$-Localisation forcing} $\Locf^{b,h}_\kappa(C)$ as the forcing notion with conditions $p=(s,F)$, where $s\in\Loc_{<\kappa}^{b,h}$ is a function with domain $\gamma\in C$ and $F\subset\prod b$ with $|F|<h(\gamma)$. The ordering is given by $(t,G)\leq(s,F)$ iff $s\subset t$ and $F\subset G$ and for all $\alpha\in\dom(t)\setminus\dom(s)$ and $f\in F$ we have $f(\alpha)\in t(\alpha)$.
\end{dfn}
It is clear from the definition, under the assumptions of case (iii), that $\Locf^{b,h}_\kappa(C)$ adds a generic element of $\Loc_\kappa^{b,h}$ that localises all elements of $\prod b$ from the ground model.

\begin{lmm}
Let $C$ be a club set such that $h$ is discontinuous on $C$, then $\Locf_\kappa^{b,h}(C)$ is strategically ${<}\kappa$-closed and ${<}\kappa^+$-c.c..
\end{lmm}
\begin{proof}
For strategic ${<}\kappa$-closure, at stage $\alpha$ of $\c G(\Locf_\kappa^{b,h}(C),p)$, let $\sab{p_\xi=(s_\xi,F_\xi)\mid \xi\in\alpha}$ and $\sab{p'_\xi=(s_\xi',F_\xi')\mid \xi\in\alpha}$ be the sequences of previous moves by White and Black respectively. We will describe the winning strategy for White at stage $\alpha$.

If $\alpha=\beta+1$ is successor, White will decide some $\gamma\in C$ such that $h(\dom(s'_\beta))<\gamma$ and $\alpha<\gamma$. Since $|F_\beta'|<h(\dom(s_\beta'))$ and $h$ is increasing, we can define $s_\alpha\supset s_\beta'$ by letting $s_\alpha(\eta)=\sst{f(\eta)\mid f\in F_\beta'}$ for each $\eta\in[\dom(s_\beta'),\gamma)$ and let $F_\alpha=F_\beta'$. 

We have to show that White can make a move at limit $\alpha$ as well. We claim that $s_\alpha=\Cup_{\xi\in\alpha}s_\xi$ and $F_\alpha=\Cup_{\xi\in\alpha}F_\xi$ works. Clearly $\dom(s_\alpha)\in C$ by $C$ being club, and 
\begin{align*}
|F_\alpha|=|\alpha|\cdot\textstyle\sup_{\xi\in\alpha}|F_\xi|\leq\Cup_{\xi\in\alpha}h(\dom(s_\xi))<h(\Cup_{\xi\in\alpha}\dom(s_\xi))=h(\dom(s_\alpha))
\end{align*}
Here the strict inequality follows from $h$ being discontinuous on $C$. Therefore, $(s_\alpha,F_\alpha)$ is indeed a valid move for White, implying that White has a winning strategy.

For ${<}\kappa^+$-c.c., note that for any $(s,F),(s,G)\in\Locf_\kappa^{b,h}(C)$ with $\dom(s)=\gamma$ we have $|F|< h(\gamma)$ and $|G|< h(\gamma)$, hence since $h(\gamma)$ is infinite also $|F\cup G|<h(\gamma)$. Therefore $(s,F\cup G)$ is a condition below both $(s,F)$ and $(s,G)$. Thus if $\c A\subset\Locf_\kappa^{b,h}(C)$ is an antichain and $(s,F),(t,G)\in\c A$ are distinct, then we must have $s\neq t$, hence $|\c A|\leq\kappa$.
\end{proof}

\begin{crl}
$\Locf^{b,h}_\kappa(C)$ preserves cardinals and cofinality and does not any elements to $\gfbs$.
\end{crl}

\begin{thm}\label{Localisation consistency proof}
Let $C$ be a club set such that $h$ is discontinuous on $C$, then it is consistent that $\bstar{ b,h}>\kappa^+$ and that $\dstar{b,h}<2^\kappa$.
\end{thm}
\begin{proof}
Analogous to \Cref{Hechler consistency proof}, with $\Locf_\kappa^{b,h}(C)$ taking the role of $\bb D_\kappa^b(C)$.
\end{proof}

\subsection{Antilocalisation \& Anti-avoidance}

We can give similar results for $\dinf{b,h}$ and $\binf{b,h}$ to what we discussed in the previous sections. Firstly, we make three observations, without proof, analogous to \Cref{totally trivial loc,loc cofinal subset,monotonicity loc}. As a consequence we will once again assume that $h\leq b$ is always the case.

\begin{lmm}\label{totally trivial aloc}\renewcommand{\qed}{\hfill$\square$}
$b<^\infty h$ if and only if $\dinf{b,h}=1$.\\
If $b<^\infty h$, then $\binf{b,h}$ is undefined.
\end{lmm}
\begin{lmm}\label{aloc cofinal subset}\renewcommand{\qed}{\hfill$\square$}
If $\ab{\alpha_\xi\mid \xi\in\kappa}\in\gbs$ is a strictly increasing sequence and $b':\xi\mapsto b(\alpha_\xi)$ and $h':\xi\mapsto h(\alpha_\xi)$, then $\AL_{b',h'}\preceq\AL_{b,h}$.
\end{lmm}
\begin{lmm}\label{monotonicity aloc}\renewcommand{\qed}{\hfill$\square$}
Let $h\geq^* h'$ and $b\leq^*b'$, then there exist Tukey connections ${\AL_{b',h'}}\preceq {\AL_{b,h}}$.
\end{lmm}

A comprehensive overview of the trivial values for these cardinals in the classical context has been given by Cardona and Mej\'ia in section 3 of \cite{CM19}. However, the classical characterisation uses a substantial amount of finite arithmetic and thus appears quite different from the characterisation of the trivial values of $\binf{b,h}$, given below. We once again have three cases, making this similar to \Cref{bounds on bleq,bounds on bstar}. 

\begin{lmm}\label{bounds on binf}
For each $\lambda<\kappa$ define
\begin{align*}
D_\lambda=\st{\alpha\in\kappa\mid b(\alpha)\leq\lambda}\ \cup\ \st{\alpha\in\kappa\mid h(\alpha)=b(\alpha)\text{ and }\cf(b(\alpha))\leq \lambda}.
\end{align*}
\begin{enumerate}[label=(\roman*)]
\item If there exists a least cardinal $\lambda<\kappa$ such that $D_\lambda$ is cofinal in $\kappa$, then $\binf{b,h}=\lambda$.
\item If $D_\lambda$ is bounded for all $\lambda<\kappa$ and there is a stationary set $S$ such that 
\begin{enumerate}
\item $b$ is continuous on $S$, \emph{or} 
\item for each $\xi\in S$ there exists $\alpha_\xi\geq\xi$ with $h(\alpha_\xi)=b(\alpha_\xi)$ and $\cf(b(\alpha_\xi))\leq \xi$, 
\end{enumerate}
then $\binf{b,h}=\kappa$.
\item If $D_\lambda$ is bounded for all $\lambda<\kappa$ and there is a club set $C$ such that 
\begin{enumerate}
\item $b$ is discontinuous on C, \emph{and} 
\item for each $\xi\in C$ and $\alpha\geq \xi$, if $h(\alpha)=b(\alpha)$, then $\cf(b(\alpha))>\xi$, 
\end{enumerate}
then $\kappa^+\leq\binf{b,h}$.\qedhere
\end{enumerate}
\end{lmm}
\begin{proof}
(i)\quad Assume that $\lambda$ is minimal such that $D_\lambda$ is cofinal. Let $D_\lambda'=\st{\alpha\in \kappa\mid b(\alpha)=\lambda}$ and $D_\lambda''=\st{\alpha\in\kappa\mid h(\alpha)=b(\alpha)\text{ and }\cf(b(\alpha))=\lambda}$. Note that at least one of $D_\lambda'$ or $D_\lambda''$ is cofinal. We will assume without loss of generality that $D_\lambda=D'_\lambda\cup D''_\lambda$.

First, we  will show that $\binf{b,h}\leq \lambda$.

If $\alpha\in D_\lambda''$, then $\cf(b(\alpha))= \lambda$, thus we can find a continuous sequence $\langle\beta_\xi^\alpha\mid\xi<\lambda\rangle$ that is cofinal in $b(\alpha)$. In the other case where $\alpha\in D_\lambda'$, let $\beta_\xi^\alpha=\xi$ for all $\xi<\lambda=b(\alpha)$. 

Since $|[\beta_\xi^\alpha,\beta_{\xi+1}^\alpha)|<h(\alpha)\leq b(\alpha)$ for all $\alpha\in D_\lambda$, we can pick some $\phi_\xi\in\Loc_\kappa^{b,h}$ for each $\xi<\lambda$ such that $\phi_\xi(\alpha)=[\beta_\xi^\alpha,\beta_{\xi+1}^\alpha)$ for all $\alpha\in D_\lambda$. For $f\in \prod b$, let $F_f:D_\lambda\to \lambda$ map $\alpha\in D_\lambda$ to the $\xi\in\lambda$ such that $f(\alpha)\in\phi_\xi(\alpha)$, then by the pigeonhole principle $F_f^{-1}(\xi)$ is cofinal for some $\xi\in \lambda$, and thus for this $\xi$ we have $f\in^\infty \phi_\xi$. Therefore $\st{\phi_\xi\mid \xi<\lambda}$ witnesses that $\binf{b,h}\leq \lambda$.

Next we show that $\lambda\leq \binf{b,h}$. Let $\st{\phi_\xi\mid \xi\in \mu}\subset \Loc_\kappa^{b,h}$ for some $\mu<\lambda$.

If $D_\lambda'$ is cofinal, then we must have $b(\alpha)\geq \lambda$ for almost all $\alpha\in \kappa$ by minimality of $\lambda$. Suppose that $\lambda\leq b(\alpha)$ and $\Cup_{\xi<\mu}\phi_\xi(\alpha)=b(\alpha)$, then $\cf(b(\alpha))\leq\mu$, and furthermore for every $\nu<b(\alpha)$ there is some $\xi\in \mu$ such that $\nu\leq |\phi_\xi(\alpha)|$, implying that $h(\alpha)=b(\alpha)$. But then $\alpha\in D''_\mu$, which is bounded by minimality of $\lambda$. Therefore $\Cup_{\xi<\mu}\phi(\alpha)\neq b(\alpha)$ for almost all $\alpha\in\kappa$.

If $D_\lambda'$ is bounded and $D_\lambda''$ is cofinal, then for almost all $\alpha\notin D_\lambda''$ we must have $h(\alpha)\neq b(\alpha)$ or $\cf(b(\alpha))>\lambda$. Hence we have $\max\st{h(\alpha),\lambda}< b(\alpha)$ or $\cf(b(\alpha))\geq \lambda$ for almost all $\alpha\in\kappa$. If $\max\st{h(\alpha),\lambda}<b(\alpha)$, then also $\mu\cdot h(\alpha)<b(\alpha)$, thus $\Cup_{\xi<\mu}\phi_\xi(\alpha)\neq b(\alpha)$. On the other hand if $\cf(b(\alpha))\geq\lambda$, then clearly $\Cup_{\xi<\mu}\phi_\xi(\alpha)\neq b(\alpha)$ as well.

In either case we find $f\in\prod b$ such that $f\nini \phi_\xi$ for all $\xi<\mu$, and consequently $\mu<\binf{b,h}$.

(ii)\quad  We first prove that $\binf{b,h}\leq\kappa$ in case (a), then in case (b), and then we show the reverse direction, that $\kappa\leq\binf{b,h}$.

(a)\quad Assume $b$ is increasing and continuous on stationary $S\subset\kappa$. Then the set of fixed points $S'=\st{\alpha\in S\mid b(\alpha)=\alpha}$ is also stationary.

For each $\eta\in\kappa$ let $\phi_\eta\in\Loc_\kappa^{b,h}$ be defined as $\phi_\eta(\xi)=\st\eta$ if $\eta< b(\xi)$ and arbitrary otherwise. If $f\in\prod b$, then $f(\xi)<b(\xi)$ for every $\xi\in\kappa$, therefore $f$ is regressive on $S'$. By Fodor's lemma then there exists stationary $S''\subset S'$ such that $f\restriction S''$ is constant. Let $\eta$ be such that $f(\xi)=\eta$ for all $\eta\in S''$, then clearly $f\in^\infty \phi_\eta$. Hence $\st{\phi_\eta\mid \eta<\kappa}$ witnesses $\binf{b,h}\leq\kappa$.

(b)\quad Assume $S$ is stationary and for each $\xi\in S$ let $\alpha_\xi\geq \xi$ be such that $h(\alpha_\xi)=b(\alpha_\xi)$ and $\cf(b(\alpha_\xi))\leq\xi$. We define for each $\eta\in\kappa$ a slalom $\phi_\eta\in\Loc_\kappa^{b,h}$ in such a way that for each $\xi\in S$ we have $\Cup_{\eta<\xi}\phi_\eta(\alpha_\xi)=b(\alpha_\xi)$, which is possible since $\cf(b(\alpha_\xi))\leq\xi$ and $h(\alpha_\xi)=b(\alpha_\xi)$. 

If $f\in\prod b$, then for each $\xi\in S$ we define $f'(\xi)$ to be the least $\eta<\xi$ such that $f(\alpha_\xi)\in\phi_\eta(\alpha_\xi)$. Now $f'$ is regressive on $S$, so by Fodor's lemma there exists stationary $S'\subset S$ such that $f'\restriction S'$ is constant, say with value $\eta$. Then $f(\alpha_\xi)\in\phi_\eta(\alpha_\xi)$ for each $\xi\in S'$, thus $f\ini\phi_\eta$. Hence $\st{\phi_\eta\mid \eta<\kappa}$ witnesses $\binf{b,h}\leq\kappa$.

Finally for the reverse direction, let $\lambda<\kappa$ and $\st{\phi_\xi\mid \xi<\lambda}\subset\Loc_\kappa^{b,h}$. Since $D_\lambda$ is bounded, we can find $\alpha_0\in\kappa$ such that for every $\alpha>\alpha_0$ we have $\lambda<b(\alpha)$ and either $h(\alpha)<b(\alpha)$ or $\lambda<\cf(b(\alpha))$. If $h(\alpha)<b(\alpha)$, then $\lambda\cdot h(\alpha)<b(\alpha)$, meaning $\Cup_{\xi<\lambda}\phi_\xi(\alpha)\neq b(\alpha)$. On the other hand, if $\lambda<\cf(b(\alpha))$, then by $|\phi_\xi(\alpha)|<h(\alpha)\leq b(\alpha)$ we see that once again $\Cup_{\xi<\lambda}\phi_\xi(\alpha)\neq b(\alpha)$. Hence we can construct $f\in\prod b$ such that $f\nini\phi_\xi$ for all $\xi<\lambda$, which proves that $\kappa\leq\binf{b,h}$.

(iii)\quad Let $C$ be a club set with the properties mentioned in (iii) and let $\st{\phi_\eta\mid \eta\in\kappa}\subset\Loc_\kappa^{b,h}$. We can enumerate $C$ increasingly as $\ab{\alpha_\xi\mid \xi\in\kappa}$, then every $\alpha\in\kappa$ has some $\xi\in\kappa$ such that $\alpha\in[\alpha_\xi,\alpha_{\xi+1})$. Note that $|\Cup_{\eta\in\xi}\phi_\eta(\alpha)|\leq|\xi|\cdot\sup_{\eta\in\xi}|\phi_\eta(\alpha)|$. Since $b$ is increasing and discontinuous on $C$ and $\xi\leq \alpha_\xi\leq \alpha$, we see that $\xi<b(\alpha_\xi)\leq b(\alpha)$. If $h(\alpha)<b(\alpha)$, then it is clear that $\sup_{\eta\in\xi}|\phi_\eta(\alpha)|<b(\alpha)$ as well. Else, if $h(\alpha)=b(\alpha)$, then by the properties of $C$ we see that $\cf(b(\alpha))>\alpha_\xi\geq\xi$, thus it also follows that $\sup_{\eta\in\xi}|\phi_\eta(\alpha)|<b(\alpha)$. We can therefore conclude that $|\Cup_{\eta\in\xi}\phi_\eta(\alpha)|<b(\alpha)$ for all $\alpha\in\kappa$, and thus we can define $f(\alpha)$ to be some value in $b(\alpha)$ that lies outside of $\phi_\eta(\alpha)$ for each $\eta<\xi$. Then clearly $f\nini \phi_
\eta$ for all $\eta\in\kappa$.
\end{proof}

As with domination and localisation, it is not known whether cases (i) and (ii) of \Cref{bounds on binf} imply that $\dinf{b,h}=2^\kappa$. If we consider the constant function $\bar 2$, then we see that $\dinf{b,\bar 2}\leq\dinf{b,h}$ by \Cref{monotonicity aloc}. Since $\dinf{b,\bar 2}=\dneq{b}$ (see \Cref{evt difference and antiloc}), we will discuss the cases (i) and (ii) in the next section. Let us mention here, that if, by exception, we allow that $b(\alpha)$ is finite for cofinally many $\alpha$, then we may use a connection with localisation cardinals to prove that $\dinf{b,h}=2^\kappa$ (see \Cref{bounds on dneq}). We will also give a partial answer to case (i) in the form of the following lemma:
\begin{lmm}
If $\lambda$ is regular and  $D=\st{\alpha\in\kappa\mid b(\alpha)=h(\alpha)=\lambda}$ is cofinal, then $\sr D_{\bar \lambda}\preceq \AL_{b,h}$. 
\end{lmm}
\begin{proof}
Enumerate $D$ as $\st{\alpha_\xi\mid \xi\in \kappa}$. We need $\rho_-:{}^\kappa\lambda\to \Loc_\kappa^{b,h}$ and $\rho_+:\prod b\to {}^\kappa\lambda$.

Given $g\in{}^\kappa \lambda$, let $\rho_-(g)\in\Loc_\kappa^{b,h}$ be such that $\rho_-(g):\alpha_\xi\mapsto g(\xi)$ for all $\xi\in\kappa$. Given $f\in\prod b$, let $\rho_+(f):\xi\mapsto f(\alpha_\xi)$ for each $\xi\in\kappa$. Since $\alpha_\xi\in D$ for each $\xi\in\kappa$, we have $h(\alpha_\xi)=b(\alpha_\xi)=\lambda$, thus $\rho_-$ and $\rho_+$ are well-defined. If $\phi=\rho_-(g)$ and $f'=\rho_+(f)$, and $f\nini \phi$, then it is easy to see that $g\leq^* f'$, so this is a Tukey connection.
\end{proof}

Since we already know that $\dleq{\bar\lambda}=2^\kappa$ from \Cref{bounds on dleq}, the assumptions of the above lemma imply that $\dinf{b,h}=2^\kappa$.

\begin{dfn}
Let $C\subset\kappa$. We define \emph{$(b,h,C)$-antilocalisation forcing} $\ALocforc$ as the forcing with conditions $p=(s,\Phi)$ where $s\in\prod_{<\kappa}b$ is a function with domain $\gamma\in C$ and $\Phi\subset\Loc_\kappa^{b,h}$ with $|\Phi|\leq\gamma$. The ordering is given by $(t,\Psi)\leq(s,\Phi)$ if $s\subset t$ and $\Phi\subset\Psi$ and for each $\alpha\in\dom(t)\setminus\dom(s)$ and $\phi\in\Phi$ we have $t(\alpha)\notin\phi(\alpha)$.
\end{dfn}
It is clear from the definition, under the assumptions of case (iii), that $\ALocforc$ adds a generic function in $\prod b$ that is antilocalised by every ground model $(b,h)$-slalom.

\begin{lmm}
Let $C$ be a club set with such that the assumptions of (iii) from \Cref{bounds on binf} are satisfied. Then $\ALocforc$ is strategically ${<}\kappa$-closed and ${<}\kappa^+$-c.c..
\end{lmm}

\begin{proof}
For strategic ${<}\kappa$-closure, at stage $\alpha$ of $\c G(\ALocforc,p)$, let $\sab{p_\xi=(s_\xi,\Phi_\xi)\mid \xi\in\alpha}$ and $\sab{p'_\xi=(s_\xi',\Phi_\xi')\mid \xi\in\alpha}$ be the sequences of previous moves by White and Black respectively. We will describe the winning strategy for White at stage $\alpha$.

If $\alpha=\beta+1$ is successor, White will decide some $\gamma\in C$ such that $\alpha\leq\gamma$. Let $\dom(s_\beta')=\gamma'\in C$, then if $\xi\in[\gamma',\gamma)$ we have $|\Cup_{\phi\in\Phi_\beta'}\phi(\xi)|\leq|\gamma'|\cdot\sup_{\phi\in\Phi_\beta'}|\phi(\xi)|<b(\xi)$ through the same reasoning as in the proof of \Cref{bounds on binf} (iii). Therefore, we can find some value $s_\alpha(\xi)\in b(\xi)\setminus\Cup_{\phi\in\Phi_\beta'}\phi(\xi)$ to define $s_\alpha$ with $\dom(s_\alpha)=\gamma$, and let $\Phi_\alpha=\Phi_\beta'$. 

We have to show that White can make a move at limit $\alpha$ as well. We claim that $s_\alpha=\Cup_{\xi\in\alpha}s_\xi$ and $\Phi_\alpha=\Cup_{\xi\in\alpha}\Phi_\xi$ works. Clearly $\dom(s_\alpha)\in C$ by $C$ being club, and 
\begin{align*}
|\Phi_\alpha|\leq|\alpha|\cdot\textstyle\sup_{\xi\in\alpha}|\Phi_\xi|=|\alpha|\cdot\Cup_{\xi\in\alpha}\dom(s_\xi)\leq |\alpha|\leq\dom(s_\alpha)
\end{align*}
Therefore, $(s_\alpha,F_\alpha)$ is indeed a valid move for White, so White has a winning strategy.

For ${<}\kappa^+$-c.c., note that for any $(s,\Phi),(s,\Psi)\in\ALocforc$ with $\dom(s)=\gamma$ we have $|\Phi|\leq \gamma$ and $|\Psi|\leq\gamma$. We may assume without loss of generality that $\gamma$ is infinite, then also $|\Phi\cup\Psi|\leq\gamma$, hence $(s,\Phi\cup\Psi)$ is a condition below both $(s,\Phi)$ and $(s,\Psi)$. Thus if $\c A\subset\ALocforc$ is an antichain and $(s,\Phi),(t,\Psi)\in \c A$ are distinct, then we must have $s\neq t$, hence $|\c A|\leq \kappa$.
\end{proof}

\begin{crl}
$\ALocforc$ preserves cardinals and cofinality and does not add any elements to $\gfbs$.
\end{crl}

\begin{thm}\label{Antilocalisation consistency proof}
Let $C$ be a club set and $b,h\in\gbs$ be such that the assumptions of (iii) from \Cref{bounds on binf} are satisfied, then it is consistent that $\binf{b,h}>\kappa^+$ and that $\dinf{b,h}<2^\kappa$.
\end{thm}
\begin{proof}
Analogous to \Cref{Hechler consistency proof}, with $\ALocforc$ taking the role of $\bb D_\kappa^b(C)$.
\end{proof}

\subsection{Eventual Difference \& Cofinal Equality}

Eventual difference is a special case of antilocalisation, where the function $h$ that determines the size of the sets in the slaloms is as small as possible: 
\begin{lmm}\label{evt difference and antiloc}
$\ED_b\equiv \AL_{b,\bar 2}$.
\end{lmm}
\begin{proof}
Any slalom $\phi\in\Loc_\kappa^{b,\bar 2}$ is a sequence of singletons, and thus we can define $f_\phi\in\prod b$ such that $\st{f_\phi(\alpha)}=\phi(\alpha)$ for all $\alpha\in\kappa$. It follows that $g\ini\phi$ if and only if $g\eqi f_\phi$.
\end{proof}

It follows that \Cref{bounds on binf} immediately gives us the following characterisation of the trivial values of $\bneq{b}$, since we can ignore the case where $h(\alpha)= b(\alpha)$ if we let $h=\bar 2$.
\begin{crl}[cf. \Cref{bounds on binf}]\label{bounds on bneq}\renewcommand{\qed}{\hfill$\square$}
For each $\lambda<\kappa$ define $D_\lambda=\st{\alpha\in\kappa\mid b(\alpha)\leq\lambda}$.
\begin{enumerate}[label=(\roman*)]
\item If there exists a least cardinal $\lambda<\kappa$ such that $D_\lambda$ is cofinal in $\kappa$, then $\bneq{b}=\lambda$.
\item If $D_\lambda$ is bounded for all $\lambda<\kappa$ and $b$ is continuous on a stationary set, then $\bneq{b}=\kappa$.
\item If $D_\lambda$ is bounded for all $\lambda<\kappa$ and $b$ is discontinuous on a club set, then $\kappa^+\leq\bneq{b}$.\qedhere
\end{enumerate}
\end{crl}

It is equally immediate that the forcing construction from \Cref{Antilocalisation consistency proof} gives us a consistency result in case (iii) of \Cref{bounds on bneq}.
\begin{crl}[Cf. \Cref{Antilocalisation consistency proof}]\renewcommand{\qed}{\hfill$\square$}
Let $b\in\gbs$ be such that the assumptions of (iii) from \Cref{bounds on bneq} are satisfied, then it is consistent that $\bneq{b}>\kappa^+$ and that $\dneq{b}<2^\kappa$.
\end{crl}

We can use \Cref{bounds on dstar} and a special connection between localisation and antilocalisation for functions with finite values to prove the following lemma, giving a condition for which $\dneq{b}$ and $\dinf{b,h}$ are trivial. Note that this is significantly weaker than the assumption in case (i), since $b$ is not only bounded on a cofinal set, but even finite cofinally often.

\begin{lmm}\label{bounds on dneq}
If $D=\st{\alpha\in\kappa\mid b(\alpha)<\omega}$ is cofinal, then $\dneq{b}=\dinf{b,h}=2^\kappa$.
\end{lmm}
\begin{proof}
We want to give a lower bound to $\dinf{b,h}$, thus by \Cref{monotonicity aloc} we may assume that $h=\bar2$, and moreover by \Cref{aloc cofinal subset} and $D$ being cofinal, we may assume that $b=\bar n$ for some $n\in\omega$.

Note that $\sr L_{\bar n,\bar{n}}\equiv \AL_{\bar n,\bar 2}$. Namely, if $f\in\prod b$, let $\phi_f:\alpha\mapsto n\setminus\st{f(\alpha)}$, then $\phi_f\in\Loc_\kappa^{\bar n,\bar n}$, and on the other hand, if $\phi\in\Loc_\kappa^{\bar n,\bar2}$, then let $f_\phi$ be the unique function such that $\phi(\alpha)=\st{f_\phi(\alpha)}$ for all $\alpha\in\kappa$. It is clear that $f\nini\phi$ if and only if $f_\phi\ins\phi_f$.

It follows that $2^\kappa=\dstar{\bar n,\bar n}\leq \dneq{\bar n}\leq\dneq{b}=\dinf{b,\bar 2}\leq \dinf{b,h}$ if $D$ is cofinal. Here the first equality is given by \Cref{bounds on dstar}.
\end{proof}

The above argument is able to relate eventual difference to localisation by removing a single element from a finite set, which changes its cardinality. Naturally, removing a single element from an infinite set does not change the cardinality of the infinite set, thus the technique from this proof does not give us insight on how to prove $\dneq{b}=2^\kappa$ if $b$ is only bounded on a cofinal set by an infinite value.

It is equally unclear whether $\dneq b=2^\kappa$ in case (ii) of \Cref{bounds on bneq}.

\begin{qst}
Is $\dneq{b}<2^\kappa$ consistent in case (i) or (ii) of \Cref{bounds on bneq}?
\end{qst}

\section{More Relations Between Cardinals}\label{section: relations}

In this section we will give an overview of some additional relations between the cardinals we have mentioned so far.

\subsection{Relations for Fixed Parameters}

We do this by defining Tukey connections. The following theorem gives an overview of relations between the cardinals we have discussed so far, for a fixed pair of parameters $b,h$ such that the relevant cardinals are nontrivial.

\begin{thm}\label{cardinal relations fixed b h}
Let $b,h\in\gbs$ be such that the conditions of cases (iii) of \Cref{bounds on binf,bounds on bleq,bounds on bstar} are satisfied, then the following Tukey connections exist, where the relations marked by ${}^\dagger$ require the additional assumption that $h\leq^*\cf(b)$.

\hfill\begin{tikzpicture}[baseline=(ED.base),xscale=0.95, yscale=0.65]
\node[baseline=(D.base)] (AL) at (0,2) {$\strut\AL_{b,h}$};
\node[baseline=(D.base)] (D) at (2,2) {$\strut\sr D_{b}$};
\node[baseline=(D.base)] (L) at (4,2) {$\strut\sr L_{b,h}$};
\node[baseline=(D.base)] (ED) at (0,0) {$\strut\ED_{b}$};
\node[baseline=(D.base)] (B) at (2,0) {$\strut\sr D^\lc_{b}$};
\node[baseline=(D.base)] (AA) at (4,0) {$\strut\AL^\lc_{b,h}$};

\node[baseline=(D.base)] at (1.1,2) {$\strut\preceq^\dagger$};
\node[baseline=(D.base)] at (2.9,2) {$\strut\preceq^\dagger$};
\node[baseline=(D.base)] at (2.9,0) {$\strut{\preceq}^\dagger$};
\node[baseline=(D.base),rotate=90] at (0,1) {$\preceq$};
\node[baseline=(D.base),rotate=30] at (1,1) {$\preceq$};
\node[baseline=(D.base),rotate=90] at (2,1) {$\preceq$};
\node[baseline=(D.base),rotate=90] at (4,1) {$\preceq$};
\end{tikzpicture}\qedhere
\end{thm}
\begin{proof}
The relation between eventual difference and antilocalisation was already established with \Cref{evt difference and antiloc}, and combined with \Cref{monotonicity aloc} we get $\ED_b\preceq\AL_{b,h}$.

It is also easy to see for $f,g\in\prod b$, that if $g\leq^* f$, then $f+1\mathbin{\cancel{\leq^*}}g$ and $g\neqi f+1$. This implies that there are Tukey connections $\sr D^\lc_b\preceq \sr D_b$ and $\ED_b\preceq \sr D_b$.

Next, for $f\in\prod b$ and $\phi\in\Loc_\kappa^{b,h}$, we have that $f\in^*\phi$ implies $f\in^\infty\phi$, and thus we get a Tukey connection $\AL^\lc_{b,h}\preceq\sr L_{b,h}$.

Finally we will show that $\AL_{b,h}\preceq \sr D_b\preceq\sr L_{b,h}$ as long as we assume that $h\leq^*\cf(b)$. Note that this also implies $\AL^\lc_{b,h}\succeq\sr D_b^\lc$ by duality.

We define $\rho_-^0:\Loc_\kappa^{b,h}\to \prod b$ and $\rho_+^0:\prod b\to\prod b$ as a Tukey connection for $\AL_{b,h}\preceq\sr D_b$, and $\rho_-^1:\prod b\to\prod b$ and $\rho_+^1:\Loc_\kappa^{b,h}\to\prod b$ as a Tukey connection for $\sr D_b\preceq\sr L_{b,h}$. 

We will have $\rho_-^0=\rho_+^1$ sending $\phi\in\Loc_\kappa^{b,h}$ to $g\in\prod b$ where $g(\alpha)=\sup(\phi(\alpha))$ if $\sup(\phi(\alpha))<b(\alpha)$ and arbitrary otherwise. Since $h\leq^* \cf(b)$, we see that $\phi(\alpha)$ is not cofinal in $b(\alpha)$ for almost all $\alpha$, so $g(\alpha)=\sup(\phi(\alpha))$ for almost all $\alpha\in\kappa$. We let $\rho_-^1$ be the identity function, and $\rho_+^0:f\mapsto f+1$.

If $f\in\prod b$ and $\phi\in\Loc_\kappa^{b,h}$, let $g=\rho_-^0(\phi)=\rho_+^1(\phi)$ and $f+1=\rho_+^0(f)$. Then $f\in^*\phi$ implies $f(\alpha)\leq\sup(\phi(\alpha))=g(\alpha)$ for almost all $\alpha\in\kappa$, hence $f\leq^* g$. On the other hand, if $g\leq^* f$, then $g<^*f+1$, hence $\sup(\phi(\alpha))<f(\alpha)+1$ for almost all $\alpha\in\kappa$, implying $f+1\nini \phi$.
\end{proof}

In summary, we can draw the cardinal characteristics related to these relational systems in the diagram below, where the dashed lines require that $h\leq^*\cf(b)$. Note that $\bneq{b}\leq\non(\c M_\kappa)$ and $\cov(\c M_\kappa)\leq\dneq{b}$ follow from \Cref{evt difference and antiloc,monotonicity aloc} and \Cref{evt dif and meagre}.

\begin{figure}[h]
\centering
\begin{tikzpicture}[baseline,xscale=2.25,yscale=1.25]

\node (a1) at (-2,-2) {$\kappa^+$};
\node (bleq) at (-1,0) {$\bleq{b}$};
\node (bneq) at (0,2) {$\bneq{b}$};
\node (binf) at (0,0.67) {$\binf{b,h}$};
\node (bstar) at (-2,-.67) {$\bstar{b,h}$};
\node (dleq) at (1,0) {$\dleq{b}$};
\node (dneq) at (0,-2) {$\dneq{b}$};
\node (dinf) at (0,-0.67) {$\dinf{b,h}$};
\node (dstar) at (2,.67) {$\dstar{b,h}$};
\node (cM) at (-1,-2) {$\r{cov}(\c M_\kappa)$};
\node (nM) at (1,2) {$\r{non}(\c M_\kappa)$};
\node (c) at (2,2) {$2^{\kappa}$};

\draw (dinf) edge[->,dashed] (dleq); 
\draw (dleq) edge[->,dashed] (dstar);
\draw (bleq) edge[->,dashed] (binf); 
\draw (bstar) edge[->,dashed] (bleq);


\draw (dneq) edge[->] (dinf); 
\draw (binf) edge[->] (bneq);

\draw (binf) edge[->] (dstar); 
\draw (bstar) edge[->] (dinf);

\draw (bleq) edge[->] (bneq); 
\draw (dneq) edge[->] (dleq);

\draw (cM) edge[->] (dneq); 
\draw (bneq) edge[->] (nM);

\draw (bleq) edge[->] (dleq); 
\draw (a1) edge[->] (cM); 
\draw (a1) edge[->] (bstar); 
\draw (nM) edge[->] (c); 
\draw (dstar) edge[->] (c); 

\end{tikzpicture}
\caption{Diagram of the relations between cardinals on bounded spaces}\label{figure: diagram}
\end{figure}
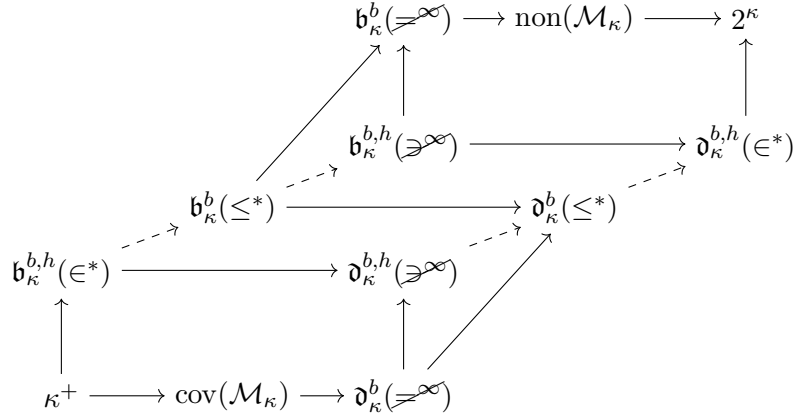

We saw that $\AL_{b,h}\preceq \sr D_b\preceq\sr L_{b,h}$ if $h\leq^*\cf(b)$, and in \Cref{relation star and leq 1} we showed that reversely $\sr L_{b,h}\preceq \sr D_b$ if $h=^*b$. Under the same assumption, we can also prove that $\sr D_b\preceq\AL_{b,h}$. Therefore, the dashed arrows in the above diagram collapse to become equalities in the case where $b=^*\cf(b)=^*h$.

\begin{thm}\label{h eq b makes domination}
If $h=^*b$, then $\sr D_b\preceq \AL_{b,h}$. 
\end{thm}
\begin{proof}
We define $\rho_-:\prod b\to\Loc_\kappa^{b,h}$ and $\rho_+:\prod b\to\prod b$ as a Tukey connection for $\sr D_b\preceq\AL_{b,h}$. 

We will have $\rho_-$ sending $g\in\prod b$ to $\phi\in\Loc_\kappa^{b,h}$ where $\phi(\alpha)=g(\alpha)$ whenever $h(\alpha)=b(\alpha)$ and arbitrary otherwise. This is well-defined, since $|g(\alpha)|<b(\alpha)=h(\alpha)$ for almost all $\alpha\in\kappa$. Let $\rho_+$ be the identity function.

If $f,g\in\prod b$, let $\phi=\rho_-(g)$. Then $f\nini \phi$ implies $f(\alpha)\notin \phi(\alpha)=g(\alpha)$ for almost all $\alpha\in\kappa$, hence $g\leq^* f$.
\end{proof}

The relation $\ED_b\preceq\AL_{b,h}$ can also be reversed for certain choices of $h$ and $b$, by the following theorem.

\begin{thm}\label{antiloc below event dif}
Let $b$ be increasing (not strictly) and $\ab{I_\alpha\mid \alpha<\kappa}$ be an interval partition of $\kappa$ with $|I_\alpha|=h(\alpha)$ for each $\alpha\in\kappa$ such that $b(\alpha)=b(\xi)=b(\alpha)^{h(\alpha)}$ for all $\xi\in I_\alpha$ and $\alpha\in\kappa$, then $\ED_b\equiv\AL_{b,h}$.
\end{thm}
\begin{proof}
We already know $\ED_b\preceq\AL_{b,h}$ from \Cref{cardinal relations fixed b h}, thus we show that $\AL_{b,h}\preceq \ED_b$. We do this by giving a Tukey connection $\rho_-:\Loc_\kappa^{b,h}\to\prod b$ and $\rho_+:\prod b\to\prod b$. For each $\alpha\in\kappa$ let $\pi_\alpha:b(\alpha)\to {}^{I_\alpha}b(\alpha)$ be a bijection, which exist by $b(\alpha)^{h(\alpha)}=b(\alpha)$.

Given $\phi\in\Loc_\kappa^{b,h}$, let $\lambda_\alpha=|\phi(\alpha)|$ and enumerate each $\phi(\alpha)=\sst{x^\alpha_\xi\in b(\alpha)\mid \xi\in\lambda_\alpha}$. We define $g_\xi^\alpha=\pi_\alpha(x^\alpha_\xi)\in {}^{I_\alpha}b(\alpha)$ for all $\alpha\in \kappa$ and $\xi\in\lambda_\alpha$. Fix $\iota_\alpha=\min(I_\alpha)$, then for every $\xi\in\lambda_\alpha$ we see that $\iota_\alpha+\xi\in I_\alpha$, since $\lambda_\alpha<h(\alpha)=|I_\alpha|$. Let $g\in\prod b$ be a function such that $g(\iota_\alpha+\xi)=g^\alpha_\xi(\iota_\alpha+\xi)$ for each $\alpha\in\kappa$ and $\xi\in\lambda_\alpha$. This is well defined, because $b(\alpha)=b(\xi)$ for all $\xi\in I_\alpha$. We let $\rho_-(\phi)=g$.

 We define $\rho_+(f)\in\prod b$ sending $\alpha\mapsto \pi_\alpha^{-1}(f\restriction I_\alpha)$.

Now suppose that $\phi\in\Loc_\kappa^{b,h}$ and $f\in\prod b$ and let $g=\rho_-(\phi)$ and $f'=\rho_+(f)$. Suppose that $f'\in^\infty \phi$ and $\alpha_0\in\kappa$, then there is $\alpha>\alpha_0$ such that $f'(\alpha)\in \phi(\alpha)$, so pick $\xi\in\lambda_\alpha$ such that $x^\alpha_\xi=f'(\alpha)$. Then we see 
\begin{align*}
g^\alpha_\xi=\pi_\alpha(x_\xi^\alpha)=\pi_\alpha(f'(\alpha))=\pi_\alpha(\pi^{-1}_\alpha(f\restriction I_\alpha))=f\restriction I_\alpha.
\end{align*}
In particular $f(\iota_\alpha+\xi)=g_\xi^\alpha(\iota_\alpha+\xi)=g(\iota_\alpha+\xi)$, and since $\iota_\alpha+\xi>\alpha_0$ and $\alpha_0$ was arbitrary, it follows that $f=^\infty g$.
\end{proof}

Note that the proof does not require the assumption that $b(\alpha)<\kappa$ for $\alpha\in\kappa$. Indeed, the above theorem holds even if we let $b:\kappa\to\st\kappa$, since $\kappa$ is inaccessible. It follows that the cardinal characteristics $\dinf{h}$ and $\binf{h}$ from the (unbounded) generalised Baire space do not depend on the choice of $h\in\gbs$, and can be related to the $\kappa$-meagre ideal as a consequence of \Cref{evt dif and meagre}.

\begin{crl}\label{uninteresting antiloc}
$\dinf{h}=\cov(\c M_\kappa)$ and $\binf{h}=\non(\c M_\kappa)$ for any choice of $h\in\gbs$.
\end{crl}

We will conclude this subsection with a relationship between antilocalisation and localisation for different parameters.
\begin{thm}\label{aloc loc intraction}
Let $h\cdot h'<^*b$ and $b^{g}\leq^* b'$ for all $g<h$, then there exists a Tukey connection ${\AL_{b,h}}\preceq {\sr L_{b',h'}}$.
\end{thm}
\begin{proof}
Let $\alpha_0\in\kappa$ be large enough such that $h(\alpha)\cdot h'(\alpha)<b(\alpha)$ and $b(\alpha)^{<h(\alpha)}\leq b'(\alpha)$ for all $\alpha\geq \alpha_0$. For each $\alpha\geq\alpha_0$, we fix an injection $\iota_\alpha:[b(\alpha)]^{< h(\alpha)}\inj b'(\alpha)$.

Given $\phi\in\Loc_\kappa^{b,h}$, let $\rho_-(\phi)=f'\in\prod b'$ map $\alpha\mapsto \iota_\alpha(\phi(\alpha))$ for any $\alpha\geq\alpha_0$, and arbitrary otherwise. Given $\phi'\in\Loc_\kappa^{b',h'}$, let $\rho_+(\phi)=f\in\prod b$ have $f(\alpha)\in b(\alpha)\setminus \Cup_{\xi\in\phi'(\alpha)\cap \ran(\iota_\alpha)}\iota_\alpha^{-1}(\xi)$ for any $\alpha\geq\alpha_0$, and arbitrary otherwise. Note that this is well-defined, since $|\phi'(\alpha)|< h'(\alpha)$ and $|\iota_\alpha^{-1}(\xi)|< h(\alpha)$, therefore $|\Cup_{\xi\in\phi'(\alpha)\cap \ran(\iota_\alpha)}\iota_\alpha^{-1}(\xi)|\leq h(\alpha)\cdot h'(\alpha)<b(\alpha)$.

If $\alpha\geq\alpha_0$ and $f'(\alpha)\in\phi'(\alpha)$, then $\phi(\alpha)=\iota_\alpha^{-1}(f'(\alpha))\subset \Cup_{\xi\in\phi'(\alpha)\cap\ran(\iota_\alpha)}\iota_\alpha^{-1}(\xi)$, so $f(\alpha)\notin \phi(\alpha)$. Therefore $f'\in^*\phi'$ implies $f\nini\phi$.
\end{proof}

\subsection{Infima and Suprema}

Remember that we have the following monotonicity results from \Cref{monotonicity loc,monotonicity aloc}. Let $b\leq b'$. Then:
\begin{align*}
\bstar{b',h}&\leq\bstar{b,h}&\dstar{b,h}&\leq\dstar{b',h}\\
\binf{b,h}&\leq\binf{b',h}&\dinf{b',h}&\leq\dinf{b,h}\qedhere
\end{align*}
This motivates the definition of the following cardinal characteristics.
\begin{dfn}
Given $h\in\gbs$, let:
\begin{align*}
\f{inf}_\kappa^h(\ins)&=\inf\sst{\bstar{b,h}\mid b\in\gbs}&
\f{sup}_\kappa^h(\ins)&=\sup\sst{\dstar{b,h}\mid b\in\gbs}\\
\f{inf}_\kappa^h(\nnii)&=\inf\sst{\dinf{b,h}\mid b\in\gbs}&
\f{sup}_\kappa^h(\nnii)&=\sup\sst{\binf{b,h}\mid b\in\gbs}\\
\f{inf}_\kappa(\neqi)&=\inf\sst{\dneq{b}\mid b\in\gbs}&
\f{sup}_\kappa(\neqi)&=\sup\sst{\bneq{b}\mid b\in\gbs}\qedhere
\end{align*}\vspace{-1.5\baselineskip}
\end{dfn}
Note that for $\ins$ we take infima over the $\f b$-side cardinals, while for $\nnii$ and $\neqi$ we take infima over the $\f d$-side cardinals, and vice versa for the suprema. Note also that $\f{inf}_\kappa(\neqi)=\f{inf}_\kappa^{\bar 2}(\nnii)$ and $\f{sup}_\kappa(\neqi)=\f{sup}_\kappa^{\bar 2}(\nnii)$ by \Cref{evt difference and antiloc}.

In [CM19] it is proved that the parameter $h$ does not matter if we work in $\bs$. This is based on two claims which we repeat below.

\begin{fct}\label{classical loc claim}\renewcommand{\qed}{\hfill$\square$}
For any $b,h,h'\in\bs$ such that $h\leq h'$ there exists $b'\in\bs$ such that $\f d_\omega^{b,h}(\ins)\leq\f d_\omega^{b',h'}(\ins)$ and $\f b_\omega^{b',h'}(\ins)\leq\f b_\omega^{b,h}(\ins)$.
\end{fct}

\begin{fct}\label{classical aloc claim}\renewcommand{\qed}{\hfill$\square$}
For any $b',h\in\bs$ such that $\bar 2\leq h$ there exists $b\in\bs$ such that $\f d_\omega^{b,h}(\nnii)\leq\f d_\omega^{b'}(\neqi)$ and $\f b_\omega^{b'}(\neqi)\leq\f b_\omega^{b,h}(\nnii)$.
\end{fct}

In the generalised context, \Cref{classical loc claim} can only be partially generalised, and indeed we will see that differences in the parameter $h$ can lead to consistently different cardinals. Meanwhile \Cref{classical aloc claim} is completely generalisable and we will look at this first. 
One could compare the next lemma to \Cref{antiloc below event dif}.

\begin{lmm}\label{aloc claim}
For any $b',h\in\gbs$ such that $\bar 2\leq h$ there exists $b\in\gbs$ such that $\AL_{b,h}\preceq\ED_{b'}$.
\end{lmm}
\begin{proof}
Let $\ab{I_\alpha\mid \alpha\in\kappa}$ be the (unique) interval partition of $\kappa$ such that  $I_\alpha\subset\min(I_\beta)$ for $\alpha<\beta$ and $|I_\alpha|=h(\alpha)$ for all $\alpha\in\kappa$. We define $b(\alpha)=\card{\prod_{\xi\in I_\alpha}b'(\xi)}$ and let $\pi_\alpha:b(\alpha)\bij \prod_{\xi\in I_\alpha}b'(\xi)$ be a bijection for each $\alpha\in\kappa$.

For each $\phi\in\Loc_\kappa^{b,h}$, let $\rho_-(\phi)=g_\phi\in\prod b'$ be defined as follows. For each $\alpha\in\kappa$ and we take some surjection $\iota_\alpha^\phi:I_\alpha\srj\phi(\alpha)$. Given $\xi\in I_\alpha$, let $\beta=\iota_\alpha^\phi(\xi)\in\phi(\alpha)$, then we define  $g_\phi(\xi)=\pi_\alpha(\beta)(\xi)$.

For $f'\in\prod b'$, we let $\rho_+(f')=f\in\prod b$ be given by $f(\alpha)=\pi_\alpha^{-1}(f'\restriction I_\alpha)$. 

Now $(\rho_-,\rho_+)$ forms a Tukey connection. If $\rho_+(f')=f\ini\phi$, let $\alpha\in\kappa$ be arbitrarily large such that $f(\alpha)\in\phi(\alpha)$. Take $\xi\in I_\alpha$ such that $f(\alpha)=\iota_\alpha^\phi(\xi)$, then $g_\phi(\xi)=\pi_\alpha(f(\alpha))(\xi)=f'(\xi)$, thus since $\alpha$ is arbitrarily large, we see that $f'\eqi g_\phi$.
\end{proof}
\begin{crl}
$\f{inf}_\kappa^h(\nnii)=\f{inf}_\kappa(\neqi)$ and $\f{sup}_\kappa^h(\nnii)=\f{sup}_\kappa(\neqi)$ for each $h\in\gbs$.
\end{crl}
\begin{proof}
By \Cref{monotonicity aloc} we see that $\f{inf}_\kappa(\neqi)\leq\f{inf}_\kappa^h(\nnii)$ and $\f{sup}_\kappa(\neqi)\geq\f{sup}_\kappa^h(\nnii)$.

By \Cref{aloc claim}, for any $b'\in\gbs$ we can find some $b\in\gbs$ such that $\dinf{b,h}\leq\dinf{b',\bar 2}$, thus $\f{inf}_\kappa(\neqi)\geq \f{inf}_\kappa^h(\nnii)$. Similar for $\f{sup}_\kappa(\neqi)\leq\f{sup}_\kappa^h(\nnii)$.
\end{proof}

An equality between $\f{inf}_\omega(\neqi)$ and $\non(\SN_\kappa)$ was first found by Rothberger \cite{R41}. Miller \cite{M81} used this result to remark that $\add(\c M)=\min\st{\f b_\omega(\leq^*),\f{inf}_\omega(\neqi)}$. One could also show that $\cof(\c M)=\max\st{\f d_\omega(\leq^*),\f{sup}_\omega(\neqi)}$. We show that each of these three claims generalises.

Instead of proving $\non(\SN_\kappa)=\f{inf}_\kappa(\neqi)$ directly, we show one direction by use of a Tukey connection. Let $\c X_0=\st{(g,b)\in\gbs\times\gbs\mid g<b}$ and $\c X_1=\st{\tau:\gbs\to\gbs\mid 
\forall b\in\gbs(\tau(b)<b)}$, and define a relation $J_\infty\subset\c X_0\times \c X_1$ by $(f,b)\rl J_\infty\tau$ iff $f\eqi \tau(b)$. We use this to define the relational system $\sr J_\infty$, which is equivalent to the categorical product $\bigotimes_{b\in\gbs}\ED_b^\lc$. We also define a relational system whose norms are the covering and uniformity numbers of the $\kappa$-strong measure zero ideal.
\begin{align*}
\sr J_{\infty}&=\ab{\c X_0,\c X_1,J_\infty}&\norm{\sr J_{\infty}^{}}&=\f{sup}_\kappa(\neqi)&\norm{\sr J_{\infty}^\lc}&=\f{inf}_\kappa(\neqi)\\
\sSN&=\ab{\gcs,\SN_\kappa,\in}&
\norm{\sSN}&=\cov(\SN_\kappa)&
\norm{\sSN^\lc}&=\non(\SN_\kappa)
\end{align*}
The following Tukey connection proves $\f{inf}_\kappa(\neqi)\geq\non(\SN_\kappa)$ and $\f{sup}_\kappa(\neqi)\leq\cov(\SN_\kappa)$.
\begin{thm}
$\sSN\succeq\sr J_{\infty}$.
\end{thm}
\begin{proof}
We describe $\rho_-:\c X_0\to\gcs$ and $\rho_+:\SN_\kappa\to\c X_1$.

For each $b\in\gbs$ let $\beta^b_\alpha\in\kappa$ be minimal such that there exists an injection $\iota^b_\alpha:b(\alpha)\inj{}^{\beta^b_\alpha}2$ and let $\gamma^b_\alpha$ be the ordinal sum $\sum_{\xi<\alpha}\beta^b_\xi$.

For any $(f,b)\in\c X_0$, we define $\rho_-(f,b)=f'\in\gcs$ piecewise by:
\begin{align*}
f'\restriction[\gamma^b_\alpha,\gamma^b_{\alpha+1}):\gamma^b_\alpha+\xi\mapsto\iota^b_\alpha(f(\alpha))(\xi).
\end{align*}
Given $X\in\SN_\kappa$ and $b\in\gbs$, we can find $\bar s^b=\sab{s^b_\alpha\in{}^{\gamma^b_{\alpha+1}}2\mid\alpha\in\kappa}$ such that $X$ is cofinally covered by $\bar s^b$. Define $t^b_\alpha\in{}^{\beta_\alpha^b}2$ by $t^b_\alpha:\xi\mapsto s^b_\alpha(\gamma_\alpha^b+\xi)$. Let $\rho_+(X)=\tau$, where $\tau(b):\alpha\mapsto (\iota_\alpha^b)^{-1}(t^b_\alpha)$ if this is defined and arbitrary otherwise.

Given $(f,b)\in\c X_0$ with $\rho_-(f,b)=f'$ and $X\in\SN_\kappa$ with $\rho_+(X)=\tau$, suppose that $f'\in X$, then since $X$ is cofinally covered by $\bar s^b$, there are cofinally many $\alpha$ such that $f'\in[s_\alpha^b]$, hence for such $\alpha$ we have  $f'\restriction[\gamma_\alpha^b,\gamma_{\alpha+1}^{b})=s^b_\alpha\restriction[\gamma_\alpha^b,\gamma_{\alpha+1}^b)$. But then $t_\alpha^b=\iota_\alpha^b(f(\alpha))$, and thus $\tau(b)(\alpha)=f(\alpha)$. Thus $f\eqi\tau(b)$, or equivalently $(f,b)\rl J_\infty\tau$.
\end{proof}

As said, $\non(\SN_\kappa)$ is actually equal to $\f{inf}_\kappa(\neqi)$. We prove the remaining direction below.

\begin{thm}
$\non(\SN_\kappa)\geq\f{inf}_\kappa(\neqi)$.
\end{thm}
\begin{proof}
Let $\pi:\gfcs\bij\kappa$ be some fixed bijection and $X\notin\SN_\kappa$, then there exists $f\in\gbs$ such that $X$ is not cofinally covered for any $\bar s=\ab{s_\alpha\mid \alpha\in\kappa}$ with $s_\alpha\in{}^{f(\alpha)}2$. For each $x\in X$ let $g_x:\alpha\mapsto \pi(x\restriction f(\alpha))$ and define $b:\alpha\mapsto \sup\st{\pi(s)+1\mid s\in{}^{f(\alpha)}2}$, then $g_x\in\prod b$. We set $D=\st{g_x\mid x\in X}$.

Given $h\in\prod b$, for each $\alpha$  define $s_\alpha=\pi^{-1}(h(\alpha))$ if $\dom(\pi^{-1}(h(\alpha))=f(\alpha)$ and $s_\alpha\in{}^{f(\alpha)}2$ arbitrary otherwise. If $h(\alpha)=g_x(\alpha)$, then $s_\alpha=x\restriction f(\alpha)$, hence $x\in[s_\alpha]$. Since $X$ is not cofinally covered by $\bar s=\ab{s_\alpha\mid \alpha\in\kappa}$, there exists $x\in X$ such that $h\neqi g_x$. Hence $D$ forms a witness for $\dneq{b}$.
\end{proof}

We cannot prove the dual claim that $\cov(\SN_\kappa)\leq\f{sup}_\kappa(\neqi)$, since $\f{sup}_\kappa(\neqi)\leq\non(\c M_\kappa)$ and it is consistent that $\non(\c M_\kappa)\leq\cof(\c M_\kappa)<\cov(\SN_\kappa)$. The model for this is a generalisation of the Sacks model. We give a sketch of the argument below. We will use the generalised Sacks forcing from \cite{Kan80}, which is equivalent to the forcing notion $\bb S_\kappa^{b,h}$ from \Cref{Sacks-like forcing def} with $b=h=\bar 2$. We refer to that section for definitions.

\begin{thm}
Let $\bb S_\kappa$ be the generalised Sacks forcing and let $\bar{\bb S}$ be the ${\leq}\kappa$-support iteration of $\bb S_\kappa$ of length $\kappa^{++}$. If $\b V\md\ap{2^\kappa=\kappa^+}$, then 
\begin{align*}
\b V^{\bar{\bb S}}\md\ap{\cof(\c M_\kappa)=\kappa^+<\cov(\SN_\kappa)=\kappa^{++}}.&\qedhere
\end{align*}
\end{thm}
\begin{proof}
It is important to note that $\bb S_\kappa$ is $\gbs$-bounding, see \cite{Kan80}. Because of this, there exists a $\subset$-cofinal subset of $\SN_\kappa$ whose elements can be coded terms of a family of sequences indexed by some fixed dominating family. Since the dominating family remains dominating, we can essentially characterise the elements of $\SN_\kappa$ from the ground model. Subsequently, we show that $\bb S_\kappa$ adds a generic $\kappa$-real that is not an element of any $X\in(\SN_\kappa)^\b V$. This concludes the proof, since any $\kappa^+$-sized witness of $\cov(\SN_\kappa)$ obtained in an intermediate step of the iteration is destroyed by the subsequent step, and it is known that ${\leq}\kappa$-support iteration of $\bb S_\kappa$ does not increase $\cof(\c M_\kappa)$ (see e.g. \cite{BBFM}).

Since $\b V\md\ap{2^\kappa=\kappa^+}$, we may fix some dominating family $D\subset\gbs$ of size $\kappa^+$. Let $\sigma=\st{\bar s^f\mid f\in D}$ be a family of sequences, where $\bar s^f=\sab{s_\xi^f\mid \xi\in\kappa}$ is such that $s_\xi^f\in {}^{f(\xi)}2$. We use $\sigma$ to define an element of $\SN_\kappa$:
\begin{align*}
\dsb{\sigma}=\ts\Cap_{f\in D}\Cap_{\alpha_0\in\kappa}\Cup_{\xi\in[\alpha_0,\kappa)}[s_\xi^f]\in\SN_\kappa
\end{align*}
Reversely it's easy to see that for every $X\in\SN_\kappa$ there exists some choice of $\sigma$ such that $X\subset\dsb\sigma$. We show that the set of conditions $S\in \bb S_\kappa$ with $[S]\cap\dsb{\sigma}=\emp$ is dense. It suffices to show that for each $S\in\bb S_\kappa$ there exists $f\in D$ and $T\leq S$ such that $[T]\cap \Cup_{\xi\in\kappa}[s_\xi^f]=\emp$.

Let $S\in\bb S_\kappa$ be arbitrary. We may assume without loss of generality that for each $\alpha$ there is $\beta_\alpha$ such that $\Split_\alpha(S)\subset{}^{\beta_\alpha}2$, and thus that $\ab{\beta_\alpha\mid \alpha\in\kappa}$ enumerates the club set of the heights of splitting levels of $S$. Now let $\gamma_\xi$ denote the $\xi$-th limit ordinal below $\kappa$ and consider the nonstationary set $\st{\beta_{\gamma_\xi+1}\mid \xi\in\kappa}$. Let $f_0:\xi\mapsto \beta_{\gamma_\xi+1}+1$ and pick $f\in D$ such that $f_0\leq^* f$. 

For each $\xi\in\kappa$ such that $f_0(\xi)\leq f(\xi)$ and $s^f_\xi\restriction \beta_{\gamma_\xi+1}\in S$, we prune $S$ by removing the part of the tree generated by the successor $s^f_\xi\restriction (\beta_{\gamma_\xi+1}+1)$ to get a tree $T\leq S$. Since $f_0\leq^* f$, it follows that $[T]\cap\dsb{\sigma}=\emp$. Finally, we only prune $S$ at successor splitting levels, thus $T\in\bb S_\kappa$.
\end{proof}

We will prove the connection between $\add(\c M_\kappa)$ and $\f{inf}_\kappa(\neqi)$, and between $\cof(\c M_\kappa)$ and $\f{sup}_\kappa(\neqi)$ directly, as it will be more similar to the proof of \Cref{infloc suploc}.
\begin{lmm}\label{inf sup meagre result}
$\add(\c M_\kappa)=\min\st{\bleq{},\f{inf}_\kappa(\neqi)}$ and 
$\cof(\c M_\kappa)=\max\st{\dleq{},\f{sup}_\kappa(\neqi)}$.
\end{lmm}
\begin{proof}
Remember that by \Cref{evt dif and meagre} we have $\dneq{}=\cov(\c M_\kappa)$ and $\bneq{}=\non(\c M_\kappa)$, hence by \Cref{add cof meagre} $\add(\c M_\kappa)=\min\st{\bleq{},\dneq{}}$ and $\cof(\c M_\kappa)=\max\st{\dleq{},\bneq{}}$. Moreover, it is clear from \Cref{monotonicity aloc} that $\dneq{}\leq\f{inf}_\kappa(\neqi)$ and $\f{sup}_\kappa(\neqi)\leq\bneq{}$.

Secondly, we prove that $\dneq{}<\bleq{}$ implies $\dneq{}=\f{inf}_\kappa(\neqi)$. Let $F\subset\gbs$ be a witness for $|F|=\dneq{}$ and assume $\dneq{}<\bleq{}$. Since $|F|<\bleq{}$, there exists $b\in\gbs$ such that $f<^* b$ for all $f\in F$. Let $f':\alpha\mapsto f(\alpha)$ if $f(\alpha)\in b(\alpha)$ and $f':\alpha\mapsto 0$ otherwise. Clearly $f'\in\prod b$ and for any $g\in\gbs$ we have $f\eqi g$ iff $f'\eqi g$. Therefore $F'=\st{f'\mid f\in F}\subset\prod b$ is a witness for $\dneq{b}$, thus we see:
\begin{align*}
	\dneq{}=|F|\geq\dneq{b}\geq\f{inf}_\kappa(\neqi)\geq\dneq{}.
\end{align*}
For the dual result, assume $\dleq{}<\bneq{}$ and towards contradiction let $\f{sup}_\kappa(\neqi)<\bneq{}$ as well. Let $D\subset\gbs$ witness $|D|=\dleq{}$ and for each $b\in D$ choose a witness $F_b\subset\prod b$ for $|F_b|=\bneq{b}\leq\f{sup}_\kappa(\neqi)$. Define $F=\Cup_{b\in D}F_b$, then $|F|=\dleq{}\cdot\f{sup}_\kappa(\neqi)<\bneq{}$, thus there exists $g\in\gbs$ such that $g\neqi f$ for all $f\in F$. Let $b\in D$ be such that $g<^* b$ and let $g':\alpha\mapsto g(\alpha)$ if $g(\alpha)< b(\alpha)$ and $g':\alpha\mapsto 0$ otherwise, then $g'\in\prod b$ and $g'\neqi f$ for all $f\in F_b$. But this contradicts that $F_b$ witnesses $|F_b|=\bneq{b}$.

Putting everything together, we have showed that:
\begin{align*}
	\add(\c M_\kappa)=\min\st{\bleq{},\dneq{}}&=\min\st{\bleq{},\f{inf}_\kappa(\neqi)},\\ 
	\cof(\c M_\kappa)=\max\st{\dleq{},\bneq{}}&=\max\st{\dleq{},\f{sup}_\kappa(\neqi)}.&&\qedhere
\end{align*}
\end{proof}
We will now show a generalisation of \Cref{classical loc claim}. It is not possible to prove the generalisation for every $b,h,h'\in\gbs$ with $h\leq h'$, since we will see in the next section that it is consistent that $\dstar{b,2^h}=\dstar{b',2^h}<\dstar{b,h}=\dstar{b',h}$ for all $b'\geq b$.

\begin{lmm}
For any $b,h,h'\in\gbs$ such that there exists a continuous strictly increasing sequence $\ab{\beta_\alpha\mid \alpha\in\kappa}$ with $h'(\alpha)\leq h(\xi)$ for all $\xi\geq\beta_\alpha$, there exists $b'\in\gbs$ such that $\sr L_{b,h}\preceq\sr L_{b',h'}$. 
\end{lmm}
\begin{proof}
Let $I_\alpha=[\beta_\alpha,\beta_{\alpha+1})$ for each $\alpha\in\kappa$, where we assume without loss of generality that $\beta_0=0$. Since $\ab{\beta_\xi\mid \xi\in\kappa}$ is continuous, we see that $\ab{I_\alpha\mid \alpha\in\kappa}$ is an interval partition of $\kappa$. We define $b'(\alpha)=\card{\prod_{\xi\in I_\alpha}b(\xi)}$ and a bijection $\pi_\alpha:b'(\alpha)\bij\prod_{\xi\in I_\alpha}b(\xi)$ for each $\alpha\in\kappa$.

$\rho_-:\prod b\to \prod b'$ is defined by $\rho_-(f)=f':\alpha\mapsto \pi_\alpha^{-1}(f\restriction I_\alpha)$. We define $\rho_+:\Loc_\kappa^{b',h'}\to\Loc_\kappa^{b,h}$ by $\rho_+(\phi')=\phi:\xi\mapsto \st{\pi_\alpha(x)(\xi)\mid x\in \phi'(\alpha)}$, where $\alpha$ is such that $\xi\in I_\alpha$. Note that this is well-defined, since $\xi\in I_\alpha$ implies $\xi\geq\beta_\alpha$, and thus $|\phi'(\alpha)|<h'(\alpha)\leq h(\xi)$.

Suppose that $f\in\prod b$, $\phi'\in\Loc_\kappa^{b',h'}$ and $\rho_-(f)=f'$, $\rho_+(\phi')=\phi$. Let $\alpha\in\kappa$ be such that $f'(\alpha)\in \phi'(\alpha)$. For any $\xi\in I_\alpha$, we then see that $\pi_\alpha(f'(\alpha))(\xi)=f(\xi)$, thus $f(\xi)\in \phi(\xi)$ for all $\xi\in I_\alpha$. Hence, if $f'\in^*\phi'$, it follows that $f\in^*\phi$.
\end{proof}

Classically, it can be shown analogously to \Cref{inf sup meagre result} that $\add(\c N)=\min\st{\f b_\omega(\leq^*),\f{inf}_\omega(\in^*)}$ and $\cof(\c N)=\max\st{\f d_\omega(\leq^*),\f{sup}_\omega(\in^*)}$. There is a generalisation of this result for $\gbs$ as well, although we have to replace $\add(\c N)$ and $\cof(\c N)$ by their combinatorial counterparts $\bstar{h}$ and $\dstar{h}$ (nota bene: these are the cardinals from the (un)bounded generalised Baire space $\gbs$, or equivalently with the bound $b=\bar \kappa$). Moreover, since the parameter $h$ is important, in the sense that differing $h$ leads to different cardinals, we can only prove a parametrised version of this result.

\begin{lmm}[cf. \cite{CM19} Lemma 3.12]\label{infloc suploc}
$\bstar{h}=\min\st{\bleq{},\f{inf}_\kappa^{h}(\ins)}$ and 
$\dstar{h}=\max\st{\dleq{},\f{sup}_\kappa^{h}(\ins)}$ for any $h\in\gbs$.
\end{lmm}
\begin{proof}
That $\bstar{h}\leq\min\st{\bleq{},\f{inf}_\kappa^h(\ins)}$ and $\dstar{h}\geq\max\st{\dleq{},\f{sup}_\kappa^{h}(\ins)}$ are clear.

Let $F\subset\gbs$ with $|F|<\min\st{\bleq{},\f{inf}_\kappa^{h}(\ins)}$, then there exists $b\in\gbs$ such that $f<^*b$ for all $f\in F$. Let $F'=\st{f'\mid f\in F}$ where $f':\alpha\mapsto f(\alpha)$ if $f(\alpha)< b(\alpha)$ and $f':\alpha\mapsto 0$ otherwise, then $f=^*f'\in\prod b$. Finally, there exists $\phi\in\Loc_\kappa^{b,h}$ such that $f'\in^*\phi$ for all $f'\in F'$ by $|F'|<\f{inf}_\kappa^h(\ins)\leq\bstar{b,h}$. Then also $f\in^*\phi$ for all $f\in F$, thus $|F|<\bstar{h}$.

Let $D\subset\gbs$ be a witness for $\dleq{}$ with $|D|=\dleq{}$. For each $b\in D$ choose a witness $\Phi_b\subset\Loc_\kappa^{b,h}$ for $\dstar{b,h}$ with $|\Phi_b|=\dstar{b,h}$. Let $\Phi=\Cup_{b\in D}\Phi_b$, then $|\Phi|\leq\max\st{\dleq{},\f{sup}_\kappa^h(\ins)}$. If $f\in\gbs$, then there is $b\in D$ such that $f<^*b$. Again, let $f':\alpha\mapsto f(\alpha)$ if $f(\alpha)< b(\alpha)$ and $f':\alpha\mapsto 0$ otherwise, then $f'\in\prod b$. Therefore, there exists $\phi\in\Phi_b$ such that $f'\in^*\phi$, and thus such that $f\in^*\phi$. This shows that $\Phi$ is a witness for $\dstar{h}$.
\end{proof}

\section{Consistency of Strict Inequalities}\label{section: consistency}

In this section we will discuss consistency of strict relations between cardinals we have discussed so far. We will first discuss the $\kappa$-Cohen model and the $\kappa$-Hechler model. Afterwards we show how we can separate cardinals of the form $\dstar{b,h}$ and of the form $\dinf{b,h}$ for different parameters $b,h$. In the first case, we can find $\kappa^+$ many parameters $\ab{b_\xi,h_\xi\mid\xi\in\kappa^+}$ and a forcing extension in which the cardinals $\dstar{b_\xi,h_\xi}$ are mutually distinct. In the second case, we can describe $\kappa$ many parameters $\ab{b_\xi,h_\xi\mid \xi\in\kappa}$ such that for any finite $A\subset \kappa$ there exists a forcing extension in which the cardinals $\dinf{b_\xi,h_\xi}$ with $\xi\in A$ are mutually distinct.

Before we continue, let us establish some notation regarding trees and product forcing.

Let $\c F$ be a set of functions such that $\dom(f)=\kappa$ for each $f\in\c F$. Let $\c F_{<\kappa}=\st{f\restriction\alpha\mid f\in\c F\la \alpha\in\kappa}$ be the set of initial segments of functions in $\c F$. Notably, if $\c F=\prod b$, then $\c F_{<\kappa}=\prod_{<\kappa}b$ and if $\c F=\Loc_\kappa^{b,h}$, then $\c F_{<\kappa}=\Loc_{<\kappa}^{b,h}$.

A subset $T\subset\c F_{<\kappa}$ is called a \emph{tree on $\c F$} if for every $u\in T$ and $\beta\in\dom(u)$ we have $u\restriction\beta\in T$. A subset $C\subset T$ is a \emph{chain} if for any $u,v\in C$ we have $u\subset v$ or $v\subset u$, and $C$ is called \emph{maximal} if there exists no chain $C'\subset T$ with $C\subsetneq C'$. A function $b:\alpha\to\kappa$ where $\alpha\leq \kappa$ is called a \emph{branch} of $T$ if there exists a maximal chain $C\subset T$ such that $b=\Cup C$. The set of branches of $T$ is denoted by $[T]$. We define the \emph{subtree} of $T$ generated by $u\in T$ as:
\begin{align*}
(T)_u=\st{v\in T\mid u\subset v\lo v\subset u}
\end{align*}
Given $u\in \c F_{<\kappa}$ and $f\in\c F$ such that $u\subset f$, let $x=f(\dom(u))$, then we write $u^\frown \ab x$ for $f\restriction (\dom(u)+1)$. If $u\in T$, let $v\in T$ be a \emph{successor} of $u$ if there exists $x$ such that $v=u^\frown \ab x$. We denote the set of successors of $u$ in $T$ by $\suc(u,T)$.

We call $u$ a {\bf $\lambda$-splitting} node (of $T$), if $\lambda\leq\card{\suc(u,T)}$. We say $u$ is a {\bf splitting} node if it is a 2-splitting node, and a {\bf non-splitting} node otherwise. If $u$ is a $\lambda$-splitting node, but not a $\mu$-splitting node for any cardinal $\mu$ with $\lambda<\mu$, then we say that $u$ is a {\bf sharp} $\lambda$-splitting node.  We let $\Split_\alpha(T)$ be the set of all $u\in T$ such that $u$ is splitting and 
\begin{align*}
\ot(\st{\beta\in\dom(u)\mid u\restriction\beta\text{ is splitting}})=\alpha
\end{align*}
where $\ot(X)$ denotes the order-type of $(X,\in)$. We let $\Lev_\xi(T)=\st{u\in T\mid \dom(u)=\xi}$ denote the $\xi$-th level of the tree $T$.

Given a forcing notion $\bb P$, then a sequence $\ab{\leq_\alpha\mid \alpha\in\kappa}$ on $\bb P$ is called a \emph{fusion ordering} when:
\begin{itemize}
	\item $q\leq_0 p$ iff $q\leq p$, and
	\item $q\leq_\beta p$ implies $q\leq_\alpha p$ for all $\alpha<\beta$.
\end{itemize}
A \emph{fusion sequence} is a sequence of conditions $\ab{p_\alpha\mid \alpha\in\kappa}$ such that $p_{\beta}\leq_\alpha p_\alpha$ for all $\alpha\leq\beta\in \kappa$. We say that $\bb P$ is \emph{closed under fusion} if every fusion sequence $\ab{p_\alpha\mid \alpha\in\kappa}$ has some $p$ with $p\leq_\alpha p_\alpha$ for all $\alpha\in\kappa$.

Let ${\c A}$ be a set of ordinals and $\bb P_\xi$ be a forcing notion for each $\xi\in \c A$. We will denote $\leq_{\bb P_\xi}$ as $\leq^\xi$, or more commonly as $\leq$ when there is no possibility for confusion, and we denote $\ft_{\bb P_\xi}$ as $\ft_\xi$. Consider the set $\c C$ of (choice) functions $p:{\c A}\to\bb \Cup_{\xi\in {\c A}}\bb P_\xi$ such that $p(\xi)\in \bb P_\xi$ for each $\xi\in {\c A}$. For any $p\in \c C$, we define the {\bf support} of $p$ as $\supp(p)=\{\xi\in {\c A}\mid p(\xi)\neq\ft_{\xi}\}$. We define the ${\leq}\kappa$-support product as follows: 
\begin{align*}
\bar{\bb P}=\textstyle\prod^{\leq\kappa}_{\xi\in {\c A}}\bb P_\xi=\st{p\in\c C\smid \card{\supp(p)}\leq \kappa}.
\end{align*}
If $p,q\in\bar{\bb P}$, then $q\leq_{\bar{\bb P}} p$ iff $q(\xi)\leq^\xi p(\xi)$ for all $\xi\in {\c A}$. We will again write $q\leq p$ instead of $q\leq_{\bar{\bb P}} p$ if the context is clear, and we will write $\bar{\ft}$ for $\ft_{\bar {\bb P}}$.

If $\bar{\bb P}$ is a product of forcing notions $\bb P_\xi$ such that each $\bb P_\xi$ has a fusion ordering $\leq_\alpha^\xi$, then we define for each $p,q\in\bar{\bb P}$, $\alpha\in\kappa$, and $Z\subset \c A$ with $|Z|<\kappa$ a \emph{generalised fusion relation} $q\leq_{Z,\alpha}p$ iff $q\leq_{\bar{\bb P}} p$ and for each $\xi\in Z$ we have $q(\xi)\leq_\alpha^\xi p(\xi)$.

A {\bf generalised fusion sequence} is a sequence $\ab{(p_\alpha,Z_\alpha)\mid \alpha\in\kappa}$ such that:

\begin{enumerate}
\item $p_\alpha\in \bar {\bb P}$ and $Z_\alpha\in [\c A]^{<\kappa}$ for each $\alpha\in\kappa$,
\item $p_{\beta}\leq_{Z_\alpha,\alpha}p_\alpha$ and $Z_\alpha\subset Z_{\beta}$ for all $\alpha\leq\beta\in\kappa$,
\item for limit $\delta$ we have $Z_\delta=\Cup_{\alpha\in\delta} Z_\alpha$,
\item $\Cup_{\alpha\in\kappa} Z_\alpha=\Cup_{\alpha\in\kappa}\supp(p_\alpha)$.\qedhere
\end{enumerate}
We say that $\bar{\bb P}$ is \emph{closed under generalised fusion} if every generalised fusion sequence $\ab{(p_\alpha,Z_\alpha)\mid \alpha\in\kappa}$ has some $p\in\bar{\bb P}$ with $\supp(p)=\Cup_{\alpha\in\kappa}\supp(p)$ and $p\leq_{\alpha,Z_\alpha} p_\alpha$ for all $\alpha\in\kappa$. Note that point 4.\ implies that for each $\xi\in\c A$ we have $p(\xi)\leq_\alpha p_\alpha(\xi)$ for almost all $\alpha\in\kappa$.

Suppose $X\subset\bar{\bb P}$ and $\c B\subset \c A$, we define 
\begin{align*}
X\restriction \c B&=\st{p\restriction \c B\mid p\in X}
\end{align*}
Let $\c B^c=\c A\setminus \c B$ and $G\subset\bar {\bb P}$ be $\bar{\bb P}$-generic over $\b V$, then clearly $\bar{\bb P}$ and $(\bar{\bb P}\restriction \c B)\times (\bar{\bb P}\restriction \c B^c)$ are forcing equivalent, $(G\restriction \c B)\times (G\restriction \c B^c)$ is $(\bar{\bb P}\restriction \c B)\times (\bar{\bb P}\restriction \c B^c)$-generic and 
\begin{align*}
\b V[G]=\b V[(G\restriction \c B)\times (G\restriction \c B^c)]=\b V[G\restriction \c B][G\restriction \c B^c].
\end{align*}

\subsection{The $\kappa$-Cohen model}

We will first define $\kappa$-Cohen forcing.

\begin{dfn}
Let $\lambda\geq\kappa$ be a cardinal. We define \emph{$\kappa$-Cohen forcing} adding $\lambda$ many $\kappa$-Cohen reals as the set $\bb C_\kappa^\lambda$ with as conditions functions $p$ with $\dom(p)\in[\lambda\times\kappa]^{<\kappa}$ and values $\ran(p)\subset \kappa$ ordered by $q\leq p$ iff $q\supset p$.
\end{dfn}

The effect of $\kappa$-Cohen forcing is analogous to the effect of Cohen forcing in the context of $\bs$. Namely, $\kappa$-Cohen reals lie outside of any $\kappa$-meagre set coded in the ground model, and any $\kappa^+$ many $\kappa$-Cohen reals form a $\kappa$-nonmeagre set. Consequently, if $\cf(\lambda)>\kappa$, then after forcing with $\bb C_\kappa^\lambda$ we have $\cov(\c M_\kappa)=\lambda=2^\kappa$, while $\non(\c M_\kappa)=\kappa^+$. Particularly, $\kappa^+=\non(\c M_\kappa)<\cov(\c M_\kappa)=\kappa^{++}$ holds in the \emph{$\kappa$-Cohen model}, which we will define as the result of forcing with $\bb C_\kappa^{\kappa^{++}}$ over $\b V\md\ap{\s{GCH}}$. We will not give a proof of these facts here, as they are well-known and could be found, for example, as Proposition 47 in \cite{BBFM}.

The effect this has on the bounded cardinals we have discussed so far, is that the $\f d$-side cardinals are all large in the $\kappa$-Cohen model, since $\cov(\c M_\kappa)$ is a lower bound for these cardinals, whereas the $\f b$-side cardinals are all small, since $\non(\c M_\kappa)$ is an upper bound. That is, the diagram from \Cref{figure: diagram} is divided as follows:

\begin{figure}[h]
\centering
\begin{tikzpicture}[baseline,xscale=2,yscale=1]

\node (a1) at (-2,-2) {$\kappa^+$};
\node (bleq) at (-1,0) {$\bleq{b}$};
\node (bneq) at (0,2) {$\bneq{b}$};
\node (binf) at (0,0.67) {$\binf{b,h}$};
\node (bstar) at (-2,-.67) {$\bstar{b,h}$};
\node (dleq) at (1,0) {$\dleq{b}$};
\node (dneq) at (0,-2) {$\dneq{b}$};
\node (dinf) at (0,-0.67) {$\dinf{b,h}$};
\node (dstar) at (2,.67) {$\dstar{b,h}$};
\node (cM) at (-1,-2) {$\r{cov}(\c M_\kappa)$};
\node (nM) at (1,2) {$\r{non}(\c M_\kappa)$};
\node (c) at (2,2) {$\kappa^{++}$};

\draw (dinf) edge[] (dleq); 
\draw (dleq) edge[] (dstar);
\draw (bleq) edge[] (binf); 
\draw (bstar) edge[] (bleq);


\draw (dneq) edge[] (dinf); 
\draw (binf) edge[] (bneq);

\draw (binf) edge[dotted] (dstar); 
\draw (bstar) edge[dotted] (dinf);

\draw (bleq) edge[] (bneq); 
\draw (dneq) edge[] (dleq);

\draw (cM) edge[] (dneq); 
\draw (bneq) edge[] (nM);

\draw (bleq) edge[dotted] (dleq); 
\draw (a1) edge[dotted] (cM); 
\draw (a1) edge[] (bstar); 
\draw (nM) edge[dotted] (c); 
\draw (dstar) edge[] (c); 

\draw[dashed, thicker,rounded corners=10] (-1.65,-2.5) -- (-1.65,-1) -- (1.65,1) -- (1.65,2.5);

\end{tikzpicture}
\caption{The $\kappa$-Cohen model}
\end{figure}
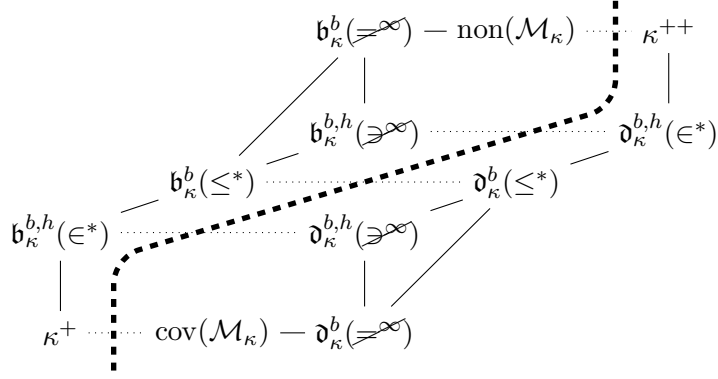

\subsection{The $\kappa$-Hechler models}

We have already encountered generalised Hechler forcing in \Cref{Hechler general definition} as the forcing notion $\bb D_\kappa^b(C)$, where $b$ is a function with domain $\kappa$ and $C\subset \kappa$ is club. The most natural form of this forcing notion is $\bb D_\kappa\equiv\bb D_\kappa^{\bar\kappa}(\kappa)$. We will restate the definition for $\bb D_\kappa$ individually below.

\begin{dfn}
We define \emph{$\kappa$-Hechler forcing} as the forcing notion $\bb D_\kappa$ with conditions $(s,f)$ such that $s\in\gfbs$ and $f\in\gbs$. The ordering is defined by $(t,g)\leq(s,f)$ iff $s\subset t$ and $f\leq g$ and for every $\alpha\in\dom(t)\setminus\dom(s)$ we have $f(\alpha)\leq t(\alpha)$. 
\end{dfn}

As with $\kappa$-Cohen forcing, the effect of $\kappa$-Hechler forcing is in many aspects similar to the situation on $\bs$. For instance, $\kappa$-Hechler forcing adds a dominating $\kappa$-real and a $\kappa$-Cohen real, is ${<}\kappa$-closed, $\kappa$-centred (which we define below), and thus has the ${<}\kappa^+$-c.c.. As a consequence of this, a ${<}\kappa$-support iteration of $\kappa$-Hechler forcing of length $\kappa^{++}$ over $\b V\md\ap{2^\kappa=\kappa^+}$ will increase both $\f b_\kappa(\leq^*)$ and $\cov(\c M_\kappa)$, and thus by \Cref{add cof meagre} it will result in a model where $\add(\c M_\kappa)=\kappa^{++}$. We will call this model the \emph{$\kappa$-Hechler model}. Classically, Hechler forcing does not add random reals, and therefore $\cov(\c N)$ and $\add(\c N)$ stay small in the Hechler model. A similar thing happens in the generalised case, where the analogue for $\add(\c N)$, namely $\bstar{h}$, remains of size $\kappa^+$ in the $\kappa$-Hechler model. We will show in this section that all of the cardinal characteristics on bounded spaces on the $\f b$-side will remain small as well. Specifically $\f{sup}_\kappa(\neqi)=\kappa^+$ in the $\kappa$-Hechler model.

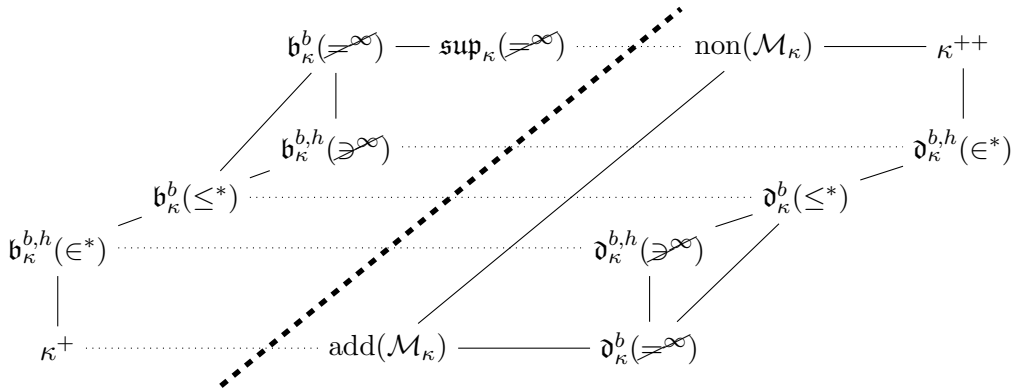
\begin{figure}[H]
\centering
\begin{tikzpicture}[baseline,xscale=2.75,yscale=1]

\node (a1) at (-2.33,-2) {$\kappa^+$};
\node (bleq) at (-1.67,0) {$\bleq{b}$};
\node (bneq) at (-1,2) {$\bneq{b}$};
\node (sup) at (-0.2,2) {$\f{sup}_\kappa(\neqi)$};
\node (binf) at (-1,0.67) {$\binf{b,h}$};
\node (bstar) at (-2.33,-.67) {$\bstar{b,h}$};
\node (dleq) at (1.25,0) {$\dleq{b}$};
\node (dneq) at (0.5,-2) {$\dneq{b}$};
\node (dinf) at (0.5,-0.67) {$\dinf{b,h}$};
\node (dstar) at (2,.67) {$\dstar{b,h}$};
\node (cM) at (-0.75,-2) {$\r{add}(\c M_\kappa)$};
\node (nM) at (1,2) {$\r{non}(\c M_\kappa)$};
\node (c) at (2,2) {$\kappa^{++}$};

\draw (dinf) edge[] (dleq); 
\draw (dleq) edge[] (dstar);
\draw (bleq) edge[] (binf); 
\draw (bstar) edge[] (bleq);


\draw (dneq) edge[] (dinf); 
\draw (binf) edge[] (bneq);

\draw (binf) edge[dotted] (dstar); 
\draw (bstar) edge[dotted] (dinf);

\draw (bleq) edge[] (bneq); 
\draw (dneq) edge[] (dleq);

\draw (cM) edge[] (dneq); 
\draw (bneq) edge[] (sup);
\draw (sup) edge[dotted] (nM);
\draw (cM) edge[] (nM);

\draw (bleq) edge[dotted] (dleq); 
\draw (a1) edge[dotted] (cM); 
\draw (a1) edge[] (bstar); 
\draw (nM) edge[] (c); 
\draw (dstar) edge[] (c); 

\draw[dashed, thicker,rounded corners=10] (-1.55,-2.5) -- (0.65,2.5);

\end{tikzpicture}
\caption{The $\kappa$-Hechler model}
\end{figure}
The above diagram summarises the situation in the $\kappa$-Hechler model, where the diagram is adapted from \Cref{figure: diagram} with the addition of $\f{sup}_\kappa(\neqi)$, and with $\cov(\c M_\kappa)$ replaced by the smaller $\add(\c M_\kappa)$.

A closely related model to the $\kappa$-Hechler model, is the \emph{dual $\kappa$-Hechler model}, which we will define as the result of a ${<}\kappa$-support iteration of $\kappa$-Hechler forcing of length (the ordinal) $\kappa^{++}\cdot\kappa^+$ over $\b V\md\ap{\s{GCH}}$. In this case, a cofinal set of $\kappa$-Hechler reals of size $\kappa^+$ will form a dominating family, and since they are $\kappa$-Cohen generic as well, they form a witness for a $\kappa$-nonmeagre set of size $\kappa^+$. Therefore, in the dual $\kappa$-Hechler model we have $\f d_\kappa(\leq^*)=\non(\c M_\kappa)=\kappa^+$, and by \Cref{add cof meagre} also $\cof(\c M_\kappa)=\kappa^+$. For similar reasons as before, the $\kappa$-Hechler forcing does not affect the size of the cardinal characteristics on bounded spaces, and thus we will show that the cardinals on the $\f d$-side remain large.

We can summarise the dual $\kappa$-Hechler model as follows:

\begin{figure}[H]
\centering
\begin{tikzpicture}[baseline,xscale=2.75,yscale=1]

\node (a1) at (-2,-2) {$\kappa^+$};
\node (bleq) at (-1.25,0) {$\bleq{b}$};
\node (bneq) at (-0.5,2) {$\bneq{b}$};
\node (binf) at (-0.5,0.67) {$\binf{b,h}$};
\node (bstar) at (-2,-.67) {$\bstar{b,h}$};
\node (dleq) at (1.67,0) {$\dleq{b}$};
\node (dneq) at (1,-2) {$\dneq{b}$};
\node (dinf) at (1,-0.67) {$\dinf{b,h}$};
\node (dstar) at (2.33,.67) {$\dstar{b,h}$};
\node (inf) at (0.2,-2) {$\f{inf}_\kappa(\neqi)$};
\node (cM) at (-1,-2) {$\r{cov}(\c M_\kappa)$};
\node (nM) at (0.75,2) {$\r{cof}(\c M_\kappa)$};
\node (c) at (2.33,2) {$\kappa^{++}$};

\draw (dinf) edge[] (dleq); 
\draw (dleq) edge[] (dstar);
\draw (bleq) edge[] (binf); 
\draw (bstar) edge[] (bleq);


\draw (dneq) edge[] (dinf); 
\draw (binf) edge[] (bneq);

\draw (binf) edge[dotted] (dstar); 
\draw (bstar) edge[dotted] (dinf);

\draw (bleq) edge[] (bneq); 
\draw (dneq) edge[] (dleq);

\draw (cM) edge[dotted] (inf); 
\draw (inf) edge[] (dneq);
\draw (bneq) edge[] (nM);
\draw (cM) edge[] (nM);

\draw (bleq) edge[dotted] (dleq); 
\draw (a1) edge[] (cM); 
\draw (a1) edge[] (bstar); 
\draw (nM) edge[dotted] (c); 
\draw (dstar) edge[] (c); 

\draw[dashed, thicker,rounded corners=10] (-0.65,-2.5) -- (1.55,2.5);

\end{tikzpicture}
\caption{The dual $\kappa$-Hechler model}
\end{figure}

We refer to section 4 of \cite{BBFM} for proofs of the sizes of $\add(\c M_\kappa)$ and $\cof(\c M_\kappa)$ and of $\bstar{h}$ and $\dstar{h}$ in the two models and focus here on the cardinalities of the cardinals defined on bounded spaces.

It will be sufficient to prove that $\bneq{b}=\kappa^+$ in the $\kappa$-Hechler model and $\dneq{b}=\kappa^{++}$ in the dual $\kappa$-Hechler model and thus we focus on the $\neqi$ relation in particular. The main lemma we need concerns $\kappa$-centred forcing notions, which is a generalisation of $\sigma$-centredness to the context of $\gbs$.

\begin{dfn}
Let $\bb P$ be a forcing notion. In the context of $\gbs$, we will define a subset $P\subset\bb P$ to be \emph{${<}\kappa$-linked} if for every subset $Q\in[P]^{<\kappa}$ there exists $r\in\bb P$ such that $r\leq q$ for all $q\in Q$. We say that $\bb P$ is \emph{$\kappa$-centred} if $\bb P=\Cup_{\gamma\in\kappa}\bb P_\gamma$ where each $\bb P_\gamma$ is ${<}\kappa$-linked.
\end{dfn}

\begin{lmm}\label{approximation cof eq}
Let $\bb P$ be a $\kappa$-centred forcing notion and $\dot f$ be a $\bb P$-name such that $\fc_{\bb P}\ap{\dot f\in\prod b}$, then there is a family $\st{g_\gamma\mid \gamma\in\kappa}\subset\prod b$ in the ground model such that for any $h\in\prod b$ with $h\eqi g_\gamma$ for all $\gamma\in\kappa$ we have $\fc_{\bb P}\ap{h\eqi\dot f}$. Consequently, there exists $g\in\prod b$ such that $\fc_{\bb P}\ap{g\eqi \dot f}$.
\end{lmm}
\begin{proof}
We let $\bb P=\Cup_{\gamma\in\kappa}\bb P_\gamma$ such that each $\bb P_\gamma$ is ${<}\kappa$-linked and we let $\dot f$ be a $\bb P$-name such that $\fc_{\bb P}\ap{\dot f\in\prod b}$. Let $\alpha,\gamma\in\kappa$, then there exists $\beta_\alpha^\gamma\in b(\alpha)$ such that for all $p\in\bb P_\gamma$ we have $p\nfc\ap{\dot f(\alpha)\neq\beta_\alpha^\gamma}$: if not, let $p_\beta\in\bb P_\gamma$ be such that $p_\beta\fc\ap{\dot f(\alpha)\neq\beta}$, then $\st{p_\beta\mid \beta\in b(\alpha)}$ has no common extension, contradicting that $\bb P_\gamma$ is ${<}\kappa$-linked. Let $g_\gamma:\alpha\mapsto \beta^\gamma_\alpha$.

Suppose that $h\in\prod b$ and $h\eqi g_\gamma$ for all $\gamma\in\kappa$ and let $\alpha_0\in\kappa$ and $p\in\bb P$ be arbitrary. There exists $\gamma\in\kappa$ such that $p\in \bb P_\gamma$, and since $h\eqi g_\gamma$ we can find $\alpha\geq \alpha_0$ such that $h(\alpha)=g_\gamma(\alpha)=\beta_\alpha^\gamma$. But, by construction of $\beta_\alpha^\gamma$ we know that $p\nfc\ap{\dot f(\alpha)\neq \beta_\alpha^\gamma}$, thus there exists $p'\leq p$ such that $p'\fc\ap{\dot f(\alpha)=\beta_\alpha^\gamma=h(\alpha)}$. Since $\alpha_0$ and $p$ were arbitrary, we see that $\fc_{\bb P}\ap{\dot f\eqi h}$.

To find $g\in\prod b$ such that $\fc_{\bb P}\ap{g\eqi \dot f}$, fix some surjection $\pi:\kappa\srj\kappa$ such that $|\pi^{-1}(\gamma)|=\kappa$ for all $\gamma\in\kappa$ and let $g:\alpha\mapsto \beta_{\alpha}^{\pi(\alpha)}$. 
\end{proof}

It is not difficult to see that $\bb D_\kappa$ is $\kappa$-centred, and thus that the above lemma applies to $\bb D_\kappa$. However, we will use an iteration of $\bb D_\kappa$ and it is not the case that every iteration of $\kappa$-centred forcing notions is itself $\kappa$-centred. However, for short iterations (of length ${<}(2^\kappa)^+$) of forcing notions where we can give a canonical bound in the ground model, we can show that $\kappa$-centredness is preserved. We will make this precise with the following definition and lemma, although we will refer to Lemma 55 of \cite{BBFM} for a proof of the lemma.

\begin{dfn}
A forcing notion $\bb P$ is \emph{$\kappa$-centred with canonical bounds} if $\bb P=\Cup_{\gamma\in\kappa}\bb P_\gamma$ where each $\bb P_\gamma$ is ${<}\kappa$-linked, $\bb P$ is ${<}\kappa$-closed and there exists $f^{\bb P}:\gfbs\to\kappa$ such that for every $\lambda<\kappa$ and decreasing sequence $\ab{p_\alpha\mid \alpha<\lambda}$ with $p_\alpha\in\bb P_{\gamma_\alpha}$ there exists $p\in\bb P_{\gamma}$ with $\gamma=f^{\bb P}(\ab{\gamma_\alpha\mid \alpha\in\lambda})$ such that $p\leq p_\alpha$ for all $\alpha\in\lambda$.
\end{dfn}
\begin{lmm}[{[BBTFM18] Lemma 55}]\label{preservation centred}
(Let $\kappa$ be inaccessible.) Let $\mu<(2^\kappa)^+$ and $\sab{\bb P_\alpha,\dot{\bb Q}_\alpha\mid \alpha<\mu}$ be a ${<}\kappa$-support iteration, where each $\dot{\bb Q}_\alpha$ is the name of a forcing notion that is $\kappa$-centred with canonical bounds given by a function $f^{\dot{\bb Q}_\alpha}\in\b V$. Then $\bb P_\mu$ is $\kappa$-centred.
\end{lmm}

Naturally, $\bb D_\kappa$ is $\kappa$-centred with canonical bounds.

\begin{lmm}
$\bb D_\kappa$ is $\kappa$-centred with canonical bounds.
\end{lmm}
\begin{proof}
Let $\bb D_\kappa^s=\st s\times \gbs\subset\bb D_\kappa$, then $\bb D_\kappa^s$ is clearly ${<}\kappa$-linked and $\bb D_\kappa=\Cup_{s\in\gfbs}\bb D_\kappa^s$ is therefore $\kappa$-centred. It is also easy to see that $\bb D_\kappa$ is ${<}\kappa$-closed, since for any decreasing sequence $\ab{(s_\alpha,g_\alpha)\mid \alpha\in\lambda}$ with $\lambda<\kappa$ we can find a function $g$ such that $g\geq g_\alpha$ for each $\alpha\in\lambda$ and $(s,g)\leq (s_\alpha,g_\alpha)$, where $s=\Cup_{\alpha\in\lambda}s_\alpha$. Let $f^{\bb D_\kappa}:{}^{<\kappa}(\gfbs)\to\gfbs$ send $\ab{s_\alpha\mid \alpha\in\lambda}\mapsto \Cup_{\alpha\in\lambda}s_\alpha$ (assuming this is well-defined, and arbitrary otherwise), then $f^{\bb D_\kappa}$ witnesses that $\bb D_\kappa$ is $\kappa$-centred with canonical bounds.
\end{proof}
We are now ready to give proofs for the values of the $\f b$-side cardinals in the $\kappa$-Hechler model and the $\f d$-side cardinals in the dual $\kappa$-Hechler model.
\begin{thm}
Let $\bb P=\sab{\bb P_\alpha,\dot{\bb Q}_\alpha\mid \alpha\in\kappa^{++}}$ be a ${<}\kappa$-support iteration,  where $\dot{\bb Q}_\alpha$ is a $\bb P_\alpha$-name for the $\kappa$-Hechler forcing $\bb D_\kappa$, let $\b V\md\ap{2^\kappa=\kappa^+}$ and let $G$ be $\bb P$-generic over $\b V$. For any $b,h\in\gbs$ (including new functions) such that $h\leq^*\cf(b)$ we have:
\begin{align*}
\b V[G]&\md\ap{\bstar{b,h}=\bleq{b}=\binf{b,h}=\bneq{b}=\kappa^{+}}\qedhere
\end{align*}
\end{thm}
\begin{proof}
It suffices to prove that $\b V[G]\md\ap{\bneq{b}=\kappa^{+}}$. Let $b\in(\gbs)^{\b V[G]}$, then there exists some $\alpha\in\kappa^{++}$ such that $b\in(\gbs)^{\b V[G_\alpha]}$. Note that $\b V[G_\alpha]\md\ap{2^\kappa=\kappa^+}$ holds for any $\alpha<\kappa^{++}$, and thus $\b V[G]\md\ap{|(\prod b)^{\b V[G_\alpha]}|=\kappa^+}$. We will show that $(\prod b)^{\b V[G_\alpha]}$ is a witness for $\b V[G]\md\ap{\bneq{b}=\kappa^+}$. We will assume for notational convenience that in fact $b\in\b V$.

Let $\dot f$ be a $\bb P$-name such that $\fc_{\bb P}\ap{\dot f\in\prod b}$, and let $f\in\b V[G]$ be the function named by $\dot f$. Then there is some $\mu<\kappa^{++}$ such that $f\in\b V[G_\mu]$, and thus we can assume without loss of generality that $\dot f$ is a $\bb P_\mu$-name. Since $\b V\md\ap{\mu<(2^\kappa)^+=\kappa^{++}}$, we see that $\bb P_\mu$ is $\kappa$-centred by \Cref{preservation centred}. Therefore there exists $h\in(\prod b)^\b V$ such that $\fc_{\bb P_\mu}\ap{\dot f\eqi h}$, as was needed.
\end{proof}
\begin{thm}
Let $\bb P=\sab{\bb P_\alpha,\dot{\bb Q}_\alpha\mid \alpha\in\kappa^{++}\cdot\kappa^+}$ be a ${<}\kappa$-support iteration,  where $\dot{\bb Q}_\alpha$ is a $\bb P_\alpha$-name for the $\kappa$-Hechler forcing $\bb D_\kappa$, let $\b V\md\ap{2^\kappa=\kappa^+}$ and let $G$ be $\bb P$-generic over $\b V$. For any $b,h\in\gbs$ (including new functions) such that $h\leq^*\cf(b)$ we have:
\begin{align*}
\b V[G]&\md\ap{\dstar{b,h}=\dleq{b}=\dinf{b,h}=\dneq{b}=\kappa^{++}}\qedhere
\end{align*}
\end{thm}
\begin{proof}
It suffices to prove that $\b V[G]\md\ap{\dneq{b}=\kappa^{++}}$. Let $b\in(\gbs)^{\b V[G]}$, then there exists some $\alpha\in\kappa^{++}\cdot\kappa^+$ such that $b\in(\gbs)^{\b V[G_\alpha]}$. Let $\beta=\alpha+\kappa^{++}$, then $\beta<\kappa^{++}\cdot\kappa^+$ and $\b V[G_\beta]\md\ap{\add(\c M_\kappa)=\kappa^{++}=2^\kappa}$ by similar arguments as those given for the $\kappa$-Hechler model. This implies particularly that $\b V[G_\beta]\md\ap{\dneq{b}=\kappa^{++}=2^\kappa}$ as well. We will work in $\b V[G_\beta]$ as our new ground model. Let $\bb P'=\bb P/G_\beta\in\b V[G_\beta]$ be the quotient forcing and $H$ be $\bb P'$-generic such that $\b V[G]=\b V[G_\beta][H]$. Note that $\bb P'$ is itself (equivalent to) a ${<}\kappa$-support iteration of $\bb D_\kappa$ of length $\kappa^{++}\cdot\kappa^+$.

Since $\b V[G_\beta]\md\ap{2^\kappa=\kappa^{++}}$, we see that $\bb P'$ is $\kappa$-centred by \Cref{preservation centred}. Let $\sst{\dot f_\alpha\mid \alpha\in\kappa^{+}}$ be a set of names such that $\fc_{\bb P'}\ap{\dot f_\alpha\in \prod b}$. Using \Cref{approximation cof eq}, we can find $\st{g_\alpha^\gamma\mid \gamma\in\kappa}\subset\prod b$ in $\b V[G_\beta]$ such that for any $h\in\prod b$ with $h\eqi g_\alpha^\gamma$ for all $\gamma\in\kappa$, we also have $\fc_{\bb P'}\ap{h\eqi \dot f_\alpha}$. Finally, since $\b V[G_\beta]\md\ap{\dneq{b}=\kappa^{++}}$, we see that $|\st{g_\alpha^\gamma\mid \alpha\in\kappa^+\la \gamma\in\kappa}|<\dneq{b}$, and therefore there exists $h\in\prod b$ such that $h\eqi g_\alpha^\gamma$ for all $\alpha\in\kappa^+$ and $\gamma\in\kappa$. Hence $\fc_{\bb P'}\ap{h\eqi \dot f_\alpha}$ for each $\alpha\in\kappa^+$, showing that $\b V[G]\md\ap{\kappa^+=|\sst{\dot f_\alpha\mid \alpha\in\kappa^+}|<\dneq{b}=\kappa^{++}}$.
\end{proof}

\subsection{Separating Cardinals of the Form $\dstar{b,h}$}

It was first shown in \cite{BBFM} that $\dstar{\pow}<\dstar{\id}$ is consistent, by forcing with a generalisation of Sacks forcing due to Kanamori \cite{Kan80}. In \cite{vdV}, we generalised this construction using a perfect-tree forcing where the number of successors of splitting nodes is governed by a function $h$. We used this to prove that there can in fact consistently be $\kappa^+$ many functions $h_\xi$ such that all cardinals $\dstar{h_\xi}$ are mutually different.

It turns out that the same approach could also be used to produce models where cardinals of the form $\dstar{b,h}$ have different values from each other. Since our forcing construction is in practice nearly identical to the construction from \cite{vdV}, we will state most results without proof, and instead refer to the relevant proof in the aforementioned article. 

\begin{rmk}\label{discrepancy}
Note, that we used a different definition of $h$-slalom in this article, where the bound $\phi(\alpha)<h(\alpha)$ is strict, whereas in our previous article $h$-slaloms $\phi$ are defined such that $\phi(\alpha)\leq h(\alpha)$. The effect of this minor change, is that we will often have to use a successor cardinality whenever $h$ is used, or replace $<$ by $\leq$, compared to the analogous result in \cite{vdV}. 
\end{rmk} 

The main difference between our current forcing notion and the forcing notion from \cite{vdV}, is that we work with trees on $\prod b$ instead of trees on $\gbs$. For many of the properties we need, this difference is irrelevant.

\begin{dfn}[Cf. \cite{vdV} Definition 1.1]\label{Sacks-like forcing def}
Let $\bb S_\kappa^{b,h}$ be the forcing notion with conditions $T$ being trees on $\prod b$ such that
\begin{enumerate}[label=(\roman*)]
\item for any $u\in T$ there exists $v\in T$ such that $u\subseteq v$ and $v$ is splitting,
\item if $u\in\Split_\alpha(T)$, then $u$ is an $h(\alpha)$-splitting node in $T$,
\item if $C\subset T$ is a chain of splitting nodes with $|C|<\kappa$, then $\Cup C$ is a splitting node in $T$.
\end{enumerate}
We order $\bb S_\kappa^{b,h}$ by $T\leq S$ iff 
\begin{itemize}
\item $T\subset S$ and
\item for every $u\in T$, if $\suc(u,T)\neq \suc(u,S)$ then $|\suc(u,T)|<|\suc(u,S)|$.\qedhere
\end{itemize} 
\end{dfn}

\begin{dfn}
Define an ordering by $T\leq_\alpha S$ iff $T\leq S$ and $\Split_\alpha(T)=\Split_\alpha(S)$. 
\end{dfn}
It is easy to see that this is a fusion ordering.

\begin{thm}[Cf. \cite{vdV} Lemmas 1.3 and 1.7]\renewcommand{\qed}{\hfill$\square$}
$\frcS{b,h}$ is ${<}\kappa$-closed, closed under fusion, and $\frcS{b,h}$ has the ${<}(2^\kappa)^{+}$-c.c.. Therefore, $\frcS{b,h}$ preserves all cardinals ${\leq}\kappa$ and ${>}2^\kappa$ and does not destroy the inaccessibility of $\kappa$.
\end{thm}

We will use our forcing notion in ground models where $2^\kappa=\kappa^+$, thus the above gives us the preservation of all cardinals with the exception of $\kappa^+$. The preservation of $\kappa^+$ will follow from \Cref{having bF Sacks} below. We need two main ingredients if we want to separate two localisation cardinals $\dstar{b,h}$ and $\dstar{b',h'}$, using a preservation property that we will call the $(b,h)$-Sacks property.

\begin{dfn}\label{Sacks property}
$\bb P$ has the \emph{$(b,h)$-Sacks property} if for every $\bb P$-name $\dot f$ and $p\in \bb P$ such that $p\fc\ap{\dot f\in\prod  b}$ there exists a $\phi\in\Loc_\kappa^{b,h}$ and $q\leq p$ such that $q\fc\ap{\dot f( \xi)\in \phi(\xi)}$ for all $\xi<\kappa$.
\end{dfn}

Hence, if $\Phi\subset\Loc_\kappa^{b,h}$ is a witness for $\dstar{b,h}$ in the ground model, then $\Phi$ will still witness $\dstar{ b,h}$ in extensions by forcing notions with the $(b,h)$-Sacks property. 

Our strategy is the following. Assuming that $b$ is sufficiently larger than $h$, there is some  $F\in\gbs$ such that $\frcS{b,h}$ has the $(b,F)$-Sacks property, and reversely that for any $F$ there exists some $h$ such that $\frcS{b,h}$ does not have the $F$-Sacks property. In other words, assuming $b$ is large enough, for any $F_0$ we may find $h$ and $F_1$ such that $\frcS{b,h}$ does not have the $(b,F_0)$-Sacks property, but does have the $(b,F_1)$-Sacks property. This implies that we can increase $\dstar{b,F_0}$ while the ground model forms a witness for $\dstar{b,F_1}$.

As for what it means that $b$ is sufficiently larger than $h$, it will be sufficient that $(2^h)^+\leq b$. In this section, we will assume without mention that $b$ is always sufficiently large.

We call $T\in\frcS{b,h}$ a \emph{sharp} tree (below $\alpha$) if each $u\in \Split_\xi(T)$ is a sharp $h(\xi)$-splitting node (for all $\xi<\alpha$). The set $(\frcS{b,h})^*=\sst{T\in\frcS{b,h}\mid T\text{ is sharp}}$ embeds densely into $\frcS{b,h}$, thus for each $T\in\frcS{b,h}$ we can choose some $T^*\leq T$ such that $T^*$ is sharp.

\begin{thm}[Cf. \cite{vdV} Theorem 1.8]\label{having bF Sacks}
Let $F:\xi\mapsto (h(\xi)^{|\xi|})^+$, then $\frcS{b,h}$ has the $(b,F)$-Sacks property.
\end{thm}
\begin{proof}
We will follow the proof of Theorem 1.8 in \cite{vdV}, summarised below. 

We create a fusion sequence $\ab{T_\xi\mid \xi\in\kappa}$ and a family of sets $\st{B_\xi\mid \xi\in\kappa}$ such that for every $\xi\in\kappa$ we have $|B_\xi|\leq h(\xi)^{|\xi|}<F(\xi)$ and $T_{\xi+1}\fc\ap{\dot f( \xi)\in  B_\xi}$.

To ensure we can choose $B_\xi$ small enough, we let each $T_\xi$ be sharp below $\xi$. That is, we choose $T_\xi$ such that $|\suc(u,T_\xi)|=h(\alpha)$ for each $u\in\Split_\alpha(T_\xi)$ with $\alpha<\xi$. 

We define a sharp tree $T_\xi^*\leq T_\xi$ (which does not remove any nodes from the $\xi$-th splitting level by the above property, since $T_\xi$ is already sharp below $\xi$) and the set of successors of the $\xi$-th splitting level of $T_\xi^*$:
\begin{align*}
V_\xi=\Cup\sst{\suc(u,T_\xi^*)\mid u\in \Split_\xi(T_\xi^*)}
\end{align*}
Subsequently we find $T_{\xi+1}$ by pruning above the $\xi$-th splitting level such that for each $v\in V_\xi$ the subtree $(T_{\xi+1})_v\fc\ap{\dot f(\xi)= \beta^v_\xi}$ for some ordinal $\beta_\xi^v$. Finally we define $B_\xi=\sst{\beta_\xi^v\mid v\in V_\xi}$ and note that $|B_\xi|\leq |V_\xi|\leq h(\xi)^{|\xi|}<F(\xi)$. The slalom $\phi\in\Loc_\kappa^{b,F}$ witnessing the $(b,F)$-Sacks property is defined as $\phi:\xi\mapsto B_\xi$.
\end{proof}
\begin{crl}[cf. \cite{vdV} Corollary 1.9]
$\frcS{b,h}$ preserves $\kappa^+$. If $\b V\md\ap{2^\kappa=\kappa^+}$, then $\frcS{b,h}$ preserves all cardinalities.
\end{crl}

\begin{thm}[cf. \cite{vdV} Theorem 1.11]\label{stationary sacks property}
If $S=\st{\alpha\in\kappa\mid F(\alpha)\leq h(\alpha)}$ is stationary, then $\frcS{b,h}$ does not have the $(b,F)$-Sacks property.
\end{thm}
\begin{proof}
Let $\dot f$ name the generic $\frcS{b,h}$-real in $(\prod b)^{\b V^{\frcS{b,h}}}$ and $\phi\in\b V$ be a $(b,F)$-slalom. For any $\alpha_0\in\kappa$ can find $T\in \frcS{b,h}$ and $\alpha\geq\alpha_0$ such that $\Split_\alpha(T)=\Lev_\alpha(T)$, then each $u\in \Split_\alpha(T)$ has $h(\alpha)$ successors, thus we can find a successor $v\in\suc(u,T)$ for which $v(\alpha)\notin \phi(\alpha)$, because $|\phi(\alpha)|<F(\alpha)\leq h(\alpha)$. Then $(T)_v\fc\ap{\dot f(\alpha)\notin \phi( \alpha)}$.
\end{proof}

The properties above are preserved under ${\leq}\kappa$-supported products of forcing notions $\frcS{b,h}$, in the sense of the following two lemmas. In comparing our notation, we again have to take care with the discrepancy mentioned in \Cref{discrepancy}.

\begin{lmm}[cf. \cite{vdV} Lemma 2.5]\label{sacks property product}
Let $B\subset A$ be sets of ordinals and $\ab{h_\xi\mid \xi\in A}$ a sequence of functions in with $2^{h_\xi}\in\prod b$ for each $\xi\in A$. Define $\bar{\bb S}=\prod^{\leq\kappa}_{\xi\in A}\bb S^{h_\xi}_\kappa$ and let $G$ be an $\bar{\bb S}$-generic filter. If $F:\alpha\mapsto ((\sup_{\xi\in A\setminus B}h_\xi(\alpha))^{|\alpha|})^+$ has the property that $F\leq^* b$, then for each $f\in(\prod b)^{\b V[G]}$ there exists some  $\phi\in(\Loc_\kappa^{b,F})^{\b V[G\restriction B]}$ such that $f\in^*\phi$.
\end{lmm}
\begin{proof}
This proof is a product version of \Cref{having bF Sacks} and follows the construction made in the proof of Lemma 2.5 of \cite{vdV}. We give a summary of the construction below.

We create a generalised fusion sequence $\ab{(p_\xi,Z_\xi)\mid \xi\in\kappa}$ with $p_\xi\in\bar{\bb S}$ and $\dot D_\xi$ naming sets of ordinals $D_\xi\in\b V[G\restriction B]$ with $|D_\xi|<F(\xi)$ such that $p_{\xi+1}\fc\ap{\dot f(\xi)\in\dot D_\xi}$. 

We assume that $|Z_\xi|\leq |\xi|$ and for each $\beta\in Z_\xi$ we ensure $p_\xi(\beta)$ is sharp below $\xi$. We can define the set of successor nodes of the $\xi$-th splitting level for each $\beta\in Z_\xi$ as:
\begin{align*}
V_\xi^\beta=\Cup\st{\suc(u,p_\xi(\beta))\mid u\in\Split_\xi(p_\xi(\beta))}
\end{align*}
For each function $g:Z_\xi\to \Cup_{\beta\in Z_\xi}V_\xi^\beta$  picking an element $g(\beta)\in V_\xi^\beta$ for each $\beta\in Z_\xi$, we will decide (a name for) a value of $\dot f(\xi)$. Since our sets of ordinals $D_\xi$ are elements of $\b V[G\restriction B]$, we only care about the part of the name decided by the support outside of $B$, that is, in $Z_\xi\setminus B$:
\begin{align*}
\c V_\xi&=\st{g:Z_\xi\to \textstyle\Cup_{\beta\in Z_\xi}V_\xi^\beta\mid g(\beta)\in V_\xi^\beta\text{ for all }\beta\in Z_\xi}\\
\c V_\xi'&=\st{g\restriction (Z_\xi\setminus B)\mid g\in\c V_\xi}
\end{align*}
Because $|\c V_\xi'|\leq F(\xi)$, we can describe a name $\dot D_\xi$ that only depends on the support in $B$ (and thus names an element of $\b V[G\restriction B]$) such that \begin{align*}
p_{\xi+1}\fc\ap{\dot f(\xi)\in \dot D_\xi\text{ and }|\dot D_\xi|\leq F(\xi)}
\end{align*}
We refer to \cite{vdV} for a detailed construction of such names. If $p$ is the limit of the generalised fusion sequence $\ab{(p_\xi,Z_\xi)\mid \xi\in \kappa}$, then we finally define the name $\dot \phi$ such that $p\fc\ap{\dot\phi(\xi)=\dot D_\xi}$. Then $\dot \phi$ is a name for a $(b,F)$-slalom in $\b V[G\restriction B]$ localising the function $f\in\b V[G]$.
\end{proof}

\begin{lmm}[cf. \cite{vdV} Lemma 2.7]\label{nonsacks property product}
Let $B\subset A$ be sets of ordinals and $\ab{h_\xi\mid \xi\in A}$ a sequence of functions in with $2^{h_\xi}\in\prod b$ for each $\xi\in A$. Define $\bar{\bb S}=\prod^{\leq\kappa}_{\xi\in A}\bb S^{b,h_\xi}_\kappa$ and let $G$ be an $\bar{\bb S}$-generic filter. Let $S_\xi=\st{\alpha\in\kappa\mid F(\alpha)\leq h_\xi(\alpha)}$ be stationary for each $\xi\in B$, then $\b V[G]\md\ap{|B|\leq\dstar{b,F}}$.
\end{lmm}
\begin{proof}
We give a summary of the proof of Lemma 2.7 in \cite{vdV}.

Given a family $\Phi=\st{\phi_\xi\mid \xi\in\mu}\subset\Loc_\kappa^{b,F}$ in $\b V[G]$ with $\mu<|B|$, we can find $A'\subset A$ with $|A'|=\kappa^+\cdot \mu$ such that $\Phi\in\b V[G\restriction A']$. Let $\beta\in B\setminus A'$ and let $f=\Cap_{p\in G}p(\beta)$ be the generic $\kappa$-real added by the $\beta$-th term of the product $\bar{\bb S}$.

Using the method of \Cref{stationary sacks property} on $\bb S_\kappa^{b,h_\beta}$ we can show that $f\nins\phi_\xi$ for each $\xi\in\mu$.
\end{proof}
\begin{thm}[cf. \cite{vdV} Theorem 2.9]
If there exists $F\in \prod b$ such that $2^F\leq b$ and $2^{F(\alpha)}\leq F(\beta)$ and $F(\alpha)=F(\alpha)^{|\alpha|}$ for all $\alpha<\beta\in\kappa$, then there exist functions $h_\xi\in\prod b$ for each $\xi<\kappa$ such that for any cardinals $\lambda_\xi>\kappa$ it is consistent that simultaneously $\dstar{b,h_\xi}=\lambda_\xi$ for all $\xi<\kappa$.
\end{thm}
\begin{proof}
Let $\ab{S_\xi\mid \xi\in\kappa}$ be a disjoint family of stationary sets. We will assume without loss of generality that $\ab{\lambda_\xi\mid \xi\in\kappa}$ is an increasing sequence. We define the following functions:
\begin{align*}
h_\xi(\alpha)&=\begin{cases}
(F(\alpha))^+&\text{ if }\alpha\in S_\xi,\\
(2^{F(\alpha)})^+&\text{ otherwise}.
\end{cases}&
H_\xi(\alpha)&=\begin{cases}
F(\alpha)&\text{ if }\alpha\in \Cup_{\eta\in\xi}S_\eta,\\
2^{F(\alpha)}&\text{ otherwise}.
\end{cases}
\end{align*}
By our assumptions on $F$, it is clear that both $h_\xi$ and $H_\xi$ are increasing.
Let $A=\Cup_{\xi\in\kappa}A_\xi$ be the union of disjoint sets of ordinals such that $|A_\xi|=\lambda_\xi$, and for each $\beta\in A_\xi$ define $h_\beta'=H_\xi$. We consider $\bar{\bb S}=\prod^{\leq\kappa}_{\beta\in A}\bb S_\kappa^{b,h'_\beta}$.

Note that $\frcS{b,H_\xi}$ has the $(b,h_\eta)$-Sacks property for all $\eta<\xi$. On the other hand, if $\eta\geq\xi$, then $h_\eta(\alpha)\leq H_\xi(\alpha)$ for all $\alpha\in S_\eta$. It follows from \Cref{sacks property product,nonsacks property product} that:
\begin{align*}
\b V[G]\md\ap{\lambda_\xi=|A_\xi|\leq\dstar{b,h_\xi}\leq (\Loc_\kappa^{b,h_\xi})^{\b V[G\restriction \Cup_{\eta\leq\xi}A_\eta]}=\kappa^+\cdot\textstyle\sup_{\eta\leq\xi}|A_\eta|=\lambda_\xi}.&\qedhere
\end{align*}
\end{proof}
We may even do better than this, if we replace our $\kappa$-sized disjoint family of stationary sets by an almost disjoint family of stationary sets. The existence of such a family  of size $\kappa^+$ follows from $\lozenge_\kappa$. With a preparatory forcing, we can show the existence of  a family of functions $\ab{h_\xi\mid \xi\in\kappa^+}$ and show that there is a forcing extension in which all cardinals $\dstar{b,h_\xi}$ are mutually different for each $\xi\in\kappa^+$. We refer to the third section of \cite{vdV} for a detailed construction.

\subsection{Separating Cardinals of the Form $\dinf{b,h}$}

The forcing notion we consider dn this section can be seen as a $\gbs$-analogue of the forcing notion used by Klausner \& Mej\'ia in \cite{KM21}, and we will generally give a reference to the analogous definitions and lemmas as given in that article. In comparing our results to \cite{KM21}, we once more have to be careful of the discrepancy mentioned in \Cref{discrepancy}.

In \cite{KM21} it is shown that there exists a model in which there are uncountably many mutually different cardinals of the form $\f d^{b,h}_\omega(\nnii)$. In fact, with a little bit more work, Cardona, Klausner \& Mej\'ia showed in \cite{CKM21} that there exists a model in which there are $2^{\aleph_0}$ many mutually different cardinals of each of the forms $\f d^{b,h}_\omega(\nins)$, $\f d^{b,h}_\omega(\nnii)$, $\f b^{b,h}_\omega(\nins)$ and $\f b^{b,h}_\omega(\nnii)$. We will not prove such a strong result, and so far have only managed to give a family of $\kappa$ many functions $b_\alpha,h_\alpha$ such that for any finite $A\subset\kappa$ it is consistent that $\dinf{b_\alpha,h_\alpha}$ are mutually distinct for all $\alpha\in A$.

We will increase a cardinal of the form $\dinf{b,h}$ by generically adding a slalom $\phi\in\Loc_\kappa^{b,h}$ such that $f\ini\phi$ for all $f\in\prod b$ from the ground model. At the same time, we will make sure that the forcing notion we use preserves $\dinf{b',h'}$ for some other parameters $b',h'$. As with our construction from the previous section, we will consider a forcing notion consisting of generalised perfect trees. In this case we use trees on $\Loc_{<\kappa}^{b,h}$ instead of trees on $\prod_{<\kappa}b$, and our generic object will be a $(b,h)$-slalom.

One main difference between our forcing notion and the forcing notion described in \cite{KM21}, is that we will not work with uniform trees. That is, the forcing notion of \cite{KM21} is comparable to Silver forcing, and has partial functions as conditions. Our forcing notion is comparable to Sacks forcing. We made the choice to use a Sacks-like forcing notion to form a better comparison with the Sacks-like forcing notion from the previous section. Both Silver-like and Sacks-like forcing notions will have the same effect on $\dinf{b,h}$, thus this change is not relevant for the results we will prove.

We will first define a norm on subsets of $[b(\alpha)]^{<h(\alpha)}$.

\begin{dfn}\label{norm}
Given $M\subset [b(\alpha)]^{< h(\alpha)}$ let $\norm {M}_{b,\alpha}$ be the least cardinal $\lambda\in\kappa$ for which there exists $y\in[b(\alpha)]^\lambda$ such that for all $x\in M$ we have $y\not\subset x$, i.e., the least size of a subset $y$ of $b(\alpha)$ such that no superset of $y$ is contained in $M$.
\end{dfn}

\begin{dfn}[cf. \cite{KM21} 3.1]\label{dinf forcing definition}
We define a forcing notion $\frcQ{b,h}$ where conditions $T\in\frcQ{b,h}$ are trees $T\subset\Loc_{<\kappa}^{b,h}$ such that
\begin{enumerate}[label=(\roman*)]
\item for any $u\in T$ there exists $v\in T$ such that $u\subset v$ and $v$ is splitting,
\item\label{dinf forcing def: number of splitting} if $u\in\Split_\alpha(T)$, then $\norm{\suc(u,T)}_{b,\dom(u)}\geq|\alpha|$,
\item\label{dinf forcing def: minimal} for any $u\in T$, if $\norm{\suc(u,T)}_{b,\dom(u)}<2$, then $u$ is non-splitting
\item\label{dinf forcing def: club increasing} if $C\subset T$ is a chain of splitting nodes with $|C|<\kappa$, then $\Cup C$ is a splitting node in $T$.
\end{enumerate}
ordered by $T\leq S$ iff
\begin{itemize}
\item $T\subset S$ and
\item for every $u\in T$, if $\suc(u,T)\neq\suc(u,S)$, then $\norm{\suc(u,T)}_{b,\dom(u)}<\norm{\suc(u,S)}_{b,\dom(u)}$.\qedhere
\end{itemize}
\end{dfn}

Here \ref{dinf forcing def: minimal} is necessary to ensure that the intersection of all trees in a generic filter forms a branch. If we allow $\norm{p(\alpha)}_{b,\alpha}=1$ without $p(\alpha)$ being a singleton, then we have no way of decreasing the norm of $p(\alpha)$ any further. 

For $A$ a set of ordinals and $T\in\frcQ{b,h}$, we define a \emph{collapse} of $T\in\frcQ{b,h}$ on $A$ as a condition $T'\leq T$ such that $u$ is non-splitting in $T'$ for all $u\in\Lev_\alpha(T')$ with $\alpha\in A$ and $\suc(u,T')=\suc(u,T)$ for all $u\in \Lev_\alpha(T')$ with $\alpha\in\kappa\setminus A$. It is clear from the definition of the forcing notion that such a collapse exists for any set $A$ that is the complement of a club set.

It is clear that $\frcQ{b,h}$ adds a generic element to $\Loc_\kappa^{b,h}$, in the sense that if $G\subset\frcQ{b,h}$ is a generic filter over $\b V$, then $\phi_G=\Cap G\in\Loc_\kappa^{b,h}$ and $\b V[G]=\b V[\phi_G]$.

\begin{dfn}[cf. \cite{KM21} 3.2]
Define $(\frcQ{b,h})^*\subset\frcQ{b,h}$ as the set of all $T\in\frcQ{b,h}$ such that for each $\alpha\in\kappa$ there exists $s_\alpha(T)\in\kappa$ such that $\Split_\alpha(T)=\Lev_{s_\alpha(T)}(T)$ and $\norm{\suc(u,T)}_{b,s_\alpha(T)}\geq|s_\alpha(T)|$ for all $u\in \Split_\alpha(T)$.
\end{dfn}

\begin{lmm}
$(\frcQ{b,h})^*$ densely embeds into $\frcQ{b,h}$.
\end{lmm}
\begin{proof}
Let $T\in\frcQ{b,h}$. Given $\alpha=\alpha_0\in \kappa$, let
\begin{align*}
\alpha_{n+1}=\sup\st{\dom(u)\smid u\in\Split_{\alpha_n}(T)}.
\end{align*}
Note that $\ab{\alpha_n\mid n\in\omega}$ is increasing. Let $C=\st{\sup_{n\in\omega}\alpha_n\mid\alpha\in\kappa}$, then $C$ is easily seen to be club. Note that if $\xi\in C$ and $u\in\Split_{\xi}(T)$, then $\dom(u)=\xi$. By \ref{dinf forcing def: number of splitting} of \Cref{dinf forcing definition} we see that $|\dom(u)|\leq\norm{\suc(u,T)}_{b,\dom(u)}$ for all $u\in\Split_{\xi}(T)$ with $\xi\in C$.

Finally, let $T'\leq T$ be a collapse of $T$ on $\kappa\setminus C$, then $T'\in(\frcQ{b,h})^*$.
\end{proof}

\begin{lmm}\label{dinf closure}
$\frcQ{b,h}$ is ${<}\kappa$-closed and ${<}(2^\kappa)^+$-c.c..
\end{lmm}
\begin{proof}
Let $\lambda<\kappa$ and $\sab{T_\xi\in\frcQ{b,h}\mid \xi\in\lambda}$ be a descending sequence of conditions, then $\Cap_{\xi\in\lambda}T_\xi$ is a condition below all $T_\xi$. The key observation in proving  ${<}\kappa$-closure is the following claim: if $u\in T=\Cap_{\xi\in\lambda}T_\xi$, then there is $\eta\in\lambda$ such that $\suc(u,T)=\suc(u,T_\eta)$. Checking that $T$ satisfies \Cref{dinf forcing definition} is easy with this claim in mind.

Suppose that $u\in T$ and let $\lambda_\xi=\norm{\suc(u,T_\xi)}_{b,\dom(u)}$, then the ordering on $\frcQ{b,h}$ dictates that $\ab{\lambda_\xi\mid \xi\in\lambda}$ is a descending sequence of cardinals, hence there is $\eta\in\lambda$ such that $\lambda_\xi=\lambda_\eta$ for all $\xi\in[\eta,\lambda)$. But then $\suc(u,T_\xi)=\suc(u,T_\eta)$ for all $\xi\in[\eta,\lambda)$ by the ordering on $\frcQ{b,h}$.

That $\frcQ{b,h}$ is ${<}(2^\kappa)^+$-c.c.\ follows quickly from $|\frcQ{b,h}|=2^\kappa$.
\end{proof}

As a corollary to the above lemma, we see that $\frcQ{b,h}$ preserves all cardinalities ${\leq}\kappa$ and ${>}2^\kappa$. We will prove the preservation of $\kappa^+$ later, after \Cref{early reading}, thus actually we see that all cardinalities are preserved if we assume $2^\kappa=\kappa^+$ in the ground model.

Define an ordering $T\leq_\alpha S$ iff $T\leq S$ and $\Split_\alpha(T)=\Split_\alpha(S)$. It is easy to see that this is a fusion ordering.
\begin{lmm}[cf. \cite{KM21} 3.5(b)]\label{dinf fusion}
$\frcQ{b,h}$ is closed under fusion, that is, if $\ab{T_\alpha\mid \alpha\in\kappa}$ is a sequence in $\frcQ{b,h}$ such that $T_\beta\leq_\alpha T_\alpha$ for any $\beta>\alpha$, then $T=\Cap_{\alpha\in\kappa}T_\alpha\in\frcQ{b,h}$ and $T\leq_\alpha T_\alpha$ for all $\alpha\in\kappa$.
\end{lmm}
\begin{proof}
Note that $\Split_\alpha(T_{\beta})=\Split_\alpha(T)$ for $\alpha<\beta\in\kappa$. Showing that $T\in\frcQ{b,h}$ is easy.
\end{proof}
In the above lemma, we will note that if $T_\alpha\in(\frcQ{b,h})^*$ for all $\alpha$, then $T\in(\frcQ{b,h})^*$ as well.

\begin{lmm}[cf. \cite{KM21} 3.4]\label{lmm3.4}
If $\alpha\in\kappa$, $T\in\frcQ{b,h}$ and $\c D\subset \frcQ{b,h}$ is open dense, then there exists $T'\leq_\alpha T$ such that for any $v\in\Split_{\alpha+1}(T')$ we have $(T')_v\in \c D$. Furthermore, if $T\in(\frcQ{b,h})^*$, then we can also find $T'\in(\frcQ{b,h})^*$ satisfying the above.
\end{lmm}
\begin{proof}
Enumerate $\Split_{\alpha+1}(T)$ as $\ab{v_\xi\mid \xi\in\mu}$ and choose some $T_\xi\leq (T)_{v_\xi}\cap \c D$ for each $\xi\in\mu$, then $T'=\Cup_{\xi\in\mu}T_\xi$ satisfies the above. Note that by construction $\Split_{\alpha+1}(T)\subset T'$. Finally note that if $T\in(\frcQ{b,h})^*$ and we choose each $T_\xi\in(\frcQ{b,h})^*$, then $T'\in(\frcQ{b,h})^*$ holds as well.
\end{proof}

The following lemma shows that $\frcQ{b,h}$ satisfies a generalisation of Baumgartner's Axiom $\s A$ to the context of $\gbs$.

\begin{lmm}[cf. \cite{KM21} 3.5(c)]\label{axiom A}
If $A\subset\frcQ{b,h}$ is an antichain, $T\in\frcQ{b,h}$ and $\alpha\in\kappa$, then there is a condition $T'\leq_\alpha T$ in $\frcQ{b,h}$ such that $T'$ is compatible with less than $\kappa$ elements of $A$.
\end{lmm}
\begin{proof}
Let $\c D$ be the set of $S\in\frcQ{b,h}$ such that $S\leq R$ for some $R\in A$ or such that $S$ is incompatible with all elements of $A$, then $\c D$ is open dense. By \Cref{lmm3.4} there is $T'\in\frcQ{b,h}$ with $T'\leq_\alpha T$ such that $(T')_v\in\c D$ for all $v\in\Split_{\alpha+1}(T')$. Note that $|\Split_{\alpha+1}(T')|<\kappa$ and that $R\in A$ is compatible with $T'$ iff there exists $v\in\Split_{\alpha+1}(T')$ such that $(T')_v\leq R$. It follows that less than $\kappa$ many elements of $A$ are compatible with $T'$.
\end{proof}

We will also prove that $\frcQ{b,h}$ has continuous reading of names. In fact, we will prove a different property that implies continuous reading, which is referred to as \emph{early reading of names} in \cite{KM21}. The preservation of $\kappa^+$ is a straightforward consequence of this property.

\begin{dfn}[cf. \cite{KM21} 3.6]
Let $T\in\frcQ{b,h}$ and $\dot\tau$ be a $\frcQ{b,h}$-name such that $T\fc\ap{\dot\tau:\kappa\to\b V}$. Then we say that $T$ \emph{reads $\dot\tau$ early} if $(T)_u$ decides $\dot\tau\restriction \alpha$ for every $\alpha\in\kappa$ and $u\in\Lev_\alpha( T)$.
\end{dfn}

\begin{lmm}[cf. \cite{KM21} 3.7]\label{early reading}
Let $T\in\frcQ{b,h}$ and $\dot\tau$ a $\frcQ{b,h}$-name such that $T\fc\ap{\dot\tau:\kappa\to\b V}$. Then there exists $T'\leq T$ with $T'\in(\frcQ{b,h})^*$ such that $T'$ reads $\dot \tau$ early.
\end{lmm}
\begin{proof}
Let $\c D_\alpha=\sst{T\in \frcQ{b,h}\mid T\text{ decides }\dot\tau\restriction \alpha}$ and note that $\c D_\alpha$ is open dense for each $\alpha\in\kappa$. We will construct a fusion sequence.

Let $T_0\leq T$ be such that $T_0\in(\frcQ{b,h})^*$. Given $T_\alpha$, use \Cref{lmm3.4} to define $T_{\alpha+1}\in(\frcQ{b,h})^*$ such that $T_{\alpha+1}\leq_\alpha T_\alpha$ and for any $v\in\Split_{\alpha+1}(T_{\alpha+1})$ we have $(T_{\alpha+1})_v\in\c D_{\alpha}$. For limit $\gamma\in\kappa$ we already constructed the descending chain of conditions $\ab{T_\xi\mid \xi\in\gamma}$, so we let $T_\gamma=\Cap_{\xi<\gamma} T_\xi$.

Let $T_\kappa$ be the fusion limit of $\ab{T_\alpha\mid \alpha\in\kappa}$, then by construction we see that if $v\in\Split_{\alpha+1}(T_\kappa)$, then $(T_\kappa)_v$ decides $\dot \tau\restriction \alpha$.

Finally, note that $\st{s_\alpha(T_\kappa)\mid \alpha\in\kappa}$ is club, so $C=\st{\alpha\in\kappa\mid \alpha=s_\alpha(T_\kappa)}$ is club as well. We define $T'$ to be a collapse of $T_\kappa$ on $\kappa\setminus C$. Let $u\in T'$ with $\dom(u)=\alpha\in C$, then for any $\beta<\alpha$, we see that $s_\beta(T_\kappa)<s_\alpha(T_\kappa)=\alpha$, and therefore $(T_\kappa)_{u\restriction (s_\beta(T_\kappa)+1)}$ decides $\dot\tau\restriction\beta$. But this implies that $(T_\kappa)_u$ decides $\dot\tau\restriction\alpha$, and hence also $(T')_u$ decides $\dot\tau\restriction\alpha$. On the other hand, if $u\in T'$ with $\dom(u)=\beta\notin C$, let $\alpha\in C$ be minimal with $\beta<\alpha$, then $s_\beta(T_\kappa)<\alpha$. Let $v\in T'$ be such that $u\subset v$ and $\dom(v)=s_\beta(T_\kappa)+1$, then $(T_\kappa)_v$ decides $\dot \tau\restriction\beta$. But, since $\Cup_{\beta\leq \xi<\alpha}\Lev_\xi(T')$ contains no splitting nodes, we see that $(T')_u=(T')_v$, hence $(T')_u$ decides $\dot\tau\restriction\beta$.
\end{proof}
\begin{crl}
$\frcQ{b,h}$ preserves $\kappa^+$.
\end{crl}
\begin{proof}
Let $\dot\tau$ be a name and $T\in\frcQ{b,h}$ be such that $T\fc\ap{\dot\tau:\kappa\to\kappa^+}$ and $T$ reads $\dot\tau$ early. For each $\alpha\in\kappa$ there is a set $B_\alpha\subset\kappa^+$ with $|B_\alpha|\leq |\Lev_\alpha(T)|<\kappa$ such that $T\fc\ap{\dot\tau(\alpha)\in B_\alpha}$, thus $T\fc\ap{\dot\tau[\kappa]\subset \Cup_{\alpha\in\kappa} B_\alpha}$. Therefore $T\fc\ap{\dot \tau\text{ is not surjective} }$.
\end{proof}

We will now look at the effect that $\frcQ{b,h}$ has on the cardinality of anti-avoidance numbers. We first note that the generic slalom added by $\frcQ{b,h}$ does not antilocalise any $f\in\prod b$ from the ground model. By adding many such slaloms we can increase $\dinf{b,h}$.

\begin{lmm}[cf. \cite{KM21} 3.3]\label{generic increase}
Let $\phi_G\in\Loc_\kappa^{b,h}$ be $\frcQ{b,h}$-generic over $\b V$, let $h'\leq^* b'\in\gbs$ be cofinal increasing cardinal functions and let $S\subset\kappa$ be stationary such that $h(\alpha)\leq h'(\alpha)\leq b'(\alpha)\leq b(\alpha)$ for all $\alpha\in S$. If $\phi'_G\in\Loc_{h'}^{b'}$ satisfies $\phi'_G(\alpha)=\phi_G(\alpha)\cap b'(\alpha)$ for all $\alpha\in S$, then for any $f\in(\prod b')^{\b V}$ we have $f\in^\infty \phi_G'$.
\end{lmm}
Note that this holds specifically for $b=b'$ and $h=h'$, in which case $\phi_G'=\phi_G$.
\begin{proof}
Working in $\b V$, fix some $f\in\prod b'$ and $T\in(\frcQ{b,h})^*$ and $\alpha_0\in\kappa$. Since $C=\st{s_\xi(T)\mid \xi\in\kappa}$ is a club set, we can choose $\alpha>\alpha_0$ such that $\alpha\in S\cap C$. We see that $\norm{\suc(u,T)}_{b,\alpha}> 1$ for any $u\in\Lev_\alpha(t)$, since $u$ is splitting. \Cref{norm} implies that there is some $v\in \suc(u,T)$ such that $\st{f(\alpha)}\subset v(\alpha)$. Then clearly $(T)_v\fc\ap{f(\alpha)\in\dot \phi_G(\alpha)}$, and because $f(\alpha)\in b'(\alpha)$ it follows that $(T)_v\fc\ap{f(\alpha)\in \dot \phi_G'(\alpha)}$. Since $\alpha_0$ was arbitrary, $\b V[G]\md\ap{f\in^\infty\phi_G'}$.
\end{proof}

On the other hand, we can give a condition for parameters $b',h'$ such that the cardinal $\dinf{b',h'}$ is preserved by $\frcQ{b,h}$. We first prove a condition such that $\frcQ{b,h}$ has the $(b',h')$-Sacks property (see \Cref{Sacks property}) and then we use \Cref{aloc loc intraction} to give a similar property to preserve antilocalisation.

\begin{lmm}[cf. \cite{KM21} 3.13]\label{localisation property}
Let $b,h,b',h'\in\gbs$ be increasing cofinal cardinal functions such that for almost all $\alpha\in\kappa$ we have  $\card{\prod_{\xi\leq\alpha}[b(\xi)]^{< h(\xi)}}< h'(\alpha)\leq b'(\alpha)$.
If $\dot f$ is a $\frcQ{b,h}$-name and $T\in\frcQ{b,h}$ is such that $T\fc\ap{\dot f\in\prod b'}$, then there is $\phi\in\Loc_\kappa^{b',h'}$ and $T'\leq T$ such that $T'\fc\ap{\dot f(\xi)\in \phi(\xi)}$ for all $\xi<\kappa$.
\end{lmm}
\begin{proof}
Using \Cref{early reading}, let $T'\in(\frcQ{b,h})^*$ be such that $T'\leq T$ and  $T'$ reads $\dot f$ early, then we can define a unique $y_v\in b'(\alpha)$ for each $v\in\Lev_{\alpha+1}(T')$ with $(T')_v\fc\ap{\dot f(\alpha)= y_v}$. Note that $|\Lev_{\alpha+1}(T')|\leq \card{\prod_{\xi\leq\alpha}[b(\xi)]^{< h(\xi)}}< h'(\alpha)$. Therefore if $\phi:\alpha\mapsto\st{y_v\mid v\in\Lev_{\alpha+1}(T')}$, then $\phi\in\Loc_\kappa^{b',h'}$ and by construction we see that $T'\fc\ap{\dot f\in^*  \phi}$.
\end{proof}

\begin{lmm}[cf. \cite{KM21} 3.15]\label{localisation property corollary}
Let $b,h,b',h',\tilde b,\tilde h\in\gbs$ be increasing cofinal cardinal functions such that for almost all $\alpha\in\kappa$
\begin{itemize}
\item $\card{\prod_{\xi\leq\alpha}[b(\xi)]^{< h(\xi)}}< h'(\alpha)\leq \tilde b(\alpha)^{<\tilde h(\alpha)}\leq b'(\alpha)$, and 
\item $\tilde h(\alpha)\cdot h'(\alpha)<\tilde b(\alpha)$.
\end{itemize}
If $\dot \psi$ is a $\frcQ{b,h}$-name and $T\in\frcQ{b,h}$ is such that $T\fc\ap{\dot \psi\in\Loc_\kappa^{\tilde b,\tilde h}}$, then there is $g\in\prod \tilde b$ and $T'\leq T$ such that $T'\fc\ap{g\nini\dot \psi}$.
\end{lmm}
\begin{proof}
Working in the extension, we define the functions $\rho_-,\rho_+$ witnessing that $\AL_{\tilde b,\tilde h}\preceq\sr L_{b',h'}$ and the injections $\iota_\alpha$ for $\alpha\in\kappa$ as in the proof of \Cref{aloc loc intraction}. Note that if we assume that $\iota_\alpha\in\b V$ for each $\alpha$, then due to the constructive definition of $\rho_-$ and $\rho_+$, it follows that $\rho_-\restriction (\Loc_\kappa^{\tilde b,\tilde h})^{\b V}$ and $\rho_+\restriction (\Loc_\kappa^{b',h'})^{\b V}$ are in the ground model. 

Let $\dot f$ be a name for $\rho_-(\dot \psi)$ and by \Cref{localisation property}, let $\phi\in\Loc_\kappa^{b',h'}$ and $T'\leq T$ be such that $T'\fc\ap{\dot f\in^*\phi}$. Let $g=\phi_+(\phi)\in\prod\tilde b$, then \Cref{aloc loc intraction} shows that $T'\fc\ap{ g\nini \dot \psi}$.
\end{proof}

If we wish to increase $\dinf{b,h}$, we will have to add many new $\kappa$-reals. We will do so using a ${\leq}\kappa$-support product of forcing notions of the form $\frcQ{b,h}$. We could use a ${\leq}\kappa$-support iteration as well, but this has the drawback that we cannot increase $2^\kappa$ past $\kappa^{++}$. We will show that the product behaves nicely, especially that the property of \Cref{localisation property corollary} is preserved under products, and that adding many $\frcQ{b,h}$-generic elements will indeed increase the size of $\dinf{b,h}$.

For the remainder of this section, we will fix some set of ordinals $\c A$ and $b_\zeta,h_\zeta\in\gbs$ for each $\zeta\in\c A$ and we fix the abbreviations $\bb Q_\zeta=\frcQ{b_\zeta,h_\zeta}$  and $\bb Q_\zeta^*=(\frcQ{b_\zeta,h_\zeta})^*$. We define the ${\leq}\kappa$-support products $\bar{\bb Q}=\prod^{\leq\kappa}_{\zeta\in\c A}\bb Q_\zeta$ and $\bar{\bb Q}^*=\prod^{\leq\kappa}_{\zeta\in\c A}\bb Q_\zeta^*$. Since each $\bb Q^*_\zeta$ densely embeds in $\bb Q_\zeta$, it is easy to see that $\bar{\bb Q}^*$ densely embeds in $\bar{\bb Q}$. We will often implicitly assume without mention that all conditions are in $\bar{\bb Q}^*$.

\begin{lmm}
If $\kappa^+=2^\kappa$, then $\bar{\bb Q}$ has the ${<}\kappa^{++}$-cc.
\end{lmm}
\begin{proof}
Since each $\bb Q_\zeta$ has the ${<}\kappa^{++}$-cc, we can use a standard application of the  $\Delta$-system lemma to prove that $\bar{\bb Q}$ has the ${<}\kappa^{++}$-cc as well.
\end{proof}
\begin{lmm}\label{product closure}
$\bar{\bb Q}$ is ${<}\kappa$-closed.
\end{lmm}
\begin{proof}
If $\ab{p_\alpha\mid \alpha<\gamma}$ is a sequence of conditions in $\bar{\bb Q}$ such that $p_{\alpha'}\leq_{\bar{\bb Q}} p_\alpha$ for all $\alpha\leq \alpha'$, then we define $\La_{\alpha<\gamma}p_\alpha:{\c A}\to \Cup_{\zeta\in {\c A}}\bb Q_\zeta$ to be the function $\zeta\mapsto\Cap_{\alpha\in\gamma}p_\alpha(\zeta)$, which is a condition in $\bb Q_\zeta$ by \Cref{dinf closure}. If we assume that $\gamma<\kappa$, then it is easy to see that $|\supp(\La_{\alpha<\gamma} p_\alpha)|\leq \kappa$ as well, so $\bar{\bb Q}$ is ${<}\kappa$-closed.
\end{proof}
We will be using the notation $\La_{\alpha<\gamma}p_\alpha$ from the above proof in the rest of this section as well.

\begin{lmm}[cf. \cite{KM21} 5.5]\label{product fusion}
$\bar{\bb Q}^*$ is closed under generalised fusion, that is, for any generalised fusion sequence $\ab{(p_\alpha,Z_\alpha)\mid \alpha\in\kappa}$ there exists $p\leq_{\alpha,Z_\alpha}p_\alpha$ with $\supp(p)=\Cup_{\alpha\in\kappa}\supp(p_\alpha)$.
\end{lmm}
\begin{proof}
Let $S=\Cup_{\alpha\in\kappa}\supp(p_\alpha)$, then $\Cup_{\alpha\in\kappa}Z_\alpha=S$, so for each $\zeta\in S$ we can fix $\alpha_\zeta\in\kappa$ such that $\zeta\in Z_{\alpha_\zeta}$. Then also $\zeta\in Z_\alpha$ for any $\alpha\geq \alpha_\zeta$, since $Z_\alpha\supset Z_{\alpha_\zeta}$. If $\alpha_\zeta\leq \alpha\leq \beta<\kappa$, then $p_\beta\leq_{\alpha,Z_\alpha}p_\alpha$, and thus by $\zeta\in Z_\alpha$ we see that $p_\beta(\zeta)\leq_\alpha p_\alpha(\zeta)$. Therefore $\ab{p_{\alpha_\zeta+\alpha}(\zeta)\mid \alpha\in\kappa}$ is a fusion sequence in $\bb Q_\zeta^*$ and we can define $p(\zeta)=\Cap_{\alpha\in\kappa}p_{\alpha_\zeta+\alpha}(\zeta)$, then $p(\zeta)\in\bb Q_\zeta^*$ by \Cref{dinf fusion}.
\end{proof}

For $p\in\bar{\bb Q}$ and $\alpha\in\kappa$, we define the set of \emph{possibilities} $\poss{p}{\alpha}$ to be the set of functions $\eta$ with domain $\supp(p)$ such that $\eta(\zeta)\in\Lev_\alpha(p(\zeta))$ for all $\zeta\in\supp(p)$. We similarly define $\possq p{\alpha}=\poss{p}{\alpha+1}$. Remember that $( p(\zeta))_{u}$ is the subtree of $p(\zeta)$ generated by $u$. If $\eta\in\poss p{\alpha}$, we define $\eta\la p$ to be the condition with $(\eta\la p)(\zeta)=(p(\zeta))_{\eta(\zeta)}$ for $\zeta\in\supp(p)$ and $(\eta\la p)(\zeta)=\ft_\zeta$ otherwise. We sometimes abuse this notation also to define $\eta\la q$ for $q\leq p$ with larger support, where we let $(\eta\la q)(\zeta)=q(\zeta)$ for all $\zeta\in\supp(q)\setminus\supp(p)$.

For $p\in\bar{\bb Q}^*$, we define $\Split(p)=\Cup_{\zeta\in\supp(p)}\st{s_\alpha(p(\zeta))\mid \alpha\in\kappa}$ and let $\ab{\bar s_\alpha(p)\mid \alpha\in\kappa}$ be the strictly increasing enumeration of $\Split(p)$. Let $\c Z^p(\alpha)=\st{\zeta\in\supp(p)\mid \exists\xi(s_\xi(p(\zeta))=\alpha)}$. 

We call $p\in\bar{\bb Q}^*$ \emph{modest} if for any $\alpha\in\kappa$ we have $|\poss p\alpha|<\kappa$ and $|\c Z^p(\alpha)|\leq \alpha$, and moreover $|\c Z^p(\bar s_\alpha(p))|=1$ in case $\alpha$ is successor.

\begin{lmm}
The set of modest conditions is dense in $\bar{\bb Q}^*$ (hence in $\bar{\bb Q}$ as well).
\end{lmm}
\begin{proof}
Let $p\in\bar{\bb Q}^*$. We will assume for convenience (and without loss of generality) that $|\supp(p)|=\kappa$. Enumerate $\supp(p)$ as  $\ab{\zeta_\alpha\mid \alpha\in\kappa}$ and let $v_\alpha\in\Lev_{\alpha+1}(p(\zeta_\alpha))$ be arbitrary. We define $q$ as 
\begin{align*}
q(\zeta_\alpha)&=( p(\zeta_\alpha))_{v_\alpha}&&\text{ for all }\alpha\in\kappa&&&\\
q(\zeta)&=\ft_\zeta&&\text{ if }\zeta\notin\supp(p)&&&
\end{align*}
Then $q\leq p$ and $|\poss q\alpha|<\kappa$ and $|\c Z^q(\alpha)|\leq \alpha$ for all $\alpha\in\kappa$. Note that if $r\leq q$ is such that $\supp(r)=\supp(q)$, then $|\poss r\alpha|< \kappa$ and $|\c Z^r(\alpha)|\leq \alpha$ still hold for all $\alpha\in\kappa$.

We will define $r\leq q$ such that $r$ is modest. Note that $C=\st{\bar s_\alpha(q)\mid \alpha\in \kappa\text{ is limit}}$ contains a club set. Let $A_\zeta=\st{\bar s_\alpha(q)\mid \alpha\in \kappa\text{ is successor}\la \zeta\neq\min(\c Z^q(\bar s_\alpha(q)))}$. For any $\zeta\in\supp(q)$, we define $r(\zeta)$ to be a collapse of $q(\zeta)$ on $A_\zeta$, and for any $\zeta\notin\supp(q)$ we let $r(\zeta)=\ft_\zeta$. It is clear from this construction that $\Split(r)=\Split(q)$, and that $\zeta\in \c Z^r(\bar s_\alpha(r))$ implies that $\alpha$ is limit or that $\zeta=\min(\c Z^q(\bar s_\alpha(q)))$. In other words, $r$ is modest.

Finally $r(\zeta)\in \bb Q_{\zeta}^*$ is implied by $p(\zeta)\in\bb Q_{\zeta}^*$, because for each $\alpha\in\kappa$ there is $\beta\geq \alpha$ such that $s_\alpha(r(\zeta))=s_{\beta}(p(\zeta))$, and for any $\alpha\in\Split(r)$ and $\zeta\in \c Z^r(\alpha)$ we have $\suc(u,r(\zeta))=\suc(u,p(\zeta))$ for all $u\in\Lev_\alpha(r)$. Notably, for the norm we see that
\begin{align*}
\norm{\suc(u,r(\zeta))}_{b,\dom(u)}=\norm{\suc(u,p(\zeta))}_{b,\dom(u)}\geq\beta\geq\alpha.&\qedhere
\end{align*}
\end{proof}
The reason we are interested in modest conditions, is that modest conditions behave similarly to conditions of the single forcing notion $\frcQ{b,h}$. By using modest conditions, we get a bound below $\kappa$ to the number of possibilities up to a certain height $\alpha$, which will be crucial in preservation of the Sacks property (\Cref{product localisation property}). 
We will first use modest conditions to generalise \Cref{lmm3.4} to products.

In order to state the lemma, we define one more ordering on $\bar{\bb Q}^*$ as follows: $q\leq_\alpha^* p$ if 
\begin{itemize}
\item $q\leq p$ and 
\item $\Lev_{\bar s_\alpha(p)}(q(\zeta))=\Lev_{\bar s_\alpha(p)}(p(\zeta))$ for all $\zeta\in\supp(p)$ and 
\item $s_0(q(\zeta))>\bar s_\alpha(p)$ for any $\zeta\in\supp(q)\setminus\supp(p)$. 
\end{itemize}

\begin{lmm}[cf. \cite{KM21} 5.4]\label{product lmm3.4}
If $\alpha\in\kappa$, $p\in\bar{\bb Q}^*$ is modest and $\c D\subset\bar{\bb Q}$ is open dense, then there exists $q\in\bar{\bb Q}^*$ with $q\leq_\alpha^* p$ such that for any $\eta\in\possq q{\bar s_\alpha(q)}$ we have $\eta\la q\in\c D$.
\end{lmm}
\begin{proof}
By assumption $p$ is modest, thus if $\ab{\eta_\xi\mid \xi\in\mu}$ enumerates $\possq p{\bar s_\alpha(p)}$, then $\mu<\kappa$. We create a descending sequence of conditions $\ab{p_\xi\mid\xi\in \mu}$. Let $p_0=p$. If $\gamma< \mu$ is limit, let $p_\gamma=\La_{\xi<\gamma}p_\xi$ using  \Cref{product closure}. If $p_\xi$ has been defined, let $p_{\xi+1}'\leq \eta_\xi\la p_\xi$ be such that $p_{\xi+1}'\in \c D$. We define $p_{\xi+1}$ from $p_{\xi+1}'$ in a pointwise manner. If $\zeta\in\supp(p)$, we let 
\begin{align*}
p_{\xi+1}(\zeta)=p_{\xi+1}'(\zeta)\cup\textstyle\Cup\st{(p_\xi(\zeta))_v\mid v\in\Lev_{\bar s_\alpha(p)+1}(p_\xi(\zeta))\text{ and }v\neq\eta_\xi(\zeta)}.
\end{align*}
In plain words, we keep $p_{\xi+1}(\zeta)$ almost equal to $p_\xi(\zeta)$, except that we replace the part of $p_\xi(\zeta)$ that extends $\eta_\xi(\zeta)$ with the tree $p_{\xi+1}'(\zeta)$. On the other hand, if $\zeta\in\supp(p_{\xi+1}')\setminus\supp(p)$, we let $p_{\xi+1}(\zeta)=(p_{\xi+1}'(\zeta))_u$ for some $u\in \Lev_{\bar s_\alpha(p)+1}(p_{\xi+1}'(\zeta))$. Finally if $\zeta\notin\supp(p'_{\xi+1})$ we let $p_{\xi+1}(\zeta)=\ft_\zeta$. 

Note that $p_{\xi+1}'(\zeta)=(p_{\xi+1}'(\zeta))_{\eta_\xi(\zeta)}=(p_{\xi+1}(\zeta))_{\eta_\xi(\zeta)}$ for all $\zeta\in\supp(p)$, and thus $\eta_\xi\la p_{\xi+1}=p_{\xi+1}'$.

Finally define $q=\La_{\xi\in\mu}p_\xi$, then we see that $q\leq_\alpha^* p$ and $\eta\la q\leq \eta\la p_{\xi+1}=p_{\xi+1}'\in\c D$ for each $\eta\in\possq q{\bar s_\alpha(q)}=\possq p{\bar s_\alpha(p)}$.
\end{proof}

We can also generalise the notion of early reading in the most apparent sense.
\begin{dfn}
Let $p\in\bar{\bb Q}$ and $\dot\tau$ be a $\bar{\bb Q}$-name such that $p\fc\ap{\dot\tau:\kappa\to\b V}$, then we say that $p$ reads $\dot\tau$ \emph{early} iff $\eta\la p$ decides $\dot\tau\restriction \alpha$ for every $\alpha\in\kappa$ and $\eta\in\poss p\alpha$.
\end{dfn}

\begin{lmm}[cf. \cite{KM21} 5.6]\label{product early reading}
Let $p\in\bar{\bb Q}$ and $\dot\tau$ a $\bar{\bb Q}$-name such that $p\fc\ap{\dot\tau:\kappa\to\b V}$. Then there exists $q\leq p$ with $q\in\bar{\bb Q}^*$ such that $q$ reads $\dot\tau$ early.
\end{lmm}

\begin{proof}
Let $\c D_\alpha=\st{q\in\bar{\bb Q}\mid q\text{ decides }\dot \tau\restriction\alpha}$, and note that $\c D_\alpha$ is open dense for each $\alpha\in\kappa$. We prove the lemma by constructing $q'\leq p$ such that $\eta\la q'\in\c D_\alpha$ for all $\eta\in\possq {q'}\alpha$. We claim that this is sufficient: define $q$ such that $q(\zeta)$ is a collapse of $q'(\zeta)$ on $\st{\bar s_\alpha(q')\mid \alpha\text{ is successor}}$ for each $\zeta\in\supp(q')$ and $q(\zeta)=\ft_\zeta$ otherwise, then $q$ reads $\dot \tau$ early. 

The condition $q'$ will be the limit of a generalised fusion sequence $\ab{(p_\alpha,Z_\alpha)\mid \alpha\in\kappa}$. Each $p_\alpha$ will be modest and have the following property:
\begin{align*}
(\star_\alpha)\qquad s_\alpha(p_\alpha(\zeta))<s_0(p_\alpha(\zeta')\text{ for all }\zeta\in Z_\alpha\text{ and }\zeta'\in\supp(p_\alpha)\setminus Z_\alpha
\end{align*}
Given $p_\alpha$ satisfying $(\star_\alpha)$, let $\beta_\alpha=\sup\st{s_\alpha(p_\alpha(\zeta))\mid\zeta\in Z_\alpha}$ and suppose $p_{\alpha+1}\leq_{\beta_\alpha}^* p_\alpha$, then it follows from $(\star_\alpha)$ that $p_{\alpha+1}\leq_{\alpha,Z_\alpha} p_\alpha$.

Firstly, we let $p_0\leq p$ be modest such that $p_0\in\c D_{\bar s_0(p_0)}$. This can be easily achieved by letting $p_0^0\leq p$ be modest, finding modest $p_0^{n+1}\leq p_0^n$ such that $p_0^{n+1}\in\c D_{\bar s_0(p_0^n)}$ and letting $p_0=\La_{n\in\omega}p_0^n$. We set $Z_0=\c Z^{p_0}(\bar s_0(p_0))$, then $|Z_0|<\kappa$ by modesty and $p_0$ satisfies $(\star_0)$. We may also assume that $|Z_0|$ is infinite and that $|\supp(p_0)|=\kappa$.

Next, for limit $\gamma$, we have $Z_\gamma=\Cup_{\alpha\in\gamma}Z_\alpha$ and $\hat p_\gamma=\La_{\alpha<\gamma}p_\gamma$. We let $p_\gamma\leq_{\gamma,Z_\gamma}\hat p_\gamma$ be such that it has $(\star_\gamma)$. This is possible, since we may keep $p_\gamma(\zeta)=\hat p_\gamma(\zeta)$ for all $\zeta\in Z_\gamma$ and thus trivially have $p_\gamma\leq_{\gamma,Z_\gamma}\hat p_\gamma$. The construction of the successor step will show that $\beta_\gamma=s_\gamma(p_\gamma(\zeta))=s_\gamma(p_\gamma(\zeta'))$ for any $\zeta,\zeta'\in Z_\gamma$. If $\eta\in\poss {p_\gamma}{\beta_\gamma}$ and $\delta<\beta_\gamma$, then there is $\alpha<\gamma$ such that $\eta\la p_\alpha\in \c D_\delta$, therefore $\eta\la p_\gamma\in\c D_{\beta_\gamma}$.

Finally we construct $p_{\alpha+1}$ from $p_\alpha$. Let $\lambda=|Z_\alpha|$ and enumerate $Z_\alpha$ as $\ab{\zeta_\xi\mid \xi<\lambda}$. We use bookkeeping to fulfil the promise that $\Cup_{\alpha\in\kappa}Z_\alpha=\Cup_{\alpha\in\kappa}\supp(p_\alpha)$, thereby  setting $Z_{\alpha+1}=Z_\alpha\cup\st{\zeta_\lambda}$ for some appropriate $\zeta_{\lambda}\in\supp(p_\alpha)\setminus Z_\alpha$. We construct a descending sequence of conditions $\sab{p_\alpha^\xi\mid \xi\leq \lambda+\alpha+1}$ by recursion over a strictly increasing sequence of ordinals $\sab{\delta_\alpha^\xi\mid \xi\leq \lambda+\alpha+1}$, to obtain the following properties: 
\begin{enumerate}[label=(\roman*)]
\item $p_\alpha^0\leq^*_{\beta_\alpha}p_\alpha$, where $\beta_\alpha=\sup\st{s_\alpha(p_\alpha(\zeta))\mid \zeta\in Z_\alpha}$
\item $p_\alpha^{\xi'}\leq_{\delta_\alpha^{\xi}}^*p_\alpha^{\xi}$ for all $\xi<\xi'\leq\lambda+\alpha+1$
\item $\delta_\alpha^\xi=s_{\alpha+1}(p_\alpha^\xi(\zeta_\xi))$ for all $\xi<\lambda$
\item $\delta_\alpha^\xi<s_{\alpha+1}(p_\alpha^\xi(\zeta_{\xi'}))$ for all $\xi<\xi'<\lambda$
\item $\delta_\alpha^\xi<s_0(p_\alpha^\xi(\zeta))$ for all $\xi<\lambda$ and $\zeta\in \supp(p_\alpha^\xi)\setminus Z_\alpha$
\item $\delta_\alpha^{\lambda+\epsilon}=s_{\epsilon}(p_\alpha^{\lambda+\epsilon}(\zeta_\lambda))$ for all $\epsilon<\alpha$
\item $\delta_\alpha^{\lambda+\epsilon}<s_{\alpha+2}(p_\alpha^{\lambda+\epsilon}(\zeta_\xi))$ for all $\epsilon<\alpha$ and $\xi<\lambda$
\item $\delta_\alpha^{\lambda+\epsilon}<s_0(p_\alpha^{\lambda+\epsilon}(\zeta))$ for all $\epsilon<\alpha$ and $\zeta\in\supp(p_\alpha^{\lambda+\epsilon})\setminus Z_{\alpha+1}$
\item\label{lmm3.4 point} For all $\xi\leq \lambda+\alpha+1$ and any $\eta\in\possq{p_\alpha^\xi}{\delta_\alpha^\xi}$ we have $\eta\la p_\alpha^\xi\in \c D_{\delta_\alpha^\xi}$.
\end{enumerate}

We will set $p_{\alpha+1}=p_\alpha^{\lambda+\alpha+1}$. By construction $p_{\alpha+1}$ satisfies $(\star_{\alpha+1})$ and $p_{\alpha+1}\leq_{\alpha,Z_\alpha}p_\alpha$.

\begin{figure}[t]\centering{\footnotesize
		\begin{tikzpicture}[xscale=1.25,yscale=1]

			\draw[thin, gray, dotted] (0,0) -- (0,13);
			\draw[thin, gray, dotted] (1,0) -- (1,13);
			\draw[thin, gray, dotted] (2,0) -- (2,13);
			\draw[thin, gray, dotted] (3,0) -- (3,13);
			\draw[thin, gray, dotted] (4,0) -- (4,13);
			\draw[thin, gray, dotted] (5,0) -- (5,13);
			\draw[thin, gray, dotted] (6,0) -- (6,13);
			
			\draw[thin, gray] (0,0) -- (0,1.2);
			\draw[thin, gray] (1,0) -- (1,1.2);
			\draw[thin, gray] (2,0) -- (2,1.2);
			\draw[thin, gray] (3,0) -- (3,1.2);
			\draw[thin, gray] (4,0) -- (4,1.2);
			\draw[thin, gray] (5,0) -- (5,1.2);
			\draw[thin, gray] (6,0) -- (6,1.2);
			
			\draw[thin, gray] (0,1.8) -- (0,5.2);
			\draw[thin, gray] (1,1.8) -- (1,5.2);
			\draw[thin, gray] (2,1.8) -- (2,5.2);
			\draw[thin, gray] (3,1.8) -- (3,5.2);
			\draw[thin, gray] (4,1.8) -- (4,5.2);
			\draw[thin, gray] (5,1.8) -- (5,5.2);
			\draw[thin, gray] (6,1.8) -- (6,5.2);
			
			\draw[thin, gray] (0,6.8) -- (0,10.2);
			\draw[thin, gray] (1,6.8) -- (1,10.2);
			\draw[thin, gray] (2,6.8) -- (2,10.2);
			\draw[thin, gray] (3,6.8) -- (3,10.2);
			\draw[thin, gray] (4,6.8) -- (4,10.2);
			\draw[thin, gray] (5,6.8) -- (5,10.2);
			\draw[thin, gray] (6,6.8) -- (6,10.2);
			
			\draw[thin, gray] (0,5.8) -- (0,6.2);
			\draw[thin, gray] (1,5.8) -- (1,6.2);
			\draw[thin, gray] (2,5.8) -- (2,6.2);
			\draw[thin, gray] (3,5.8) -- (3,6.2);
			\draw[thin, gray] (4,5.8) -- (4,6.2);
			\draw[thin, gray] (5,5.8) -- (5,6.2);
			\draw[thin, gray] (6,5.8) -- (6,6.2);
			
			\draw[thin, gray] (0,10.8) -- (0,11.2);
			\draw[thin, gray] (1,10.8) -- (1,11.2);
			\draw[thin, gray] (2,10.8) -- (2,11.2);
			\draw[thin, gray] (3,10.8) -- (3,11.2);
			\draw[thin, gray] (4,10.8) -- (4,11.2);
			\draw[thin, gray] (5,10.8) -- (5,11.2);
			\draw[thin, gray] (6,10.8) -- (6,11.2);
			
			\draw[thin, gray] (0,11.8) -- (0,12);
			\draw[thin, gray] (1,11.8) -- (1,12);
			\draw[thin, gray] (2,11.8) -- (2,12);
			\draw[thin, gray] (3,11.8) -- (3,12);
			\draw[thin, gray] (4,11.8) -- (4,12);
			\draw[thin, gray] (5,11.8) -- (5,12);
			\draw[thin, gray] (6,11.8) -- (6,12);

			\draw[thin, gray] (-.5,1) -- (3.2,1);
			\draw[thin, gray] (-.5,2) -- (3.2,2);
			\draw[thin, gray] (-.5,3) -- (3.2,3);
			\draw[thin, gray] (-.5,4) -- (3.2,4);
			\draw[thin, gray] (-.5,5) -- (3.2,5);
			\draw[thin, gray] (-.5,6) -- (3.2,6);
			\draw[thin, gray] (-.5,7) -- (3.2,7);
			\draw[thin, gray] (-.5,8) -- (3.2,8);
			\draw[thin, gray] (-.5,9) -- (3.2,9);
			\draw[thin, gray] (-.5,10) -- (3.2,10);
			\draw[thin, gray] (-.5,11) -- (3.2,11);
			\draw[thin, gray] (-.5,12) -- (3.2,12);
			
			\draw[thin, gray] (3.8,1) -- (4.2,1);
			\draw[thin, gray] (3.8,2) -- (4.2,2);
			\draw[thin, gray] (3.8,3) -- (4.2,3);
			\draw[thin, gray] (3.8,4) -- (4.2,4);
			\draw[thin, gray] (3.8,5) -- (4.2,5);
			\draw[thin, gray] (3.8,6) -- (4.2,6);
			\draw[thin, gray] (3.8,7) -- (4.2,7);
			\draw[thin, gray] (3.8,8) -- (4.2,8);
			\draw[thin, gray] (3.8,9) -- (4.2,9);
			\draw[thin, gray] (3.8,10) -- (4.2,10);
			\draw[thin, gray] (3.8,11) -- (4.2,11);
			\draw[thin, gray] (3.8,12) -- (4.2,12);
			
			\draw[thin, gray] (4.8,1) -- (5.2,1);
			\draw[thin, gray] (4.8,2) -- (5.2,2);
			\draw[thin, gray] (4.8,3) -- (5.2,3);
			\draw[thin, gray] (4.8,4) -- (5.2,4);
			\draw[thin, gray] (4.8,5) -- (5.2,5);
			\draw[thin, gray] (4.8,6) -- (5.2,6);
			\draw[thin, gray] (4.8,7) -- (5.2,7);
			\draw[thin, gray] (4.8,8) -- (5.2,8);
			\draw[thin, gray] (4.8,9) -- (5.2,9);
			\draw[thin, gray] (4.8,10) -- (5.2,10);
			\draw[thin, gray] (4.8,11) -- (5.2,11);
			\draw[thin, gray] (4.8,12) -- (5.2,12);
			
			\draw[thin, gray] (5.8,1) -- (6.2,1);
			\draw[thin, gray] (5.8,2) -- (6.2,2);
			\draw[thin, gray] (5.8,3) -- (6.2,3);
			\draw[thin, gray] (5.8,4) -- (6.2,4);
			\draw[thin, gray] (5.8,5) -- (6.2,5);
			\draw[thin, gray] (5.8,6) -- (6.2,6);
			\draw[thin, gray] (5.8,7) -- (6.2,7);
			\draw[thin, gray] (5.8,8) -- (6.2,8);
			\draw[thin, gray] (5.8,9) -- (6.2,9);
			\draw[thin, gray] (5.8,10) -- (6.2,10);
			\draw[thin, gray] (5.8,11) -- (6.2,11);
			\draw[thin, gray] (5.8,12) -- (6.2,12);
			
			\draw[thin, gray,dotted] (3,1) -- (6,1);
			\draw[thin, gray,dotted] (3,2) -- (6,2);
			\draw[thin, gray,dotted] (3,3) -- (6,3);
			\draw[thin, gray,dotted] (3,4) -- (6,4);
			\draw[thin, gray,dotted] (3,5) -- (6,5);
			\draw[thin, gray,dotted] (3,6) -- (6,6);
			\draw[thin, gray,dotted] (3,7) -- (6,7);
			\draw[thin, gray,dotted] (3,8) -- (6,8);
			\draw[thin, gray,dotted] (3,9) -- (6,9);
			\draw[thin, gray,dotted] (3,10) -- (6,10);
			\draw[thin, gray,dotted] (3,11) -- (6,11);
			\draw[thin, gray,dotted] (3,12) -- (6,12);
			
			\node[dot,southeast={$s_{\alpha'}(p_{\alpha'}(\zeta_0))$}] () at (0,2) {}; 
			\node[dot,southeast={$s_{\alpha'}(p_{\alpha'}(\zeta_1))$}] () at (1,3) {}; 
			\node[dot,southeast={$s_{\alpha'}(p_{\alpha'}(\zeta_2))$}] () at (2,4) {}; 
			\node[dot,southeast={$s_{\alpha'}(p_{\alpha'}(\zeta_3))$}] () at (3,5) {}; 
			\node[dot,southeast={$s_{\alpha'}(p_{\alpha'}(\zeta_\xi))$}] () at (4,6) {}; 
			\node[dot,southeast={$s_{0}(p_{\alpha'}(\zeta_\lambda))$}] () at (5,7) {}; 
			\node[dot,southeast={$s_{1}(p_{\alpha'}(\zeta_\lambda))$}] () at (5,8) {}; 
			\node[dot,southeast={$s_{2}(p_{\alpha'}(\zeta_\lambda))$}] () at (5,9) {}; 
			\node[dot,southeast={$s_{3}(p_{\alpha'}(\zeta_\lambda))$}] () at (5,10) {}; 
			\node[dot,southeast={$s_{\alpha'}(p_{\alpha'}(\zeta_\lambda))$}] () at (5,11) {}; 
			\node[northeast={$s_{0}(p_{\alpha'}(\zeta))$}] () at (6,12) {}; 
			
			\node[west=$\beta_\alpha$] () at (-.5,1) {}; 
			\node[west=$\delta_\alpha^0$] () at (-.5,2) {}; 
			\node[west=$\delta_\alpha^1$] () at (-.5,3) {}; 
			\node[west=$\delta_\alpha^2$] () at (-.5,4) {}; 
			\node[west=$\delta_\alpha^3$] () at (-.5,5) {}; 
			\node[west=$\delta_\alpha^\xi$] () at (-.5,6) {}; 
			\node[west=$\delta_\alpha^\lambda$] () at (-.5,7) {}; 
			\node[west=$\delta_\alpha^{\lambda+1}$] () at (-.5,8) {}; 
			\node[west=$\delta_\alpha^{\lambda+2}$] () at (-.5,9) {}; 
			\node[west=$\delta_\alpha^{\lambda+3}$] () at (-.5,10) {}; 
			\node[west=$\delta_\alpha^{\lambda+\alpha'}$] () at (-.5,11) {}; 
			
			\node[south=$\zeta_0$] () at (0,0) {};
			\node[south=$\zeta_1$] () at (1,0) {};
			\node[south=$\zeta_2$] () at (2,0) {};
			\node[south=$\zeta_3$] () at (3,0) {};
			\node[south=$\zeta_\xi$] () at (4,0) {};
			\node[south={($\xi<\lambda$)}] () at (4,-0.35) {};
			\node[south=$\zeta_\lambda$] () at (5,0) {};
			\node[south=$\zeta$] () at (6,0) {};
			\node[south={($\zeta\notin Z_{\alpha'}$)}] () at (6,-0.35) {};
			
			\draw[-),ultra thick] (0,0) -- (0,1);
			\draw[-),ultra thick] (1,0) -- (1,1);
			\draw[-),ultra thick] (2,0) -- (2,1);
			\draw[-),ultra thick] (3,0) -- (3,1);
			\draw[-),ultra thick] (4,0) -- (4,1);
			
			\draw[[-,ultra thick] (0,12) -- (0,13);
			\draw[[-,ultra thick] (1,12) -- (1,13);
			\draw[[-,ultra thick] (2,12) -- (2,13);
			\draw[[-,ultra thick] (3,12) -- (3,13);
			\draw[[-,ultra thick] (4,12) -- (4,13);
			\draw[[-,ultra thick] (5,12) -- (5,13);
			\draw[[-,ultra thick] (6,12) -- (6,13);

	\end{tikzpicture}}
	\caption{The structure of $\Split(p_{\alpha+1})$. We write $\alpha+1$ as $\alpha'$ for brevity.}
	\label{figure}
\end{figure}
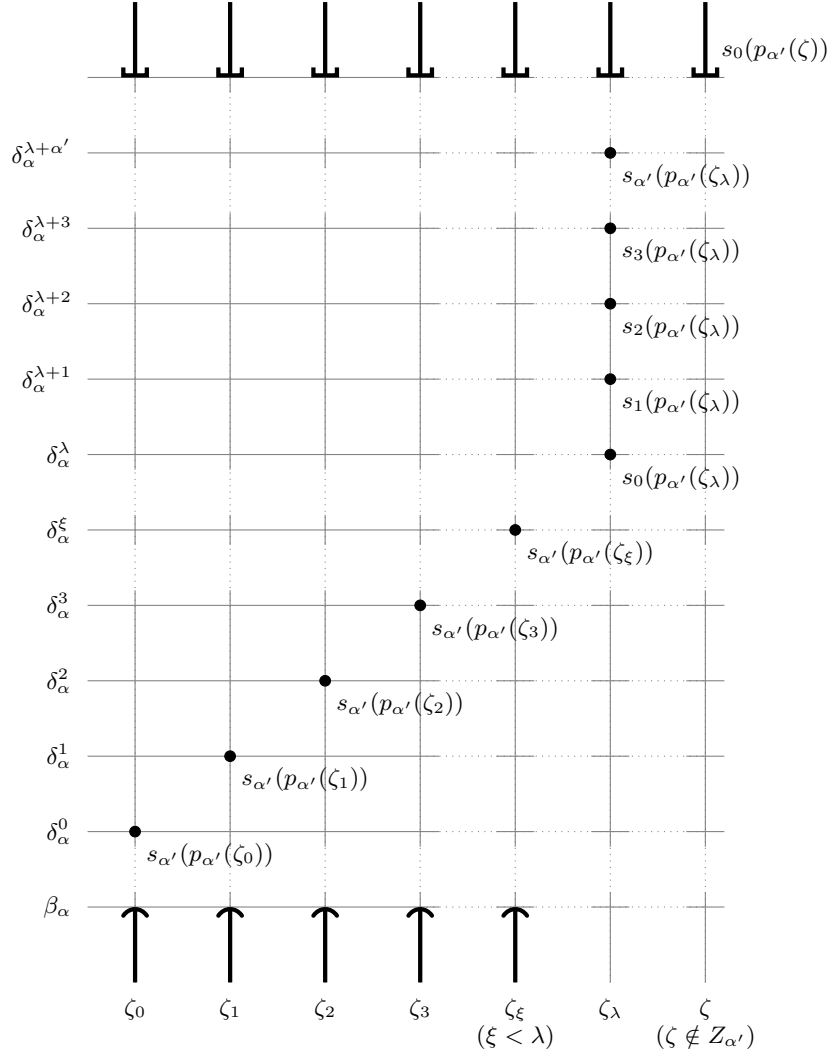

The result of this construction is summarised in \Cref{figure}. Let us clarify this diagram and the recursive construction construction. The initial splitting levels $s_\alpha(p_\alpha(\zeta_\xi))$ with $\xi<\lambda$ occur below $\beta_\alpha$ and are left unmodified during the entire construction. The ordinals $\delta_\alpha^\xi$ give us the height of $s_{\alpha+1}(p_{\alpha+1}(\zeta_\xi))$ with $\xi<\lambda$,  and the ordinals $\delta_\alpha^{\lambda+\epsilon}$ give the height of $s_\epsilon(p_{\alpha+1}(\zeta_\lambda))$. For any other $\zeta\in\supp(p_{\alpha+1})\setminus Z_{\alpha+1}$ the splitting starts strictly above $\beta_{\alpha+1}$, that is, $s_0(p_{\alpha+1}(\zeta))>s_{\alpha+1}(p_{\alpha+1}(\zeta_\lambda))$. At step $\xi$ of the recursive construction, we decide on $\delta_\alpha^\xi$ and hence on the splitting levels up to $\delta_\alpha^\xi$, making sure that the splitting levels we have not considered yet occur at a strictly higher level. We use \Cref{product lmm3.4} to make sure we satisfy \ref{lmm3.4 point} without disturbing the splitting levels up to $\delta_\alpha^\xi$. This automatically gives us modesty as well.
\end{proof}

We are now ready to give the last two lemmas that give us the effect of $\bar{\bb Q}$ on $\dinf{b,h}$.  \Cref{generic increase product} is a generalisation of \Cref{generic increase} and shows that we can increase $\dinf{b,h}$ with $\bb{\bar Q}$, and \Cref{product localisation property} gives us the preservation of the Sacks property, which is a generalisation of \Cref{localisation property}.

\begin{lmm}\label{generic increase product}
Let $\c B\subset \c A$ be sets of ordinals, with $\c B^c=\c A\setminus \c B$, and let $\ab{h_\zeta,b_\zeta\mid \zeta\in\c A}$ be a sequence of cofinal increasing cardinal functions such that $h_\zeta\leq^*b_\zeta$ for all $\zeta\in\c A$. Let $\bar{\bb Q}=\prod_{\zeta\in\c A}^{\leq\kappa}\bb Q_\zeta$, let $G$ be $\bar{\bb Q}$-generic over $\b V$ and $\b V\md\ap{2^\kappa=\kappa^+}$. If $h'\leq^*b'\in\gbs$ are cofinal increasing cardinal functions such that for each $\zeta\in\c B$ there exists a stationary set $S_\zeta\subset \kappa$ such that $h_\zeta(\alpha)\leq h'(\alpha)\leq b'(\alpha)\leq b_\zeta(\alpha)$ for all $\alpha\in S_\zeta$. Then $\b V[G]\md\ap{|\c B|\leq \dinf{b',h'}}$.
\end{lmm}
\begin{proof}
The lemma is trivial if $|\c B|\leq \kappa^+$, thus we assume $\kappa^{++}\leq|\c B|$.

We work in $\b V[G]$. Let $\mu<|\c B|$ and let $\st{f_\xi\mid \xi<\mu}\subset \prod b'$, then we want to describe some $\phi\in\Loc_\kappa^{b',h'}$ such that $f_\xi\in^\infty\phi$ for each $\xi<\mu$. Since $\bar{\bb Q}$ is ${<}\kappa^{++}$-cc, we could find $A_\xi\subset\c A$ with $|A_\xi|\leq \kappa^+$ for each $\xi<\mu$ such that $f_\xi\in \b V[G\restriction A_\xi]$. Since $|\c B|>\mu\cdot \kappa^+$, we may fix some $\beta\in \c B\setminus\Cup_{\xi<\mu}A_\xi$ for the remainder of this proof. Let $\phi_\beta=\Cap_{p\in G}p(\beta)$ be the $\bb Q_\beta$-generic $\kappa$-real added by the $\beta$-th term of the product $\bar{\bb Q}$, and let $\phi'\in\Loc_\kappa^{b',h'}$ be such that $\phi'(\alpha)= \phi_\beta(\alpha)\cap b'(\alpha)$ for each $\alpha\in S_\beta$.

Continuing the proof in the ground model, let $\dot \phi'$ be a $\bar {\bb Q}$-name for $\phi'$ and $\dot f_\xi$ be a $\bar{\bb Q}$-name for $f_\xi$, let $p\in\bar{\bb Q}^*$ and $\alpha_0<\kappa$. We want to find some $\alpha\geq\alpha_0$ and $q\leq p$ such that $q\fc\ap{\dot f_\xi(\alpha)\in \dot \phi'(\alpha)}$.

We now reason as in \Cref{generic increase}. Let $C=\st{s_\xi(p(\beta))\mid \xi\in\kappa}$, then $C$ is club. Therefore, there exists $\alpha\in S_\beta\cap C$ with $\alpha\geq \alpha_0$. Choose some $p_0\leq p$ such that $p_0(\beta)=p(\beta)$ and such that there is a $\gamma\in b'(\alpha)$ for which $p_0\fc\ap{\dot f_\xi(\alpha)=\gamma}$. This is possible, since $f_\xi\in\b V[G\restriction A_\xi]$ and $\beta\notin A_\xi$, therefore we could find $p_0'\in\bar{\bb Q}\restriction A_\xi$ with $p_0'\leq p\restriction A_\xi$ and $\gamma$ with the aforementioned property, and then let $p_0(\eta)=p'_0(\eta)$ if $\eta\in A_\xi$ and $p_0(\eta)=p(\eta)$ otherwise.

Note that $\alpha\in C$ implies that $\norm{\suc(u,p_0(\beta))}_{b_\beta,\alpha}=\norm{\suc(u,p(\beta))}_{b_\beta,\alpha}>1$ for all $u\in\Lev_\alpha(p_0(\beta))$, and that $\alpha\in S_\beta$ implies that $\gamma\in b'(\alpha)\subset b_\beta(\alpha)$. Consequently, there exists $v\in\suc(u,p_0(\beta))$ with $\gamma\in v(\alpha)$. Note that $v(\alpha)\in[b_\beta(\alpha)]^{<h_\beta(\alpha)}$, and $h_\beta(\alpha)\leq h'(\alpha)$ in virtue of $\alpha\in S_\beta$. It follows that $(p_0(\beta))_v\fc\ap{\dot\phi_\beta(\alpha)=v(\alpha)}$, where $\dot\phi_\beta$ names the generic $(b_\beta,h_\beta)$-slalom. Define $q\leq p_0$ by $q(\zeta)=p_0(\zeta)$ if $\zeta\in \c A\setminus \st\beta$ and $q(\beta)=(p_0(\beta))_v$, then we see that $q\fc\ap{\gamma\in v(\alpha)=\dot \phi_\beta(\alpha)\la \gamma\in b'(\alpha)}$, and thus $q\fc\ap{\dot f_\xi(\alpha)=\gamma\in\dot\phi_\beta(\alpha)\cap b'(\alpha)= \dot \phi'(\alpha)}$ 

Since $\alpha_0$ was arbitrary, it follows that $\b V[G]\md\ap{f_\xi\in^\infty \phi'}$ for each $\xi<\mu$.
\end{proof}

\begin{lmm}\label{product localisation property}
Let $\c B\subset\c A$ be sets of ordinals, with $\c B^c=\c A\setminus \c B$, and let $\ab{h_\zeta,b_\zeta\mid \zeta\in \c A}$ be a sequence of cofinal increasing cardinal functions such that $h_\zeta\leq^*b_\zeta$ for all $\zeta\in\c A$. Let $\bar{\bb Q}=\prod^{\leq\kappa}_{\zeta\in \c A}\bb Q_\zeta$ and let $G$ be $\bar{\bb Q}$-generic over $\b V\md\ap{2^\kappa=\kappa^+}$. If $h'\leq^*b'\in\gbs$ are cofinal increasing cardinal functions such that $\rb{\sup_{\zeta\in\c B^c}\card{\Loc_{\leq\alpha}^{b_\zeta,h_\zeta}}}^{|\alpha|}< h'(\alpha)$ for almost all $\alpha\in\kappa$ , then for each $f\in(\prod b')^{\b V[G]}$ there exists $\phi\in(\Loc_\kappa^{b',h'})^{\b V[G\restriction \c B]}$ and such that $f\in^* \phi$.
\end{lmm}
\begin{proof}
The proof is essentially that of \Cref{localisation property}. Let us assume that $f\in\b V[G]\setminus \b V[G\restriction \c B]$, since the lemma would be trivially true otherwise. We also assume $\alpha$ is large enough to satisfy the condition on the size of $h'(\alpha)$.

Let $\dot f$ be a $\bar{\bb Q}$-name for $f$ and let $q\in\bar{\bb Q}^*$ read $\dot f$ early and assume without loss of generality that $q$ is modest. For the sake of brevity, let us write $\c Z_{\dwa\alpha}^q=\Cup_{\xi\leq \alpha}\c Z^q(\xi)$. Note that: 
\begin{align*}
|\possq q\alpha|&=\textstyle\prod_{\zeta\in\c Z_{\dwa\alpha}^q}\card{\Lev_{\alpha+1}(q(\zeta))}\\
&\leq\rb{\textstyle\sup_{\zeta\in\c Z_{\dwa\alpha}^q}(|\Lev_{\alpha+1}(q(\zeta))|)}^{\card{\c Z_{\dwa\alpha}^q}}\\
&\leq\rb{\textstyle\sup_{\zeta\in\c Z_{\dwa\alpha}^q}(|\Lev_{\alpha+1}(q(\zeta))|)}^{\card{\alpha}}
\end{align*}
Here the last inequality follows from $q$ being modest, which implies that $|\c Z_{\dwa\alpha}^q|\leq\alpha$. Since we wish to construct a name $\dot \phi$ for some $\phi\in\b V[G\restriction \c B]$, we see that $\phi$ is completely decided by the part of the support in $\c B$, and thus we may use the part of the support in $\c B^c$ freely to restrict the range of possible values for $f\in\b V[G]$ in order to make sure that $q\fc\ap{\dot f\in^*\phi}$, as we did in \Cref{localisation property}. If we restrict our attention to $\c B^c$, then we see that
\begin{align*}
\rb{\textstyle\sup_{\zeta\in \c B^c\cap\c Z_{\dwa\alpha}^q}(|\Lev_{\alpha+1}(q(\zeta))|)}^{\card{\alpha}}&\leq 
\rb{\textstyle\sup_{\zeta\in \c B^c}(|\Lev_{\alpha+1}(q(\zeta))|)}^{\card{\alpha}}\\
&\leq 
\rb{\textstyle\sup_{\zeta\in \c B^c}\card{\Loc_{\leq\alpha}^{b_\zeta,h_\zeta}}}^{\card{\alpha}}< h'(\alpha)
\end{align*}
This set of possibilities is small enough to define $\phi\in(\Loc_\kappa^{b',h'})^{\b V[G\restriction \c B]}$. To be precise, we construct a sequence of names $\sab{\dot B_\alpha\mid \alpha\in \kappa}$ for sets $B_\alpha\in\b V[G\restriction\c B]$ such that $\b V[G\restriction\c B]\md\ap{|B_\alpha|< h'(\alpha)}$ and such that $q\fc\ap{\dot f(\alpha)\in \dot B_\alpha}$.

Since $q$ reads $\dot f$ early, if $\eta\in\possq q\alpha$, then let $\gamma_\eta\in\kappa$ be such that $\eta\la q\fc\ap{\dot f(\alpha)=\gamma_\eta}$. 

Given $\eta_{\c B}\in\possq{q\restriction \c B}\alpha$ let $Y(\eta_\c B)=\st{\gamma_\eta\mid \eta\in \possq q\alpha\text{ and }\eta\restriction \c B=\eta_\c B}$. Now we define the name $\dot B_\alpha=\st{\ab{ Y(\eta_\c B),\eta_\c B\la q}\mid \eta_\c B\in\possq {q\restriction \c B}\alpha}$. Since the elements of $Y(\eta_\c B)$ are only dependent on the domain $\c B^c$, it follows by the arithmetic from above that $q\fc\ap{|\dot B_\alpha|< h'(\alpha)}$, and thus if $\dot \phi$ names a slalom such that $q\fc\ap{\dot \phi(\alpha)=\dot B_\alpha}$, then $q\fc\ap{\dot f\ins\dot \phi\in \Loc_\kappa^{b',h'}}$.
\end{proof}

\begin{crl}\label{product localisation corollary}
Let $\c B\subset\c A$ be sets of ordinals, with $\c B^c=\c A\setminus \c B$, and let $\ab{h_\zeta,b_\zeta\mid \zeta\in \c A}$ be a sequence of cofinal increasing cardinal functions such that $h_\zeta\leq^*b_\zeta$ for all $\zeta\in\c A$. Let $\bar{\bb Q}=\prod^{\leq\kappa}_{\zeta\in \c A}\bb Q_\zeta$ and let $G$ be $\bar{\bb Q}$-generic over $\b V$. If $h',b',\tilde h,\tilde b\in\gbs$ are cofinal increasing cardinal functions such that $\rb{\sup_{\zeta\in\c B^c}\card{\Loc_{\leq\alpha}^{b_\zeta,h_\zeta}}}^{|\alpha|}< h'(\alpha)\leq\tilde b(\alpha)^{\tilde h(\alpha)}\leq b'(\alpha)$ and $h'(\alpha)\cdot \tilde h(\alpha)<\tilde b(\alpha)$ for almost all $\alpha\in\kappa$, then if $\psi\in(\Loc_\kappa^{\tilde b,\tilde h})^{\b V[G]}$ there exists $g\in(\prod \tilde b)^{\b V[G\restriction \c B]}$ such that $g\nini \psi$.
\end{crl}
\begin{proof}
This follows from \Cref{product localisation property}, similar to how \Cref{localisation property corollary} follows from \Cref{localisation property}.
\end{proof}

\begin{thm}
There exists a family of parameters $\ab{h_\gamma,b_\gamma\mid \gamma\in\kappa}$ such that for any finite $\gamma_n<\dots<\gamma_0<\kappa$ and cardinals $\kappa^+=\lambda_0<\dots<\lambda_n$ with $\cf(\lambda_i)>\kappa$ for each $i\in[0,n]$, there exists a forcing extension where $\dinf{b_{\gamma_i},h_{\gamma_i}}=\lambda_i$ for each $i\in[0,n]$.
\end{thm}
\begin{proof}
Let $h_0\in\gbs$ be a cofinal increasing cardinal functions such that $h_0(\alpha)\geq |\alpha|$ for all $\alpha\in\kappa$. For each $\gamma\in\kappa$ define $h_\gamma,b_\gamma$ recursively as follows: let $h_{\gamma+1}(\alpha)=b_\gamma(\alpha)^{h_\gamma(\alpha)}$ and if $\gamma$ is limit, let $h_\gamma(\alpha)=\sup_{\xi\in\gamma}h_\xi(\alpha)$, and finally let $b_{\gamma}(\alpha)=2^{h_{\gamma}(\alpha)}$ for all $\gamma\in\kappa$. 

Now let $\kappa^+=\lambda_0<\lambda_1<\dots<\lambda_n$ be a finite sequence of regular cardinals and let $\gamma_0>\dots>\gamma_n$ be a decreasing sequence of ordinals. Let $A_1,\dots,A_n$ be disjoint sets of ordinals such that $|A_i|=\lambda_i$ for each $i\in[1,n]$, and let $\c A=\Cup_{i\in[1,n]}A_i$. Now for each $\zeta\in\c A$ let $h_\zeta'=h_{\gamma_i}$ and $b_\zeta'=b_{\gamma_i}$ iff $\zeta\in A_i$, and let $\bar{\bb Q}=\prod^{\leq\kappa}_{\zeta\in\c A}\bb Q_{h_\zeta',b_\zeta'}$. Let $G$ be $\bar{\bb Q}$-generic over $\b V$ and assume that $\b V\md\ap{2^\kappa=\kappa^+}$. 

Fix $1\leq i\leq n$. From \Cref{product localisation property} it follows that $\b V[G\restriction A_i]\md\ap{\lambda_i\leq \dinf{b_{\gamma_i},h_{\gamma_i}}}$. If we let $\c B=\Cup_{i\in[1,i]}A_i$, then also $|\c B|=\lambda_i$ and consequently $\b V[G\restriction \c B]\md\ap{2^\kappa=\lambda_i}$. If $i<j\leq n$, then $h_{\gamma_j},b_{\gamma_j}$ are much smaller than $h_{\gamma_i},b_{\gamma_i}$, and thus by \Cref{product localisation corollary} we see that $(\prod b_{\gamma_i})^{\b V[G\restriction \c B]}$ forms a witness to prove that $\b V[G]\md\ap{\dinf{b_{\gamma_i},h_{\gamma_i}}\leq \lambda_i}$.
\end{proof}

\section{Discussion}\label{section: discussion}

Throughout this article we have mentioned several open questions. In this section we will collect these open questions.

Firstly, we have several open questions regarding the parameters for which the cardinal characteristics on the $\f d$-side become trivial, in the sense that they will be equal to $2^\kappa$. We do have full characterisations of the trivial parameters for the cardinals on the $\f b$-side with \Cref{bounds on binf,bounds on bleq,bounds on bstar}, and by duality it is a reasonable conjecture that the cardinals on the $\f d$-side will be trivial under the same conditions.
\begin{qst}
Assume case (ii) of \Cref{bounds on bleq}. Can we prove that $\dleq{b}=2^\kappa$?
\end{qst}
\begin{qst}
Assume case (ii) of \Cref{bounds on bstar} and that there does not exist an almost disjoint family $\c A\subset\prod b$ of functions with $|\c A|=2^\kappa$. Can we prove that $\dstar{b,h}=2^\kappa$? Specifically, if $b$ is continuous on a club set, is $\dstar{b,h}<2^\kappa$ consistent?
\end{qst}
\begin{qst}
Assume case (i) or (ii) of \Cref{bounds on binf}. Can we prove that $\dinf{b,h}=2^\kappa$?
\end{qst}
\begin{qst}
Assume case (i) or (ii) of \Cref{bounds on bneq}. Can we prove that $\dneq{b}=2^\kappa$?
\end{qst}

Next there are several open questions regarding consistency of strict relations. As we saw, the $\kappa$-Hechler models show that it is consistent that $\add(\c M_\kappa)$ is strictly larger than $\f{sup}_\kappa(\neqi)$, and that $\cof(\c M_\kappa)$ can be strictly smaller than $\f{inf}_\kappa(\neqi)$. 

It is, however, possible (relative to some large cardinal assumptions) to add $b$-dominating $\kappa$-reals, that is, functions $f\in\prod b$ such that $f\geq^* g$ for all $g\in\prod b$ from the ground model, without adding dominating $\kappa$-reals. In fact, the bounded $\kappa$-Hechler forcing $\bb D_\kappa^b(\kappa)$ from \Cref{Hechler general definition} will not add any new $\kappa$-real $f\in\gbs$ such that $f$ dominates all ground model functions in $\gbs$ when $\kappa$ is weakly compact. See for example \cite{ShelahDom}, where Shelah constructs a model where $\cov(\c M_\kappa)<\dleq{}$, under the assumption of a supercompact cardinal.

This makes the following question natural:
\begin{qst}
Is it consistent that $\dleq{b}<\dleq{}$ for all $b\in\gbs$? Is it consistent that $\bleq{}<\bleq{b}$ for all $b\in\gbs$?
\end{qst}

We can ask the same question for the relation between bounded and unbounded localisation cardinals. That is, for a suitable function $h\in\gbs$, can we separate cardinals of the form $\dstar{b,h}$ from  the cardinal $\dstar{h}$ as defined on the unbounded generalised Baire space. Shelah's model from \cite{ShelahDom} may be a promising candidate to consider for the $\f d$-side.

\begin{qst}
Is it consistent that $\f{sup}_\kappa^h(\in^*)<\dstar{h}$?
Is it consistent that $\bstar{h}<\f{inf}_\kappa^h(\in^*)$?
\end{qst}

Finally, we have shown that it is consistent that there is a family of $\kappa^+$ many parameters $h_\xi,b_\xi$, such that the corresponding cardinals of the form $\dstar{b_\xi,h_\xi}$ are mutually distinct, and that there exists a family of $\kappa$ many parameters $h_\xi,b_\xi$ such that for any finite set of cardinals $\dinf{b_\xi,h_\xi}$ there is a model in which they are mutually distinct. Classically we know that there exists a model in which there are $2^{\aleph_0}$ many distinct cardinals of each of the forms $\f d_\omega^{b,h}(\in^*)$, $\f b_\omega^{b,h}(\in^*)$, $\f d_\omega^{b,h}(\nnii)$ and $\f b_\omega^{b,h}(\nnii)$, as described in \cite{CKM21}.

\begin{qst}
Can there consistently be $2^\kappa$ many different cardinalities of the form $\dstar{b,h}$?
\end{qst}
\begin{qst}
Can there consistently be infinitely many different cardinalities of the form $\dinf{b,h}$?
\end{qst}

Moreover, the type of forcing notions we used to separate (anti)localisation cardinals, was a ${\leq}\kappa$-support product of perfect tree forcing notions. This method cannot be used to obtain dual models, such as we did with $\kappa$-Hechler forcing, and hence we do not have any separation of cardinals of the form $\bstar{b,h}$ and $\binf{b,h}$.

Separating such cardinals could be done in the classical context with a Laver-type forcing notion. However, in the generalised context, Laver-type forcing notions behave quite differently. For example, most reasonable generalisations of Laver forcing add $\kappa$-Cohen reals, see \cite{KKLW22}. Another method would be to use an appropriate version of localisation forcing, but we would need a preservation theorem to show that $\bstar{b,h}$ remains small for certain parameters. The preservation of properties of forcing notions under finite or countable support forcing notions is usually unproblematic in the classical case, if proper forcing notions are considered, but generalisations of such preservation theorems are not known for ${<}\kappa$- or ${\leq}\kappa$-support iterations of forcing notions that add $\kappa$-reals. 

In conclusion, we currently do not know how to separate even two cardinals of the form $\bstar{b,h}$. For the same reason, this question is also open for cardinals of the form $\bstar{h}$, see for example Question 71 and the discussion of section 5.2 in \cite{BBFM}.

\begin{qst}
Do there exist nontrivial parameters $b,h,b',h'\in\gbs$ such that $\bstar{b,h}<\bstar{b',h'}$ is consistent? Or such that $\binf{b,h}<\binf{b',h'}$ is consistent?
\end{qst}

\nocite{*}

\bibliographystyle{alpha}
\bibliography{bib}

\end{document}